\documentclass[12pt,leqno]{amsart}
\usepackage{latexsym,amsmath,amssymb,mathrsfs}
\usepackage{hyperref}
\usepackage{esint}
\usepackage{mathtools}
\DeclarePairedDelimiterX{\norm}[1]{\lVert}{\rVert}{#1}
\usepackage{enumitem}
\usepackage{arydshln}

\usepackage{tikz}

\title[Lipschitz mappings into metric spaces]{Differentiability of Lipschitz mappings into metric spaces: area and co-area formulas}

\author[B. Esmayli]{Behnam Esmayli}
\address{Behnam Esmayli: Department of Mathematics, Linköping University, SE-581 83 Linköping,
Sweden}
\email{\tt behnam.esmayli@liu.se}

\author[P. Goldstein]{Pawe\l{}  Goldstein}
\address{Pawe\l{} Goldstein, Institute of Mathematics, Faculty of Mathematics, Informatics and Mechanics, University of Warsaw, Banacha 2, 02-097 Warsaw, Poland}
\email{\tt P.Goldstein@mimuw.edu.pl}
\thanks{P.G.\ was supported by NCN grant no 2019/35/B/ST1/02030}

\author[P. Haj\l{}asz]{Piotr Haj\l{}asz}
\address{Piotr Haj{\l}asz: Department of Mathematics, University of Pittsburgh, 301
Thackeray Hall, Pittsburgh, PA 15260, USA}
\email{\tt hajlasz@pitt.edu}
\thanks{P.H. was supported by NSF grant  DMS-2452426.}

plus 6pt minus 12pt
\newtheorem{theorem}{Theorem}
\newtheorem{lemma}[theorem]{Lemma}
\newtheorem{corollary}[theorem]{Corollary}
\newtheorem{proposition}[theorem]{Proposition}
\newtheorem*{proposition*}{Proposition}

\theoremstyle{definition}
\newtheorem{remark}[theorem]{Remark}
\newtheorem{definition}[theorem]{Definition}
\newtheorem{example}[theorem]{Example}

\newcommand{\barint}{
\rule[.036in]{.12in}{.009in}\kern-.16in \displaystyle\int }

\newcommand{\barcal}{\mbox{$ \rule[.036in]{.11in}{.007in}\kern-.128in\int $}}

\newcommand{\bbbn}{\mathbb N}
\newcommand{\bbbz}{\mathbb Z}
\newcommand{\bbbr}{\mathbb R}
\newcommand{\bbbb}{\mathbb B}
\newcommand{\Bbbb}{\bar{\mathbb B}}
\newcommand{\qp}{{\mathbb Q}_+}

\newcommand{\Sph}{\mathbb S}

\newcommand{\eps}{\varepsilon}

\def\diam{\operatorname{diam}}
\def\card{\operatorname{card}}
\def\dist{\operatorname{dist}}

\def\H{{\mathcal H}}
\def\LL{{\mathcal L}}
\def\rank{{\rm rank\,}}
\def\lip{{\rm Lip\,}}

\def\md{\operatorname{md}}
\def\apmd{\operatorname{md_{\rm a}}}

\newcommand\mycom[2]{\genfrac{}{}{0pt}{}{#1}{#2}}

\def\mvint_#1{\mathchoice
          {\mathop{\vrule width 6pt height 3 pt depth -2.5pt
                  \kern -8pt \intop}\nolimits_{\kern -3pt #1}}%
          {\mathop{\vrule width 5pt height 3 pt depth -2.6pt
                  \kern -6pt \intop}\nolimits_{#1}}%
          {\mathop{\vrule width 5pt height 3 pt depth -2.6pt
                  \kern -6pt \intop}\nolimits_{#1}}%
          {\mathop{\vrule width 5pt height 3 pt depth -2.6pt
                  \kern -6pt \intop}\nolimits_{#1}}}

\numberwithin{theorem}{section} \numberwithin{equation}{section}

\begin{document}
\sloppy



\begin{abstract}
We give a self-contained exposition of metric differentiability for Lipschitz mappings from sets in Euclidean spaces into arbitrary metric spaces, together with the corresponding area and co-area formulas. The novelty is a new metric implicit function theorem for mappings into metric spaces, which is then used to give a direct and geometric proof of the co-area formula for Lipschitz mappings into metric spaces.
\end{abstract}

\maketitle

\section{Introduction}
The goal of this paper is to provide a self-contained introduction to the metric differentiability of Lipschitz mappings into metric spaces, an area initiated by the celebrated paper of Kirchheim~\cite{kir}. Most of the material presented here is known, although many of the proofs and results are new. However, the main novelty of the paper is a proof of the co-area formula for Lipschitz mappings into metric spaces (Theorem~\ref{INT5}), which is based on a new version of the implicit function theorem (Theorem~\ref{INT2}). We believe that this proof is simpler and more geometric than those available in the existing literature; see \cite{kar,reichel}.

We have chosen to present all the necessary details and to organize the material so that the paper can serve as lecture notes for a short graduate-level course. In particular, the exposition is aimed at graduate students who may benefit from a detailed and self-contained treatment of the subject. We put a lot of effort to make the exposition well motivated and easy to read. The paper has been reviewed repeatedly by GPT-6 Astra Max, and we believe it is free of typographical and mathematical errors.

The purpose of the Introduction is to give a brief overview of the paper’s content. For the sake of simplicity, some statements are presented here without full generality.

The Lebesgue measure $\LL^n$ in $\bbbr^n$ coincides with the Hausdorff measure $\H^n$ (Corollary~\ref{SJT2}). Accordingly, and for consistency of notation, we will use the symbol $\H^n$ to denote the Lebesgue measure throughout the paper.

While, according to the Rademacher theorem, Lipschitz mappings $f:\bbbr^n\supset\Omega\to \bbbr^m$ are differentiable a.e., Kirchheim~\cite{kir} introduced the notion of {\em metric derivative} of a map $f:\bbbr^n\supset\Omega\to X$ into any metric space and proved that Lipschitz mappings $f:\bbbr^n\supset\Omega\to X$ are metrically differentiable a.e. Here and in what follows $\Omega$ will denote an open set.

Let us first motivate the definition of the metric derivative.

If $f\colon \bbbr^n\to\bbbr^m$ is (Fr\'echet) differentiable at $x\in\bbbr^n$, then
$$
\lim_{y\to x}\frac{f(y)-f(x)-Df(x)(y-x)}{|y-x|}=0,
$$
and it follows from the triangle inequality that
$$
\lim_{y\to x} \frac{|f(y)-f(x)|-\sigma_x(y-x)}{|y-x|}=0,
\quad
\text{where}
\quad
\sigma_x(v)=|Df(x)v|.
$$
Since
\[
\begin{split}
& \frac{\big||f(z)-f(y)|-|Df(x)(z-y)|\big|}{|z-x|+|y-x|}\\
&\leq
\frac{|f(z)-f(x)-Df(x)(z-x)|}{|z-x|}+
\frac{|f(y)-f(x)-Df(x)(y-x)|}{|y-x|},
\end{split}
\]
it also follows that
$$
\lim_{(y,z)\to (x,x)}\frac{|f(z)-f(y)|-\sigma_x(z-y)}{|z-x|+|y-x|}=0.
$$
Note that $\sigma_x(v)=|Df(x)v|$ is a seminorm on $\bbbr^n$.

Recall that $\sigma\colon \bbbr^n\to [0,\infty)$ is a {\em seminorm} if
$\sigma(\lambda v)=|\lambda|\sigma(v)$ and $\sigma(v+w)\leq\sigma(v)+\sigma(w)$ for all $\lambda\in\bbbr$ and $v,w\in\bbbr^n$. Thus a seminorm is like a norm, but it may vanish on a non-trivial linear subspace of $\bbbr^n$:
\begin{equation}
\label{INeq6}
N_\sigma:=\ker \sigma=\{v\in \bbbr^n:\, \sigma(v)=0\}.
\end{equation}
\begin{definition}
\label{def101}
Let $f:\Omega\to X$ be a map between an open set $\Omega\subset\bbbr^n$ and a metric space $(X,d)$. We say that $f$ is {\em metrically differentiable} at $x\in\Omega$ if there is a seminorm $\sigma_x$ on $\bbbr^n$ such that
$$
\lim_{y\to x} \frac{d(f(y),f(x))-\sigma_x(y-x)}{|y-x|}=0.
$$
We say that $f$ is {\em strongly metrically differentiable} at $x\in\Omega$, if there is a seminorm $\sigma_x$ on $\bbbr^n$ such that
$$
\lim_{(y,z)\to (x,x)} \frac{d(f(z),f(y))-\sigma_x(z-y)}{|z-x|+|y-x|}=0.
$$
If $f$ is strongly metrically differentiable, then it is metrically differentiable (take $z=x$).

If $f$ is metrically differentiable at $x$, then the seminorm $\sigma_x$ is unique (easy exercise) and we denote it by
$\md (f,x)$, i.e.,
$$
\lim_{y\to x} \frac{d(f(y),f(x))-\md (f,x)(y-x)}{|y-x|}=0.
$$
\end{definition}
The estimates shown above along with the Rademacher theorem show that Lipschitz mappings $f:\bbbr^n\supset\Omega\to\bbbr^m$ are strongly metrically differentiable a.e.\ with $\md (f,x)(v)=|Df(x)v|$.

The next result is a celebrated theorem of Kirchheim~\cite[Theorem~2]{kir} known also as the Kirchheim-Rademacher theorem, see Theorems~\ref{T2} and Corollary~\ref{MDT4}.
\begin{theorem}[Kirchheim-Rademacher Theorem]
\label{INT1}
If $f:\bbbr^n\supset\Omega\to X$ is a Lipschitz mapping from an open set to an arbitrary metric space $X$, then $f$ is strongly metrically differentiable a.e.
\end{theorem}
A detailed proof of Theorem~\ref{INT1} appears also in \cite[Theorem~3.2]{reichel} while \cite[Theorem~3.2]{ambkir} proves only that Lipschitz mappings are metrically differentiable a.e., which is a weaker result.

Since metric spaces have no linear structure, the Kirchheim-Rademacher theorem may seem very surprising. However, by the Fr\'echet-Kuratowski Theorem~\ref{BT3}, every separable metric space admits an isometric embedding into $\ell^\infty$, the Banach space of bounded real sequences, and hence we can associate with any Lipschitz mapping $f:\bbbr^n\supset\Omega\to X$ a Lipschitz mapping into $\ell^\infty$. While we do not assume that $X$ is separable, the image $f(\Omega)\subset X$ is separable and we can assume that $f(\Omega)$ is a subset of $\ell^\infty$. Hence, we can view $f$ as a mapping
\begin{equation}
\label{INeq7}
f=(f_1,f_2,\ldots):\bbbr^n\supset\Omega\to f(\Omega)\subset\ell^\infty.
\end{equation}
Therefore, it suffices to prove Theorem~\ref{INT1} for $X=\ell^\infty$.

The coordinate functions $f_i$ are Lipschitz continuous so by Rademacher's theorem they are differentiable a.e.\ and this gives a weak notion of differentiability of Lipschitz mappings into metric spaces. This is the main idea used in the proof of the Kirchheim-Rademacher theorem. One should not be, however, deceived by this simple idea: the proof of the Kirchheim-Rademacher theorem is difficult.

The assumption that $f:\bbbr^n\supset\Omega\to X$ is defined on an open set is very strong and it is more reasonable to assume that $f:\bbbr^n\supset A\to X$ is a Lipschitz mapping defined on a measurable set. It turns out that the results of the paper easily generalize to that case. Indeed, if we embed $f(A)$ isometrically into $\ell^\infty$, then
$$
f:\bbbr^n\supset A\to f(A)\subset\ell^\infty
\quad
\text{has a Lipschitz extension}
\quad
F:\bbbr^n\to \ell^\infty,
$$
see Corollary~\ref{BT2}. Then we simply apply the results to $F$ and restrict them to $A$. One issue is that the definition of the metric derivative of $f$ has to be replaced by the {\em approximate metric derivative} which is discussed in Section~\ref{AD1}. For the sake of simplicity, in the rest of the Introduction we will assume that $f$ is defined on an open set.

If we know that $f=(f_1,f_2,\ldots):\bbbr^n\supset\Omega\to\ell^\infty$ is Fr\'echet differentiable at $x\in\Omega$, then it is strongly metrically differentiable at $x$ with
$$
\md(f,x)(v)=\sup_{i\in \bbbn}|Df_i(x)v|.
$$
It easily follows from a modification of the proof that Fr\'echet differentiability of $f:\bbbr^n\supset\Omega\to\bbbr^m$ implies strong metric differentiability, see Proposition~\ref{T4}. However, the main difficulty is that in general Lipschitz mappings into $\ell^\infty$ need not be Fr\'echet differentiable at any point, see Proposition~\ref{T5}.

If $\sigma:\bbbr^n\to [0,\infty)$ is a seminorm, then it may vanish on a linear subspace $N_\sigma:=\{v\in \bbbr^n:\, \sigma(v)=0\}$ and we define the {\em rank of the seminorm} as $\rank\sigma:=n-\dim N_\sigma$. That is, $\rank\sigma$ equals the maximum of dimensions of subspaces of $\bbbr^n$ on which $\sigma$ is a norm.
In particular we define the {\em rank of the metric derivative} as the rank of the seminorm $\md(f,x)$.

Note that if $f:\bbbr^n\supset\Omega\to\bbbr^m$ is differentiable at $x\in\Omega$, then $\md(f,x)(v)=|Df(x)v|$ and hence
$\rank\md(f,x)=\rank Df(x)$. Thus the notion of the rank of the metric derivative is consistent with the notion of the rank of the Fr\'echet derivative.

The classical Sard theorem for Lipschitz mappings (see Theorem~\ref{T15}) generalizes to the case of metric-valued mappings as follows, see Theorem~\ref{T17}.
\begin{theorem}[Metric Sard Theorem]
\label{INT6}
If $f:\mathbb{R}^n\supset\Omega\to X$ is Lipschitz and
$$
\operatorname{Crit}(f):=\{x\in\Omega:\, \rank\md (f,x)<n\},
$$
then $\mathcal{H}^n(f(\operatorname{Crit}(f)))=0$.
\end{theorem}
\begin{remark}
Adding to $\operatorname{Crit}(f)$ the set $Z$ where $\md(f,x)$ is not defined does not change anything, because $\H^n(Z)=0$ and hence $\H^n(f(Z))=0$.
\end{remark}

Again, the main difficulty is that the proof in the Euclidean case (see the proof of Theorem~\ref{T15}) uses Fr\'echet differentiability in an essential way which we cannot use now.

Later, we will prove a much deeper version of the Metric Sard Theorem, see Theorem~\ref{CT3}. We believe that Theorem~\ref{CT3} has not been known in such generality. As a corollary we conclude a general result about rectifiability of fibers of Lipschitz maps, see Theorem~\ref{8.9}:
\begin{theorem}
Let $0\leq m\leq n$, $m\in\bbbz$, $n\in\bbbn$, and let
$f:\bbbr^n\supset A\to X$ be a Lipschitz map from a measurable set to an arbitrary metric space. Then, for
$\mathcal{H}^m$-almost every $z\in X$, the fiber $f^{-1}(z)$ is
countably $\mathcal{H}^{n-m}$-rectifiable.
\end{theorem}
This result is new. Previously, the same conclusion was known only under the additional assumption that $\H^m$ is $\sigma$-finite on $X$, or under a suitable rank assumption on the metric derivative of $f$, see
\cite[Corollary~1.8]{HZ}, \cite[Theorem~1.2]{kar}, \cite[Theorem~4.16]{reichel}.

One of the main features used in this paper is the fact that a Lipschitz function coincides with a $C^1$ function outside a set of arbitrarily small measure, see Theorem~\ref{BT5}. This makes it possible to apply the classical analysis of $C^1$ maps within the framework of Lipschitz mappings into metric spaces. This technique in the context of Lipschitz mappings into metric spaces has already been used in \cite{hajlaszma,HMZ,HZ}.

In particular, it leads to a generalization of the implicit function theorem to Lipschitz mappings with values in metric spaces; see Theorems~\ref{6:corr2} and~\ref{6:thm2}. We will state here only one version which is easier to state.

Let us begin with recalling the statement of the classical implicit function theorem. Let $n\geq m$ and let
$$
\pi:\bbbr^n=\bbbr^m\times\bbbr^{n-m}\to\bbbr^m,
\quad
\pi(x,y)=x
$$
be the orthogonal projection onto the first $m$ coordinates.

If $f:\bbbr^n\supset\Omega\to\bbbr^m$ is of class $C^1$ and $\rank Df(p)=m$, then there is a diffeomorphism $G$ in a neighborhood of $p$ such that
\begin{equation}
\label{INeq1}
(f\circ G^{-1})(x,y)=\pi(x,y)=x
\quad
\text{for all $(x,y)$ in a neighborhood of } G(p).
\end{equation}

\begin{theorem}[Metric Implicit Function Theorem]
\label{INT2}
Assume $1\leq m\leq n$ are integers and the Hausdorff measure $\H^m$ is $\sigma$-finite on a metric space $X$. Let $f\colon \bbbr^n\supset\Omega\to X$ be a Lipschitz map such that $\rank \md (f,p)=m$ for a.e.\ $p\in \Omega$.
Then there is a countable family $\{K_i\}$ of pairwise disjoint compact subsets of $\Omega$,  $\LL^n(\Omega\setminus \bigcup_i K_i)=0$, such that for each $i$ we have
\begin{itemize}
\item a diffeomorphism $G_i:\bbbr^n\to\bbbr^n=\bbbr^m\times\bbbr^{n-m}$
\item and a bi-Lipschitz $\phi_i:\pi(G_i(K_i))\to X$
\end{itemize}
such that
$$
(f\circ G_i^{-1})(x,y)=
(\phi_i\circ\pi)(x,y)=
\phi_i(x)
\quad
\text{for all }(x,y)\in G_i(K_i).
$$
\end{theorem}
One of the novelties here, even in the context of the classical implicit function theorem is that the diffeomorphisms $G_i$ are defined globally on $\bbbr^n$, while the diffeomorphism $G$ in the classical statement \eqref{INeq1} is defined locally. The existence of a global diffeomorphism follows from a beautiful argument of Palais (Theorem~\ref{Pal2}) that is not known enough, see also Lemma~\ref{TFU}.

While Theorem~\ref{INT2} seems new, other and related versions of the Metric Implicit Function Theorem have been studied in \cite{azzams,davids,HZ}.

Let us now state the classical area \cite{fed2} and co-area \cite{fed3} formulas due to Federer.

If $f:\bbbr^n\supset\Omega\to\bbbr^m$ is differentiable at $x_o\in\Omega$, we define the Jacobian of $f$ at $x_o$ by
\begin{equation}
\label{INeq2}
|J_n f|(x_o):=\sqrt{\det\big(Df(x_o)^T(Df)(x_o)\big)}
\quad
\text{ if } n\leq m,
\end{equation}
and
\begin{equation}
\label{INeq3}
|J_m f|(x_o):=\sqrt{\det\big(Df(x_o)Df(x_o)^T\big)}
\quad
\text{ if } n\geq m.
\end{equation}
Note that by the Jacobian we mean here its absolute value. When $n\neq m$, the Jacobian without the absolute value is not well defined, since the orientation of the corresponding
subspaces is not uniquely determined.
\begin{theorem}[Federer]
\label{INT3}
Let $f:\bbbr^n\supset\Omega\to\bbbr^m$ be Lipschitz. Then
\begin{itemize}
\item(Area formula) If $n\leq m$, then
$$
\int_\Omega |J_n f|(x)\, d\H^n(x)=\int_{\bbbr^m}\H^0(f^{-1}(y))\, d\H^n(y),
$$
\item (Co-area formula) If $n\geq m$, then
$$
\int_\Omega |J_m f|(x)\, d\H^n(x)=\int_{\bbbr^m}\H^{n-m}(f^{-1}(y))\, d\H^m(y).
$$
\end{itemize}
\end{theorem}
\begin{remark}
The Lebesgue measure in $\bbbr^n$ coincides with the Hausdorff measure $\H^n$ (Corollary~\ref{SJT2}) and for consistency, we prefer to use Hausdorff measure notation for the integration instead of the Lebesgue one, because we measure the size of the preimages using the Hausdorff measure anyway.

Here $\H^0$ stands for the counting measure  meaning that $\H^0(f^{-1}(y))$ is the cardinality of the set $f^{-1}(y)$. We prefer to use the Hausdorff measure notation $\H^0$ for the counting measure to be consistent with the statement of the co-area formula.
\end{remark}

Let us now explain the geometric meaning of the Jacobians defined in \eqref{INeq2} and \eqref{INeq3}. First, observe that if $m=n$, then both definitions give the same value and the area and the co-area formulas coincide in this case.

{\bf Case: $n\leq m$.} If $\rank Df(x_o)<n$, then $|J_n f|(x_o)=0$. Thus assume that $\rank Df(x_o)=n$, so $Df(x_o)$ is an isomorphism between $\bbbr^n$ and an $n$-dimensional subspace of $\bbbr^m$, and the Jacobian measures how the linear mapping $Df(x_o)$ changes measures of subsets between these spaces.

The linear map $Df(x_o)$ maps the set
$$
\{v:\, \md(f,x_o)(v)\leq 1\}=\{v:\, |Df(x_o)v|\leq 1\}
$$
onto the unit ball in the $n$-dimensional space $Df(x_o)(\bbbr^n)$. Therefore,
\begin{equation}
\label{INeq4}
|J_n f|(x_o)=\frac{\omega_n}{\H^n(\{v:\, \md(f,x_o)(v)\leq 1\})}\, ,
\end{equation}
where $\omega_n$ denotes the volume of the unit ball in $\bbbr^n$.

{\bf Case: $n\geq m$.} If $\rank Df(x_o)<m$, then $|J_m f|(x_o)=0$. Thus assume that $\rank Df(x_o)=m$.
Let
$N=\ker Df(x_o)=\ker \md(f,x_o)$, so
$Df(x_o)$ is an isomorphism of the orthogonal complement $N^\perp$ onto $\bbbr^m$.
Linear algebra shows that $|J_m f|(x_o)$ equals the Jacobian of the linear map:
\begin{equation}
\label{INeq5}
Df(x_o):N^\perp\to\bbbr^m,
\quad
\text{i.e.,}
\quad
|J_m f|(x_o)=\frac{\omega_m}{\H^m(\{v\in N^\perp:\, \md(f,x_o)(v)\leq 1\})}\, .
\end{equation}

In the case of Lipschitz mappings into a metric space $X$, the metric derivative is a general seminorm in $\bbbr^n$ and in order to generalize the area and the co-area formulas to the case of Lipschitz mappings into metric spaces, we need to define Jacobians of a general seminorm in $\bbbr^n$. Formulations \eqref{INeq4} and \eqref{INeq5} representing the classical Jacobians suggest the following definitions.

\begin{definition}
\label{INd2}
Let $\sigma$ be a seminorm on $\bbbr^n$. Recall that $N_\sigma=\ker\sigma$, see \eqref{INeq6}. We define
$$
\rank\sigma=n-\dim N_\sigma=\dim N_\sigma^\perp.
$$
Let $1\leq m\leq n$. If $\rank\sigma<m$, we set $|J_m(\sigma)|=0$. If $\rank\sigma=m$, we define
$$
|J_m(\sigma)|=\frac{\omega_m}{\H^m(\{v\in N_\sigma^\perp:\, \sigma(v)\leq 1\})}\, .
$$
\end{definition}
Note that we do not define $|J_m(\sigma)|$ for seminorms $\sigma$ with rank greater than $m$. This will be justified in a moment.

The following area formula is due to Kirchheim \cite[Theorem~7]{kir}; see Theorem~\ref{AFT1}.
\begin{theorem}[Metric Area Formula]
\label{INT4}
Let $f:\bbbr^n\supset\Omega\to X$ be a Lipschitz mapping from an open set into any metric space $X$. Then
$$
\int_\Omega|J_n(\md(f,x))|\, d\H^n(x)=\int_X \H^0(f^{-1}(y))\, d\H^n(y).
$$
\end{theorem}
For applications of this result see for example \cite{ambkir2,ambkir,DV,EH2,lytchakw,wenger,wengery}.

The co-area formula, however, does not hold for mappings into an arbitrary metric space and it requires that the target space has $\sigma$-finite Hausdorff measure $\H^m$, see Example~\ref{ex:4.11}. In that case we simply say that {\em $X$ is $\H^m$-$\sigma$-finite}. It turns out (see Corollary \ref{IT1}) that for such $X$ the metric derivative of a Lipschitz map $f:\bbbr^n\supset\Omega\to X$ has rank almost everywhere at most $m$. Therefore we need to define the $m$-Jacobian $|J_m(\sigma)|$ only when $\rank\sigma\leq m$.

The metric co-area formula was proved independently by Karmanova \cite{kar} and by Reichel \cite{reichel}.  In the case $X$ is a subset of a Euclidean space, the result was proved by Ohtsuka \cite{ohtsuka}. See also Theorem~\ref{thm:gCoa}.
\begin{theorem}[Metric Co-area Formula]
\label{INT5}
Assume that a metric space $X$ is $\H^m$-$\sigma$-finite. If $1\leq m\leq n$ and $f:\bbbr^n\supset \Omega\to X$ is a Lipschitz mapping defined on an open set, then
$$
\int_\Omega |J_m (\md(f,x))|\, d\H^n(x)=\int_X\H^{n-m}(f^{-1}(y))\, d\H^m(y).
$$
\end{theorem}

For applications of this result see for example \cite{volume, hajlaszkk,KL}.

\subsection*{Structure of the paper}
In Section~\ref{LF} we collect basic and well known facts regarding Lipschitz functions used later.

In Section~\ref{sec:metricdiff} we first prove that Lipschitz mappings $f:\bbbr^n\supset\Omega\to X$ are metrically differentiable a.e., see Proposition~\ref{MDT2} and Corollary~\ref{MDT3}. This is a weaker version than the strong metric differentiability in the Kirchheim-Rademacher Theorem~\ref{INT1}, which is presented in Section~\ref{SMD}, see Theorem~\ref{T2}. This result is a consequence of Theorem~\ref{T11} which can be regarded as a version of Scorza-Dragoni theorem. We conclude Section~\ref{sec:metricdiff} with extending the results to the case of Lipschitz mappings $f:\bbbr^n\supset A\to X$ defined on measurable sets. This requires the notion of approximate metric derivative.

In Section~\ref{MST} we first prove the classical Sard Theorem~\ref{T15}, and then we prove the Metric Sard Theorem~\ref{INT6}, see Theorem~\ref{T17}. The Metric Sard Theorem is not a simple adaptation of the proof of Theorem~\ref{T15} and it requires the strong metric differentiability, Theorem~\ref{INT1}.

In Section~\ref{MIFT} we prove two versions of of the implicit function theorem for Lipschitz mappings $f:\bbbr^n\supset\Omega\to X$, see Theorem~\ref{6:corr2} and Theorem~\ref{6:thm2}. While Theorem~\ref{6:corr2} is similar to the main result in \cite{HZ}, Theorem~\ref{6:thm2} is new and it plays a crucial role in the proof of the metric co-area formula.

Section~\ref{SJ} is devoted to the study of seminorms and Jacobians of seminorms defined in Definition~\ref{INd2}.

The metric area formula, Theorem~\ref{INT4} is proved in Section~\ref{AF}, see Theorem~\ref{AFT1}. The proof follows the original argument of Kirchheim. Let us emphasize that the material of Section~\ref{AF} is independent from Section~\ref{MIFT} so those who are interested only in the proof of Theorem~\ref{INT4} may skip Section~\ref{MIFT}.

The classical co-area inequality of Federer is stated in Section~\ref{coarea}; see Theorem~\ref{CT1}.
It follows from the stronger Theorem~\ref{CT2}, proved in \cite{EH1}, which is used here to establish a new version of the Metric Sard Theorem~\ref{CT3}. This version is stronger than Theorem~\ref{INT6} and plays a central role in the proof of the Metric Co-area Theorem~\ref{INT5}.
Theorem~\ref{CT2} is difficult. Although a complete proof can be found in \cite{EH1}, we include here a proof in the special case in which the Hausdorff measure $\H^m$ on the target space $Y$ is $\sigma$-finite. This special case is sufficient for the proof of the Metric Co-area Theorem~\ref{INT5}, where the same $\sigma$-finiteness assumption is imposed and necessary.

Section~\ref{MCOF} is devoted to the proof of the Metric Co-area Theorem~\ref{thm:gCoa}.

Section~\ref{ap:sec3} compares metric differentiability, strong metric differentiability, and classical Fr\'echet differentiability for Lipschitz mappings $f:\bbbr^n\supset\Omega\to\bbbr^m$. This section is independent of the rest of the paper and could be read immediately after the Introduction. We place it at the end because its results are not used elsewhere; their purpose is to clarify, through examples, the strengths and limitations of the notions of metric and strong metric differentiability. The main result proved there is Proposition~\ref{MDT1}; see Proposition~\ref{1:prop1}.

The paper concludes with an appendix, Section~\ref{APPE}, which collects several auxiliary results from linear algebra, measure theory, and the geometry of convex sets. Among other things, it contains
the Palais Extension Theorem \ref{Pal2},
a less standard version of Lusin's theorem, Theorem~\ref{TH5}, the asymptotic isodiametric inequality, Theorem~\ref{AT1}, the proof that $\H^n=\LL^n$ in $\bbbr^n$, Corollary~\ref{SJT2}, and the isodiametric inequality in finite-dimensional Banach spaces, Theorem~\ref{AT3}. All of these results are used in the paper.

\subsection*{Notation}
Since the paper is long, we have included most of the notation used throughout. This should be helpful especially for readers who do not read the paper from beginning to end, but instead wish to consult specific sections independently.

By $C$ we will denote a positive constant whose value may change in a single
string of estimates. If $A$ and $B$ are non-negative quantities, we write
$A \lesssim B$ and $A \simeq B$ if there is a constant $C \geq 1$ such that
$A \leq CB$ and $C^{-1}B \leq A \leq CB$, respectively. If the constant depends
on parameters $m,n,\ldots$, we write $A \lesssim_{m,n,\ldots} B$ and
$A \simeq_{m,n,\ldots} B$.

The symmetric difference of sets will be denoted by $A\triangle B=(A\setminus B)\cup (B\setminus A)$.

Throughout the paper, $\Omega$ will always denote an open set. If $\Omega\subset\bbbr^n$, then by saying that $E\subset\Omega$ is closed, we mean that $E$ is closed in $\bbbr^n$.
By $C_0^\infty(\Omega)$ we will denote the space of smooth functions with compact support in $\Omega$.

The set of positive integers will be denoted by $\mathbb N$; $\qp$ will denote the set of non-negative rational numbers. The Euclidean norm in
$\mathbb R^n$ will be denoted by $|\cdot|$, and the scalar product by
$x\cdot y$. Open and closed balls in a metric space will be denoted by
$B(x,r)$ and $\bar{B}(x,r)$. If $B=B(x,r)$ is a ball and $\lambda>0$,
then $\lambda B:=B(x,\lambda r)$.
If the ball is considered in $\mathbb R^n$,
we may write $B^n(x,r)$ and $\bar{B}^n(x,r)$. We write
$\bbbb^n=B^n(0,1)$ and $\Sph^{n-1}=\partial \bbbb^n$. The diameter of a set $A$ will be
denoted by $\operatorname{diam} A$.
The closure of $A$ will be denoted by $\overline{A}$.

The $s$-dimensional Hausdorff measure will be denoted by $\mathcal H^s$, and
$\mathcal H^s_\delta$ will denote the corresponding $\delta$-Hausdorff content, see Section~\ref{HAU}.
Note that $\H^s_\infty(A)=0$ if and only if $\H^s(A)=0$.
The Hausdorff measure is always computed with respect to the metric of the
ambient space. In $\mathbb R^n$, $\mathcal H^n$ agrees with the Lebesgue
measure $\LL^n$ (Corollary~\ref{SJT2}); we will also write $|A|$ for the Lebesgue measure of a measurable set
$A\subset \mathbb R^n$. The measure $\mathcal H^0$ is the counting measure.
If $\mathcal H^s$ is $\sigma$-finite on a metric space $X$, we say that $X$ is
$\mathcal H^s$-$\sigma$-finite. The phrase ``almost everywhere'' will always
refer to the measure clear from the context, usually $\LL^n=\H^n$ in
$\mathbb R^n$. The characteristic function of a set $A$ will be denoted by
$\chi_A$.

If $A\subset\bbbr^n$ is measurable, then $x\in\bbbr^n$ is a {\em density point of $A$} if
$$
\lim_{r\to 0}\frac{|B(x,r)\cap A|}{|B(x,r)|}=1.
$$
According to the Lebesgue differentiation theorem, almost every point $x\in A$ is a density point of $A$.

The volume of the Euclidean unit ball in $\mathbb R^n$ will be denoted by
$\omega_n$. More generally, for $s\geq 0$ we set
\[
        \omega_s=\frac{\pi^{s/2}}{\Gamma(s/2+1)} .
\]

Metric spaces will usually be denoted by $X$ and $Y$, and their metrics by
$d$, or by $d_X$ and $d_Y$ if more than one metric is involved. If
$f:X\to Y$ is Lipschitz, then $\operatorname{Lip}(f)$ denotes the least
Lipschitz constant of $f$. A map is called $L$-Lipschitz if
$\operatorname{Lip}(f)\leq L$.

By $\ell^\infty$ we denote the Banach space of bounded real sequences
$x=(x_i)_{i\in\mathbb N}$ with the norm
\[
        \|x\|_\infty=\sup_{i\in\mathbb N}|x_i|.
\]
The closed subspace of sequences converging to zero will be denoted by $c_0$.
If $f:\Omega\to \ell^\infty$, we write
$f=(f_1,f_2,\ldots)$ for its coordinate functions.

If $f:\bbbr^n\supset\Omega\to \mathbb R^m$ is differentiable at $x$, then
$Df(x)$ denotes its Fr\'echet derivative. If $f:\Omega\to\ell^\infty$,
$f=(f_1,f_2,\ldots)$, and all functions $f_i$ are differentiable at $x$, then
$D_c f(x)$ denotes the componentwise derivative (Definition~\ref{D1}),
\[
        D_c f(x)v=(\nabla f_1(x)\cdot v,\nabla f_2(x)\cdot v,\ldots);
\]
for a Lipschitz $f$ obviously $D_c f(x):\mathbb R^n\to\ell^\infty$.

If $f:\bbbr^n\supset\Omega\to X$ is metrically differentiable at $x$, its
metric derivative will be denoted by $\md(f,x)$ (Definition~\ref{def101}). If $f:A\to X$ is defined on a measurable set $A\subset\mathbb R^n$ and is approximately metrically
differentiable at $x\in A$, its approximate metric derivative will be denoted
by $\apmd(f,x)$, see Definition~\ref{def1}.

A seminorm on $\mathbb R^n$ will usually be denoted by $\sigma$ or $\tau$.
Its kernel will be denoted by
\[
        N_\sigma:=\ker\sigma=\{v\in\mathbb R^n:\sigma(v)=0\},
\]
and $N_\sigma^\perp$ will denote the orthogonal complement of $N_\sigma$ in
$\mathbb R^n$. The rank of $\sigma$ is
\[
        \operatorname{rank}\sigma
        :=n-\dim N_\sigma=\dim N_\sigma^\perp .
\]
The rank of $\md(f,x)$, or of $\apmd(f,x)$, is understood as the rank of the
corresponding seminorm. The space of seminorms on $\mathbb R^n$ will be
equipped with the distance
\[
        d_\infty(\sigma,\tau)
        :=\sup_{|v|=1}|\sigma(v)-\tau(v)|.
\]

If $\sigma$ is a norm on $\mathbb R^n$, then $\mathbb R^n_\sigma$ or
$(\mathbb R^n,\sigma)$ denotes $\mathbb R^n$ equipped with the metric
$d(x,y)=\sigma(x-y)$. Balls in this metric will be denoted by
$B^n_\sigma(x,r)$ and $\bar{B}^n_\sigma(x,r)$. For
$A\subset\mathbb R^n_\sigma$ we write
\[
        \operatorname{diam}_\sigma A
        :=\sup\{\sigma(x-y):x,y\in A\}.
\]
The Hausdorff measure and Hausdorff content in $\mathbb R^n_\sigma$ will be
denoted by $\mathcal H^s_\sigma$ and $\mathcal H^s_{\sigma,\delta}$,
respectively.

Let $1\leq m\leq n$. If $\sigma$ is a seminorm on $\mathbb R^n$, its
$m$-dimensional Jacobian is denoted by $|J_m(\sigma)|$. Thus
$|J_m(\sigma)|=0$ if $\operatorname{rank}\sigma<m$, while for
$\operatorname{rank}\sigma=m$,
\[
        |J_m(\sigma)|
        =
        \frac{\omega_m}
        {\mathcal H^m(\{v\in N_\sigma^\perp:\sigma(v)\leq 1\})}.
\]
If $L:\mathbb R^n\to\mathbb R^m$ is linear, then
\[
        |J_n L|=\sqrt{\det(L^T L)} \quad \text{when } n\leq m,
\]
and
\[
        |J_m L|=\sqrt{\det(LL^T)} \quad \text{when } n\geq m.
\]
For a differentiable map $f$, $|J_k f|(x)$ means the corresponding Jacobian of
$Df(x)$.

The Grassmannian of $m$-dimensional linear subspaces of $\mathbb R^n$ will be
denoted by $\operatorname{Gr}(m,n)$.

For a map $f:A\to X$ and $y\in X$, the multiplicity function, also called the
Banach indicatrix, is
\[
        N(f,A,y):=\operatorname{card}(f^{-1}(y)\cap A)
        =\mathcal H^0(f^{-1}(y)\cap A).
\]
For a Lipschitz map $f:A\to X$, where $A\subset\mathbb R^n$ is measurable, we
write
\[
        \operatorname{Crit}_m(f)
        :=
        \{x\in A:\operatorname{rank} \apmd(f,x)<m\}.
\]
If $A=\Omega\subset\bbbr^n$ is open and the ordinary metric derivative is used, then
\[
        \operatorname{Crit}(f)
        :=
        \{x\in\Omega:\operatorname{rank} \md(f,x)<n\}.
\]

If $(X,\mu)$ is a measure space and $f:X\to[0,\infty]$ is arbitrary, the upper
integral of $f$ will be denoted by
$\int_X^* f\,d\mu$.
It is defined as the infimum of $\int_X \phi\,d\mu$ over all measurable functions
$\phi:X\to[0,\infty]$ such that $f\leq \phi$ $\mu$-a.e. See Section~\ref{coarea}.

\section*{Acknowledgements}
Paweł Goldstein appreciates the hospitality of the University of Pittsburgh; his visits there were partially supported by grants from the IDUB programme and the IMAI Centre of the University of Warsaw.

This work was partially supported by the Simons Foundation grant (award no. SFI-MPS-T-Institutes-00010825) and by State Treasury funds as part of a task commissioned by the Minister of Science and Higher Education under the project ``Organization of the Simons Semesters at the Banach Center - New Energies in 2026-2028'' (agreement no. MNiSW/2025/DAP/491). The authors appreciate the hospitality of IM PAN during the Simons Semester in Geometric Analysis.

OpenAI’s ChatGPT was used solely to assist in identifying errors and typographical mistakes; the manuscript itself was written in its entirety before the widespread availability of generative AI.
The work of Piotr Hajłasz was supported in part by a grant of access to OpenAI models through the ChatGPT for Academic Researchers program.
\section{Lipschitz functions}
\label{LF}
In this introductory section we will collect basic properties of Lipschitz functions that will be used in the paper. Since all results collected here are well known, the presentation will be rather short.

Recall that a function $f:X\to Y$ between metric spaces is {\em Lipschitz continuous} (or, simply, \emph{Lipschitz}) if there is a constant $L\geq 0$ such that
\begin{equation}
\label{Beq1}
d_Y(f(x),f(y))\leq Ld_X(x,y)
\quad
\text{for all } x,y\in X.
\end{equation}
We say then that $f$ is {\em $L$-Lipschitz} and the smallest constant $L$ satisfying \eqref{Beq1} is called the Lipschitz constant of $f$ and is denoted by $\lip(f)$.

The next result is an important extension property of Lipschitz functions.
\begin{theorem}[McShane]
\label{BT1}
If $A\subset X$ is a subset of a metric space and $f:A\to \bbbr$ is $L$-Lipschitz, then there is an $L$-Lipschitz function $F:X\to\bbbr$ such that $F(x)=f(x)$ for all $x\in A$.
Moreover if $|f|\leq M$ on $A$, we can guarantee that $|F|\leq M$ on $X$.
\end{theorem}
\begin{proof}
It is easy to verify that the function $\tilde{F}(x):=\inf_{y\in A}(f(y)+Ld(x,y))$ is $L$-Lipschitz and $\tilde{F}(x)=f(x)$ for $x\in A$ so we can take $F=\tilde{F}$. If in addition $|f|\leq M$ on $A$, we define $F$ as a standard truncation of $\tilde{F}$:
$$
F(x)=
\begin{cases}
M & \text{if } \tilde{F}(x)\geq M\\
\tilde{F}(x) & \text{if } -M\leq \tilde{F}(x)\leq M\\
-M & \text{if } \tilde{F}(x)\leq -M.
\end{cases}
$$
\end{proof}
By $\ell^\infty$ we will denote the Banach space of bounded real sequences $x=(x_i)_{i\in\bbbn}$ with the norm $\Vert x\Vert_\infty=\sup_{i\in\bbbn}|x_i|$.

As a consequence of Theorem~\ref{BT1} we obtain
\begin{corollary}
\label{BT2}
If $A\subset X$ is a subset of a metric space and $f:A\to \ell^\infty$ is $L$-Lipschitz, then there is an $L$-Lipschitz function $F:X\to\ell^\infty$ such that $F(x)=f(x)$ for all $x\in A$.
\end{corollary}
\begin{proof}
If $f=(f_1,f_2,\ldots):A\to\ell^\infty$, applying the McShane lemma to each of the functions $f_i$ gives an $L$-Lipschitz function $F_i:X\to\bbbr$ and it is easy to verify that the function $F=(F_1,F_2,\ldots):X\to\ell^\infty$ satisfies the claim.
\end{proof}

The $\ell^\infty$ space plays an important role in the theory of metric spaces, because every separable metric space can be {\em isometrically} embedded into $\ell^\infty$. This equips a metric space with a linear structure and that will play a crucial role in the study of differentiability properties of Lipschitz mappings into metric spaces.
\begin{theorem}[Kuratowski-Fr\'echet]
\label{BT3}
Any separable metric space admits an isometric embedding into $\ell^\infty$.
\end{theorem}
\begin{proof}
Fix $x_o\in X$ and let $\{x_i\}_{i\in\bbbn}$ be a dense subset of $X$. It is easy to verify that the mapping $\kappa\colon X\to\ell^\infty$,
$$
x\mapsto\kappa(x):=(d(x,x_i)-d(x_o,x_i))_{i\in\bbbn},
$$
is an isometric embedding, meaning that $\Vert\kappa(x)-\kappa(y)\Vert_\infty=d(x,y)$ for all $x,y\in X$. We subtract the distance to $x_o$ to make sure that the sequence of components of $\kappa(x)$ is bounded.
\end{proof}
The next theorem is a fundamental result about the differentiability properties of Lipschitz functions.
\begin{theorem}[Rademacher]
\label{BT4}
If $f:\Omega\to\bbbr$ is a Lipschitz continuous function defined on an open set $\Omega\subset\bbbr^n$, then $f$ is (Fr\'echet) differentiable a.e.
\end{theorem}
For a proof see e.g., \cite[Theorem~3.2]{evans}.

Let $X$ be a separable metric space and let $f:\bbbr^n\supset\Omega\to X$ be Lipschitz. According to Theorem~\ref{BT3} we can assume that $X\subset\ell^\infty$, so we can write
$f=(f_1,f_2,\ldots):\Omega\to X\subset \ell^\infty$. Since each of the functions $f_i$ is Lipschitz continuous it follows from the Rademacher theorem that at almost every point $x\in\Omega$ all functions $f_i$ are differentiable at $x$. This observation will play a central role in our investigation.

Another important observation is Federer's theorem (\cite[p.\ 442]{fedsa}, \cite[Theorem~3]{whitney2}) according to which a Lipschitz function coincides with a $C^1$ function outside a set of arbitrarily small measure (Theorem~\ref{BT5}). This will allow us to regard Lipschitz functions essentially as $C^1$ functions and, in particular, we will be able to use the inverse function theorem even in the context of Lipschitz mappings into metric spaces $f=(f_1,f_2,\ldots):\Omega\to X\subset \ell^\infty$.
\begin{theorem}[Federer]
\label{BT5}
If $f:\Omega\to\bbbr$ is a Lipschitz continuous function defined on an open set $\Omega\subset\bbbr^n$, then for every $\eps>0$ there is an $f_\eps\in C^1(\bbbr^n)$
and a closed set $E\subset\Omega$ such that $|\Omega\setminus E|<\eps$ and
$$
f(x)=f_\eps(x)\text{ and } Df(x)= Df_\eps(x) \text{ for all } x\in E.
$$
Moreover, if $|f|\leq M$ in $\Omega$, then we can take $f_\eps$ with $|f_\eps|\leq 2M$ in $\bbbr^n$.
\end{theorem}
In the proof we will need the classical result of Whitney \cite{whitney}.
\begin{theorem}[Whitney]
\label{BT6}
Let $K\subset\bbbr^n$ be a compact set and let $f:K\to\bbbr$, $L:K\to\bbbr^n$
be continuous functions. Then there is a function $F\in C^1(\bbbr^n)$ such that
$$
F|_K=f
\qquad
\mbox{and}
\qquad
DF|_K=L
$$
if and only if
$$
\lim_{{\scriptstyle x,y\in K,\, x\neq y}\atop {\scriptstyle |x-y|\to 0}}
\frac{|f(y)-f(x)-L(x)(y-x)|}{|y-x|} = 0.
$$

\end{theorem}
For a direct proof see e.g. \cite[Theorem~6.10]{evans}.

\begin{proof}[Proof of Theorem~\ref{BT5}]
Let $L(x)=Df(x)$. By Rademacher's theorem, $L$ is defined a.e.\ and it is a measurable function.

Assume for a moment that $f:\bbbr^n\to\bbbr$ is defined on $\bbbr^n$ and that $f$ has compact support in an open ball $B$. According to the Lusin theorem there is a compact set $K'\subset B$ such that
$f$ is differentiable on $K'$,
$f|_{K'}$, $L|_{K'}$ are continuous, and $|B\setminus K'|<\eps/2$. Hence
\begin{equation}
\label{Beq2}
\lim_{{\scriptstyle K'\ni y\to x}\atop {\scriptstyle y\neq x}}
\frac{|f(y)-f(x)-L(x)(y-x)|}{|y-x|} =0
\quad
\mbox{for all $x\in K'$.}
\end{equation}
This condition is however, weaker than the one required in the Whitney theorem -- we need uniform
convergence over a compact set as $|x-y|\to 0$. Let
$$
R(x,y)=\frac{|f(y)-f(x)-L(x)(y-x)|}{|y-x|}
$$
and let
$$
\eta_k(x) = \sup\{R(x,y):\, K'\ni y\neq x,\ |x-y|<1/k\}
\quad
k=1,2,\ldots
$$
Note that the $\eta_k$ are measurable. Also, condition \eqref{Beq2} means that for every $x\in K'$,
$\eta_k(x)\to 0$ as $k\to\infty$.  According to Egorov's theorem
there is another compact set $K\subset K'$ such that $|K'\setminus K|<\eps/2$
and
$$
\eta_k\rightrightarrows 0
\quad
\text{uniformly on $K$ as $k\to\infty$.}
$$
Hence
$$
\lim_{{\scriptstyle x,y\in K, x\neq y}\atop {\scriptstyle |x-y|\to 0}}
\frac{|f(y)-f(x)-L(x)(y-x)|}{|y-x|} =0.
$$
According to the Whitney Theorem~\ref{BT6}, there is $F\in C^1(\bbbr^n)$ such that $F|_K=f|_K$,
$DF|_K=Df|_K$, $|B\setminus K|<\eps$. Multiplying $F$ by a function $\varphi\in C_0^\infty(B)$ that is equal to $1$ in a neighborhood of $K$
we may further assume that $F$ has compact support in $B$. Since both functions $f$ and  $F$ and their derivatives
vanish outside $B$, we have that $F=f$ and $DF=Df$ on $\bbbr^n$ except for a set of measure less than $\eps$.

The general case follows from a partition of unity argument.
\begin{lemma}[Partition of unity]
\label{WT1-T3}
Let $\Omega\subset\bbbr^n$ be open. Then there is a family of balls $B(x_i,r_i)\subset\Omega$, $r_i\in (0,1)$,
$i=1,2,\ldots$ and a family of functions $\varphi_i\in C_0^\infty(B(x_i,2r_i))$ such that
\begin{enumerate}
\item
$\bigcup_{i=1}^\infty B(x_i,r_i)=\Omega$;
\item
$\bar{B}(x_i,2r_i)\subset\Omega$;
\item
Every compact set $K\subset\Omega$ intersects only finitely many balls $B(x_i,2r_i)$.
\item
$
\sum_{i=1}^\infty \varphi_i(x)=1
$
for every $x\in\Omega$.
\end{enumerate}
\end{lemma}
We will not prove this lemma.

We now prove the general case in the proof of Theorem \ref{BT5}. Let $A\subset \Omega$ be a closed set such that
$|\Omega\setminus A|<\eps/2$.
Let $\{\varphi_i\}_{i=1}^{\infty}$ be the partition of unity from Lemma~\ref{WT1-T3}.
Let $I\subset\bbbn$ be the set of all indices $i$ such that $B(x_i,2r_i)\cap A\neq\varnothing$.

For $i\in I$, set $f_i=\varphi_i f$. Clearly, we can regard $f_i$ as a
Lipschitz function on $\mathbb R^n$ that vanishes outside $B(x_i,2r_i)$. Applying the compactly supported case to
$f_i$, we find $F_i\in C^1_0(B(x_i,2r_i))$ and measurable sets
$Z_i\subset \mathbb R^n$ such that
$\sum_{i\in I}|Z_i|<\eps/2$, and
\[
F_i=f_i, \text{ and } DF_i=Df_i
\quad\text{on } \mathbb R^n\setminus Z_i .
\]
Define $F=\sum_{i\in I}F_i$.
Since the sum is locally finite in $\mathbb R^n$, we have that $F\in C^1(\mathbb R^n)$.
Note that for $x\in A$,
$$
\sum_{i\in I}\varphi_i(x)= \sum_{i=1}^\infty\varphi_i(x)=1
\qquad
\text{and}
\qquad
\sum_{i\in I}D\varphi_i(x)= \sum_{i=1}^\infty D\varphi_i(x)=0.
$$
This easily implies that
$$
Df(x)=DF(x) \ \ \text{and}\ \ f(x)=F(x)
\ \
\text{on } A\setminus Z, \ \text{ where } \ Z=\bigcup_{i\in I} Z_i.
$$
Since $|\Omega\setminus(A\setminus Z)|\leq |\Omega\setminus A|+|Z|<\eps$, there is a closed set $E\subset A\setminus Z$ satisfying $|\Omega\setminus E|<\eps$. Thus $f_\eps=F$ satisfies the first part of the claim of Theorem \ref{BT5}.

Finally, assume that $|f|\leq M$. Let $\psi\in C^\infty(\bbbr)$ be such that $\psi(x)=x$ if $|x|\leq \tfrac{3}{2}M$, $\psi(x)=-2M$ if $x\leq -2M$ and $\psi(x)=2M$ if $x\geq 2M$. Let $F\in C^1(\bbbr^n)$ be the function constructed above, so that $F$ and $DF$ coincide with $f$ and $Df$, respectively, on a closed $E\subset \Omega$, with $|\Omega\setminus E|<\eps$. Set $f_\eps=\psi\circ F$. Clearly, $f_\eps=f$ and $D(f_\eps)=Df$ on $E$ and $|f_\eps|\leq 2M$.
\end{proof}

\section{Metric Differentiability}
\label{sec:metricdiff}
The aim of this section is to prove the Kirchheim–Rademacher Theorem~\ref{INT1}, see Theorem~\ref{T2}.

In Section~\ref{PMD} we discuss basic properties of the metric and strong metric derivatives.
In Section~\ref{MFD} we show that Fr\'echet differentiability of a mapping into $\ell^\infty$ easily implies strong metric differentiability, but then we  show an example that Lipschitz mappings into $\ell^\infty$ need not be Fr\'echet differentiable. This example shows that the proof of the Kirchheim-Rademacher theorem requires new ideas.

The proof of Theorem~\ref{INT1} starts in Section~\ref{MD1}. This is where we prove a weaker version of Theorem~\ref{T2} for mappings into $\ell^\infty$, where instead of strong metric differentiability, we prove a.e.\ metric differentiability. On the other hand we show an explicit formula for the metric derivative, see Proposition~\ref{MDT2}. Section~\ref{MSDT} is devoted to a proof of a version of the Scorza-Dragoni Theorem~\ref{T11} for metric valued maps. This result is used in the proof of the Kirchheim-Rademacher Theorem~\ref{INT1} that is presented in Section~\ref{SMD}. In the final Section~\ref{AD1} we extend the results to Lipschitz functions defined on measurable subsets of $\bbbr^n$ instead of open subsets. This requires the notion of the approximate metric derivative.


\subsection{Properties of the metric derivative}
\label{PMD}
Assume that $f:\bbbr^n\supset\Omega\to X$ is metrically differentiable at $x\in\Omega$; see Definition~\ref{def101}.
It follows that
\begin{equation}
\label{Eq16}
\lim_{t\to 0}\frac{d(f(x+tv),f(x))}{|t|}=\md (f,x)(v),
\qquad
\text{for all } v\in\bbbr^n.
\end{equation}
Thus the ``directional speed'' of $f$ exists at $x$ in every direction $v$, and it defines as a function of $v$, a seminorm on $\bbbr^n$.

If $f:\bbbr^n\supset\Omega\to X$ is strongly metrically differentiable at $x\in\Omega$, then taking $y=x+tv$ and $z=x+tw$ we get
\[
\lim_{t\to 0} \frac{d(f(x+tv),f(x+tw))}{|t|}=\md(f,x)(v-w),
\qquad
\text{for all } v,w\in\bbbr^n,
\]
which is a much stronger claim than the existence of the ``directional speed'' \eqref{Eq16}.

If $f$ is $L$-Lipschitz, \eqref{Eq16} gives
\begin{equation}
\label{Eq11}
\md(f,x)(v) \leq L|v|
\quad
\text{and}
\quad
|\md(f,x)(v)-\md(f,x)(w)|\leq  L|v-w|.
\end{equation}
The first estimate is an immediate consequence of \eqref{Eq16}, while the second one follows from \eqref{Eq16} and from the triangle inequality.

One immediately checks, using the fact that a seminorm on $\bbbr^n$ is continuous, that if $f$ is metrically differentiable at $x$, it is continuous there. One, however, needs to be aware that if $f:\bbbr^n\to\bbbr^m$, the metric differentiability is much weaker than the Fr\'echet differentiability, as the next example shows.
\begin{example}
\label{ex:abs}
If $\sigma$ is a seminorm on $\bbbr^n$, then $\sigma$, as a function, is metrically differentiable at $x=0$ and $\md(\sigma,0)(v)=\sigma(v)$. This follows immediately from the definition of the metric derivative. In particular,
$f:\bbbr\to\bbbr$, $f(x)=|x|$, is metrically differentiable at the origin and $\md (f,0)(v)=|v|$, for every $v$.
On the other hand, a non-zero seminorm on $\bbbr^n$ is {\em not} strongly metrically differentiable at $x=0$ (it suffices to take $z=-y$ in the definition).
\end{example}
The above example shows that the strong metric differentiability is indeed stronger than the metric differentiability.
However, if $f:\bbbr^n\to\bbbr^m$, $m>1$, is strongly metrically differentiable at $x_o\in\bbbr^n$, then $f$ need not be Fr\'echet differentiable at $x_o$:
$$
f:(-0.5,0.5)\to\bbbr^2,
\qquad
f(t)=\begin{cases}
\left(t\cos\sqrt{-\ln |t|},t\sin\sqrt{-\ln |t|}\right),&t\neq0,\\
(0,0),&t=0.
\end{cases}
$$
is Lipschitz, strongly metrically differentiable at $t=0$, but not Fr\'echet differentiable; for this and related results see Section~\ref{ap:sec3}.

Despite this example, it is still reasonable to think of strong metric differentiability as a good replacement for the Fr\'echet differentiability in the case of metric targets:
\begin{proposition}
\label{MDT1}
If $f:\bbbr^n\to\bbbr^m$ is strongly metrically differentiable at $p\in\bbbr^n$ and all partial derivatives $\partial_{x_i}f(p)$, $i=1,2,\ldots,n$ exist, then $f$ is Fr\'echet differentiable at $p$.
\end{proposition}
For a proof, see Proposition~\ref{1:prop1}.

This result and the Kirchheim-Rademacher Theorem~\ref{INT1} immediately imply the classical Rademacher theorem. Since Lipschitz functions of one variable are differentiable a.e., Fubini's theorem implies that a Lipschitz function $f:\bbbr^n\to\bbbr^m$ has partial derivatives a.e.\ (this is a starting point of every proof of the Rademacher theorem), and Proposition~\ref{MDT1} implies that $f$ is Fr\'echet differentiable a.e.

Another reason to view strong metric differentiability as a counterpart of Fr\'echet differentiability is explained in Section~\ref{MST}. While the proof of the classical Sard Theorem~\ref{T15} employs Fr\'echet differentiability, Metric Sard Theorem~\ref{T17} follows from strong metric differentiability.

\subsection{Fr\'echet differentiability}
\label{MFD}
Let $f:\bbbr^n\supset\Omega\to\ell^\infty$, $f=(f_1,f_2,\ldots)$ be a Lipschitz map defined on an open set.
By Rademacher's theorem, each component $f_i$ is differentiable a.e. Since the union of countably many sets of measure zero has measure zero, there is a Borel set $N\subset\Omega$ of measure zero, $\H^n(N)=0$, such that for each $i\in\bbbn$ and all $x\in D:=\Omega\setminus N$, $f_i$ is differentiable at $x$.
\begin{definition}
\label{D1}
Let $f:\bbbr^n\supset\Omega\to\ell^\infty$, $f=(f_1,f_2,\ldots)$ be a Lipschitz map defined on an open set and let the set $D\subset\Omega$ be defined as above. For $x\in D$, the {\em componentwise derivative} of $f$ at $x$ is a linear map $D_{\rm c}f(x):\bbbr^n\to\ell^\infty$ defined by the $\infty\times n$ matrix
$$
D_{\rm c}f(x)=
\left\lceil
\begin{array}{cccc}
\partial_1f_1 & \partial_2f_1 & \ldots & \partial_nf_1\\
\partial_1f_2 & \partial_2f_2 & \ldots & \partial_nf_2\\
\partial_1f_3 & \partial_2f_3 & \ldots & \partial_nf_3\\
\vdots & \vdots & \ldots & \vdots
\end{array}
\right\rceil:\bbbr^n\to\ell^\infty,
$$
so for $v\in\bbbr^n$,
$$
D_{\rm c}f(x)v=\big(\nabla f_1(x)\cdot v,\nabla f_2(x)\cdot v,\ldots\big).
$$
\end{definition}

In this section we will discuss a special case of metric differentiability which is very easy to prove.
\begin{proposition}
\label{T4}
If a Lipschitz map $f:\bbbr^n\supset\Omega\to\ell^\infty$ is Fr\'echet differentiable at $x\in\Omega$, then the componentwise derivative $D_{\rm c}f(x)$ is well-defined and equals the Fr\'echet derivative. Moreover, $f$ is strongly metrically differentiable at $x$ and
\begin{equation}
\label{MDeq1}
\md(f,x)(v)=\|D_{\rm c}f(x)(v)\|_\infty=\sup_{i\in\bbbn} |\nabla f_i(x)\cdot v|
\quad
\text{for all } v\in\bbbr^n.
\end{equation}
\end{proposition}
\begin{proof}
Let a linear map $L=(L_1,L_2,\ldots):\bbbr^n\to\ell^\infty$ be the Fr\'echet derivative of $f$ at $x$. Then
$$
\frac{\Vert f(y)-f(x)-L(y-x)\Vert_\infty}{|y-x|}=\sup_{i\in\bbbn}\frac{|f_i(y)-f_i(x)-L_i(y-x)|}{|y-x|}\to 0
\quad
\text{as $y\to x$}.
$$
It follows that for each $i\in\bbbn$, $f_i$ is differentiable at $x$ and $\nabla f_i(x)=L_i$. Therefore, $D_{\rm c}f(x)=L$, the componentwise derivative of $f$ at $x$.
Since
$$
f(z)-f(y)-D_{\rm c}f(x)(z-y)=
\big(f(z)-f(x)-L(z-x)\big)-\big(f(y)-f(x)-L(y-x)\big),
$$
the triangle inequality yields
\[
\begin{split}
&\left|\frac{\Vert f(z)-f(y)\Vert_\infty-\Vert D_{\rm c}f(x)(z-y)\Vert_\infty}{|z-x|+|y-x|}\right|\\
&\leq
\frac{\Vert f(z)-f(x)-L(z-x)\Vert_\infty}{|z-x|}+
\frac{\Vert f(y)-f(x)-L(y-x)\Vert_\infty}{|y-x|}
\stackrel{(y,z)\to (x,x)}\longrightarrow 0,
\end{split}
\]
so $f$ is strongly metrically differentiable at $x$ and \eqref{MDeq1} is satisfied.
\end{proof}
Let us emphasize that Proposition~\ref{T4} is easy and it does not give a clue how to prove Theorem~\ref{T2}. In fact, there are Lipschitz mappings $f:\bbbr^n\supset\Omega\to\ell^\infty$ that are nowhere Fr\'echet differentiable. The next example is well known.
\begin{proposition}
\label{T5}
The Lipschitz mapping $f:(0,1)\to\ell^\infty$ defined by
$$
f(x)=(f_1(x),f_2(x),\ldots),
\quad
f_i(x)=\frac{\sin (ix)}{i},
$$
is nowhere Fr\'echet differentiable.
\end{proposition}
\begin{proof}
First, note that $f$ maps $(0,1)$ into the closed subspace $c_0\subset\ell^\infty$ of sequences convergent to $0$, and that $f$ is Lipschitz continuous:
$$
\Vert f(x)-f(y)\Vert_\infty=\sup_{i\in\bbbn} \left|\frac{\sin (ix)-\sin(iy)}{i}\right|\leq |x-y|.
$$
Suppose to the contrary that $f$ is Fr\'echet differentiable at some $x\in(0,1)$. Then by Proposition~\ref{T4} the Fr\'echet derivative equals the componentwise derivative, so
$$
Df(x):\bbbr\to\ell^\infty,
\quad
Df(x)t=(t\cos x, t\cos(2x),t\cos(3x),\ldots)
$$
and the Fr\'echet differentiability would imply that
$$
\left\Vert \frac{f(x+t)-f(x)-Df(x)t}{t}\right\Vert_\infty=
\left\Vert \frac{f(x+t)-f(x)}{t}-Df(x)1\right\Vert_\infty
$$
would converge to $0$ as $t\to 0$. This is however, impossible because $(f(x+t)-f(x))/t\in c_0$ while
$$
Df(x)1=(\cos x,\cos(2x),\cos(3x),\ldots)\in\ell^\infty\setminus c_0
$$
and an element of $\ell^\infty\setminus c_0$ cannot be approximated by elements of $c_0$ in the $\ell^\infty$ norm.
\end{proof}

\subsection{Metric differentiability}
\label{MD1}
As was observed in \eqref{INeq7}, in order to prove metric differentiability or strong metric differentiability of $f:\bbbr^n\supset\Omega\to X$, we can assume that $X=\ell^\infty$. In the next result we prove metric differentiability of Lipschitz mappings into $\ell^\infty$. While metric differentiability is weaker than strong metric differentiability, the result will play an important role in the proof of Theorem~\ref{INT1}.

\begin{proposition}
\label{MDT2}
Let $f:\bbbr^n\supset\Omega\to\ell^\infty$, $f=(f_1,f_2,\ldots)$ be a Lipschitz mapping defined on an open set. Then $f$ is metrically differentiable a.e., and
$$
\md (f,x)(v)=\Vert D_{\rm c}f(x)v\Vert_\infty=\sup_{i\in\bbbn}|\nabla f_i(x)\cdot v|.
$$
for almost all $x\in\Omega$ and all $v\in\bbbr^n$.
\end{proposition}
\begin{remark}
\label{Rem1}
For every $v$, the function $x\mapsto\sup_i|\nabla f_i(x)\cdot v|$ is Borel measurable. This implies measurability of $x\mapsto\md(f,x)(v)$. While we assume here that the mapping is into $\ell^\infty$, the remark applies to Lipschitz mappings $f:\Omega\to X$ into an arbitrary metric space, because the image $f(\Omega)$ is separable and hence it can be isometrically embedded into $\ell^\infty$ (Theorem~\ref{BT3}).
For a deeper measurability result, see Corollary~\ref{T10}.
\end{remark}
\begin{corollary}
\label{MDT3}
If $f:\bbbr^n\supset\Omega\to X$ is a Lipschitz mapping from an open set into any metric space, then $f$ is metrically differentiable a.e.
\end{corollary}
\begin{proof}[Proof of Proposition~\ref{MDT2}]
Since the result is local  in nature, we may assume that $\Omega=\bbbr^n$. This will slightly simplify our notation.
Let $N\subset\bbbr^n$ be a set of measure zero defined at the beginning of Section~\ref{MFD}, and let $D_{\rm c}f(x):\bbbr^n\to\ell^\infty$, for $x\in\bbbr^n\setminus N$, be the componentwise derivative. Then for
any $v\in\bbbr^n$ we have
\begin{equation}
\label{Eq2}
\Vert D_{\rm c}f(x)v\Vert_\infty\leq\liminf_{t\to 0}\left\Vert\frac{f(x+tv)-f(x)}{t}\right\Vert_\infty.
\end{equation}
Indeed, if $i\in\bbbn$, then
$$
|\nabla f_i(x)\cdot v|= \lim_{t\to 0}\left|\frac{f_i(x+tv)-f_i(x)}{t}\right|\leq\liminf_{t\to 0}\left\Vert \frac{f(x+tv)-f(x)}{t}\right\Vert_\infty\, ,
$$
and \eqref{Eq2} follows upon taking the supremum over $i\in\bbbn$.

We will prove now that in fact, we have equality in \eqref{Eq2} i.e., for almost all $x\in\bbbr^n$ and all $v\in\bbbr^n$
$$
\Vert D_{\rm c}f(x)v\Vert_\infty = \lim_{t\to 0}\left\Vert \frac{f(x+tv)-f(x)}{t}\right\Vert_\infty.
$$
Fix $0\neq v\in\bbbr^n$. Assume that a line $\ell=\{x_o+tv:\, t\in\bbbr\}$, $x_o\in\bbbr^n$, intersects $N$ along a set of length zero. Then for almost all $z\in\ell$ (namely for all $z\in \ell\setminus N$) and all $i\in\bbbn$,
the directional derivatives exist and satisfy $D_vf_i(z)=\nabla f_i(z)\cdot v$. Since functions $f_i|_\ell$ are Lipschitz continuous, it follows that for all $x\in \ell$, all $t\in\bbbr$ and all $i\in\bbbn$,
$$
f_i(x+tv)-f_i(x)=
\int_0^t \frac{d}{d\tau} f_i(x+\tau v)\, d\tau=
\int_0^t\nabla f_i(x+\tau v)\cdot v\, d\tau.
$$
Let $W$ be the union of all lines $\ell$ that intersect $N$ along a set of length zero. By Fubini's theorem $|\bbbr^n\setminus W|=0$. Fix $x\in W$ and $t\in\bbbr$.
For any $\eps>0$
there is $i\in\bbbn$ such that
\begin{equation*}
\begin{split}
\Vert f(x+tv)-f(x)\Vert_\infty-\eps
&\leq
|f_i(x+tv)-f_i(x)|=\left|\int_0^t \nabla f_i(x+\tau v)\cdot v\, d\tau\right|\\
&\leq
\Big|\int_0^t \Vert D_{\rm c}f(x+\tau v)v\Vert_\infty\, d\tau\Big|.
\end{split}
\end{equation*}
(We put the absolute value over the last integral, because if $t<0$ the integral is non-positive.)
Since this inequality is true for any $\eps>0$, we have that for all $x\in W$,
\begin{equation}
\label{Eq4}
\Vert f(x+tv)-f(x)\Vert_\infty \leq \Big|\int_0^t \Vert D_{\rm c}f(x+\tau v)v\Vert_\infty\, d\tau\Big|
\quad
\text{for all $t\in\bbbr$.}
\end{equation}
All lines $\ell=\{x_o+tv:\, t\in\bbbr\}\subset W$ have the following two properties:
\begin{enumerate}
\item[(a)] The following function is measurable and bounded:
$$
\tau\mapsto\Vert D_cf(x_o+\tau v)v\Vert_\infty =\sup_{i\in\bbbn}\left|\frac{d}{d\tau}f_i(x_o+\tau v)\right|
$$
\item[(b)] For all $s\in\bbbr$, points $x=x_o+sv\in \ell$ satisfy \eqref{Eq4} i.e.,
$$
\Vert f((x_o+sv)+tv)-f(x_o+sv)\Vert_\infty\leq
\Big|\int_s^{t+s}\Vert D_{\rm c}f(x_o+\tau v)v\Vert_\infty\, d\tau\Big|.
$$
\end{enumerate}
Now (a) and Lebesgue's differentiation theorem imply that for almost all $s\in\bbbr$,
$$
\lim_{t\to 0} \frac{1}{|t|}\Big|\int_s^{t+s} \Vert D_{\rm c}f(x_o+\tau v)v\Vert_\infty\, d\tau\Big| = \Vert D_{\rm c}f(x_o+sv)v\Vert_\infty,
$$
which together with (b) yields that for almost all $s\in\bbbr$,
$$
\limsup_{t\to 0} \left\Vert \frac{f((x_o+sv)+tv)-f(x_o+sv)}{t}\right\Vert_\infty \leq \Vert D_{\rm c}f(x_o+sv)v\Vert_\infty \, .
$$
Since this is true for almost all points $x_o+sv\in \ell$ on all lines $\ell\subset W$, we conclude that there is a set $N_v\subset\bbbr^n$ of measure zero $\H^n(N_v)=0$ such that
for all $x\in \bbbr^n\setminus N_v$ we have
\begin{equation}
\label{Eq5}
\limsup_{t\to 0} \left\Vert \frac{f(x+tv)-f(x)}{t}\right\Vert_\infty \leq \Vert D_{\rm c}f(x)v\Vert_\infty.
\end{equation}
For each $0\neq v\in\bbbr^n$ we have a different exceptional set $N_v$.
Let $\{v_i\}_{i=1}^\infty\subset\Sph^{n-1}$ be countable and dense. Let $\tilde{N}=\bigcup_{i=1}^\infty N_{v_i}$. Clearly, $\H^n(\tilde{N})=0$. We will prove that for all $x\in\bbbr^n\setminus\tilde{N}$, \eqref{Eq5} is true for all $v\in\bbbr^n$.

Let $x\in\bbbr^n\setminus\tilde{N}$. Then \eqref{Eq5} is true for all $v=v_i$. Since both sides of \eqref{Eq5} are $1$-homogeneous with respect to $v$, \eqref{Eq5} is also true for $v=\lambda v_i$, $\lambda>0$. Note that the set
$V:=\{\lambda v_i:\, \lambda>0,\ i\in\bbbn\}\subset\bbbr^n$ is dense.
It is easy to check that both sides of \eqref{Eq5} define functions of $v$ that are Lipschitz continuous on $\bbbr^n$. So the fact that inequality \eqref{Eq5} between Lipschitz functions is valid on a dense subset of $\bbbr^n$, implies that it is true for all $v\in\bbbr^n$. Inequalities \eqref{Eq5} and \eqref{Eq2} yield that for almost all $x\in\bbbr^n$ (namely for all $x\in\bbbr^n\setminus (N\cup\tilde{N})$)
\begin{equation}
\label{Eq6}
\Vert D_{\rm c}f(x)v\Vert_\infty = \lim_{t\to 0} \left\Vert \frac{f(x+tv)-f(x)}{t}\right\Vert_\infty
\quad
\text{for all $v\in\bbbr^n$.}
\end{equation}
We now prove a stronger fact that
$\md (f,x)(v)= \Vert D_{\rm c}f(x)v\Vert_\infty$
is the metric derivative of $f$ for all $x\in\bbbr^n\setminus (N\cup\tilde{N})$.
It is easy to see that $v\mapsto\Vert D_{\rm c}f(x)v\Vert_\infty$ is a seminorm and it remains to show that
\begin{equation}
\label{Eq7}
\lim_{t\to 0^+}\, \sup_{|v|=1} \left|\Big\Vert\frac{f(x+tv)-f(x)}{t} \Big\Vert_\infty-\Vert D_{\rm c}f(x)v\Vert_\infty\right|=0.
\end{equation}
Let as before, $\{v_i\}_{i=1}^\infty\subset\Sph^{n-1}$ be a dense subset. Given $\eps>0$, there is $p\in\bbbn$ such that for every $v\in\Sph^{n-1}$ there is $k\in \{1,2,\ldots,p\}$ such that
\begin{equation}
\label{Eq8}
|v-v_k|<\frac{\eps}{4L},
\end{equation}
where $L$ is the Lipschitz constant of $f$. It follows from \eqref{Eq6} that there is $\delta>0$ such that for all $0<|t|<\delta$
$$
\sup_{1\leq i\leq p} \left|\Big\Vert\frac{f(x+tv_i)-f(x)}{t}\Big\Vert_\infty-\Vert D_{\rm c}f(x) v_i\Vert_\infty\right|<\frac{\eps}{2}.
$$
Using an elementary inequality
$$
\big|\Vert a\Vert-\Vert b\Vert\big|\leq \big|\Vert a_k\Vert-\Vert b_k\Vert\big|+\Vert a-a_k\Vert+\Vert b-b_k\Vert,
$$
for any $0<|t|<\delta$, any $v\in\Sph^{n-1}$ and $v_k$ satisfying \eqref{Eq8}, we have
\begin{equation*}
\begin{split}
&\left|\Big\Vert\frac{f(x+tv)-f(x)}{t}\Big\Vert_\infty-\Vert D_{\rm c}f(x) v\Vert_\infty\right|
\leq
\left|\Big\Vert\frac{f(x+tv_k)-f(x)}{t}\Big\Vert_\infty-\Vert D_{\rm c}f(x) v_k\Vert_\infty\right|\\
&+
\Big\Vert\frac{f(x+tv)-f(x+tv_k)}{t}\Big\Vert_\infty
+
\Vert D_{\rm c}f(x)(v-v_k)\Vert_\infty\\
&\leq
\frac{\eps}{2}+L|v-v_k|+L|v-v_k|<\eps
\end{split}
\end{equation*}
and \eqref{Eq7} follows. The proof is complete.
\end{proof}

\subsection{Metric Scorza-Dragoni theorem}
\label{MSDT}
The main result of this section is Theorem~\ref{T11} which can be regarded as a version of the Scorza-Dragoni theorem (cf.\,  \cite[Theorem~3.8]{dacorogna}, \cite[Theorem~6.35]{FL}).

The spaces $C(\Sph^{n-1})$ and $C(\Sph^{n-1}\times [0,1])$ are separable metric spaces with metrics
$$
d_\infty(g,h)=\sup_{|v|=1}|g(v)-h(v)|
\quad
\text{and}
\quad
\bar{d}_\infty(g,h)=\sup_{|v|=1}\sup_{0\leq t\leq 1} |g(v,t)-h(v,t)|
$$
respectively. Separability easily follows from the Stone-Weierstrass theorem.
Note also that the restriction of continuous functions on $\Sph^{n-1}\times [0,1]$ to $\Sph^{n-1}\times \{ 0\}\simeq\Sph^{n-1}$ yields a continuous (surjective) map
$$
\pi: C(\Sph^{n-1}\times [0,1])\to C(\Sph^{n-1}).
$$

Let $f:\bbbr^n\supset\Omega\to X$ be Lipschitz continuous and let $D\subset\Omega$ be a Borel set such that $\H^n(\Omega\setminus D)=0$, and $f$ is metrically differentiable at all $x\in D$; see Remark~\ref{Rem1}.
Consider the map
$$
\Phi_f:D\to C(\Sph^{n-1}),
\quad
\Phi_f(x)(v)=\md (f,x)(v),\ |v|=1.
$$
The next lemma provides an elementary, but useful estimate for the continuity of the metric derivative.
\begin{lemma}
If $f:\bbbr^n\supset\Omega\to X$ is $L$-Lipschitz and metrically differentiable at $x,y\in\Omega$, then for any $v,w\in\bbbr^n$ we have
\begin{equation}
\label{Eq12}
|\md (f,x)(v)-\md (f,y)(w)|\leq L|v-w|+\min\{|v|,|w|\}\,d_\infty\big(\Phi_f(x),\Phi_f(y)\big).
\end{equation}
\end{lemma}
\begin{proof}
If $w=0$ (or similarly if $v=0$), \eqref{Eq11} yields
$$
|\md(f,x)(v)-\md(f,y)(w)|=\md(f,x)(v)\leq L|v|=L|v-w|.
$$
Thus we may assume that $v,w\neq 0$ and that $0<|w|\leq |v|$. Again \eqref{Eq11} gives
\begin{equation*}
\begin{split}
&|\md(f,x)(v)-\md(f,y)(w)|\\
&\leq
|\md(f,x)(v)-\md(f,x)(w)|
+
|w|\left|\md(f,x)\left(\frac{w}{|w|}\right)-\md(f,y)\left(\frac{w}{|w|}\right)\right|\\
&\leq
L|v-w|+|w|d_\infty(\Phi_f(x),\Phi_f(y)).
\end{split}
\end{equation*}
\end{proof}
Assume now that $f:\bbbr^n\to X$ is Lipschitz and consider the map
$$
\Psi_f:D\to C(\Sph^{n-1}\times [0,1]),
\quad
\Psi_f(x)(v,t)=
\begin{cases}
{\displaystyle\frac{d(f(x),f(x+tv))}{t}} & \text{if $0<t\leq 1$}\\
\md(f,x)(v), & \text{if $t=0$.}
\end{cases}
$$
Note that continuity of $\Psi_f(x):\Sph^{n-1}\times [0,1]\to\bbbr$ when $x\in D$ follows from the definition of metric differentiability and from the triangle inequality.

(The assumption that $f$ is defined on $\Omega=\bbbr^n$ is needed here; otherwise, for $x\in D$, the point $x+tv$ need not belong to $\Omega$.)

\begin{lemma}
\label{T9}
$\Psi_f:D\to C(\Sph^{n-1}\times [0,1])$ is Borel measurable.
\end{lemma}
Since $\pi:C(\Sph^{n-1}\times [0,1])\to C(\Sph^{n-1})$ is continuous, and $\Phi_f=\pi\circ\Psi_f$, we immediately obtain
\begin{corollary}
\label{T10}
$\Phi_f:D\to C(\Sph^{n-1})$ is Borel measurable.
\end{corollary}

\begin{proof}[Proof of Lemma~\ref{T9}]
We need to prove that the preimage of any open set is Borel.
Since the space $C(\Sph^{n-1}\times [0,1])$ is separable, any open set is the union of a countable family of closed balls and it suffices to show that the preimage of any closed ball is Borel.

Fix arbitrary $g\in C(\Sph^{n-1}\times [0,1])$ and $r>0$. We need to show that the set
$\Psi_f^{-1}(\bar{B}(g,r))\subset D$ is Borel. We will show that in fact, this set is the intersection of a closed set with $D$.

Let $\{v_i\}_{i=1}^\infty\subset\Sph^{n-1}$ and $\{ t_j\}_{j=1}^\infty\subset (0,1]$ be countable and dense. Note that the sets
$$
E_{ij} =\Big\{x\in\bbbr^n:\ \left|\frac{d(f(x),f(x+t_jv_i))}{t_j}-g(v_i,t_j)\right|\leq r\Big\}
$$
are closed sets. We have
\begin{equation*}
\begin{split}
\Psi_f^{-1}(\bar{B}(g,r))
&=
\{x\in D:\, \bar{d}_\infty(\Psi_f(x),g)\leq r\}\\
&=
\{x\in D:\, |\Psi_f(x)(v,t)-g(v,t)|\leq r\  \text{for all $|v|=1$ and $0\leq t\leq 1$}\}\\
&=
\{x\in D:\, |\Psi_f(x)(v_i,t_j)-g(v_i,t_j)|\leq r\  \text{for all $i,j\in\bbbn$}\}\\
&=
D\cap\bigcap_{i,j=1}^\infty E_{ij}.
\end{split}
\end{equation*}
\end{proof}
\begin{theorem}
\label{T11}
Let $f:\bbbr^n\to X$ be Lipschitz continuous. Then for any $\eps>0$ there is a set $F_\eps\subset D\subset\bbbr^n$ which is closed as a subset of $\bbbr^n$, such that $\H^n(\bbbr^n\setminus F_\eps)<\eps$ and
\begin{itemize}
\item[(a)]
$\Psi_f:F_\eps\to C(\Sph^{n-1}\times [0,1])$ is continuous.
\item[(b)]
$\md(f,\cdot)(\cdot):F_\eps\times\bbbr^n\to\bbbr$ is continuous.
\item[(c)]
On every compact subset $K\subset F_\eps$ we have
\begin{equation}
\label{Eq13}
\lim_{\mycom{|x-y|\to 0}{x\in K,\ y\in\bbbr^n}}
\frac{|d(f(x),f(y))-\md(f,x)(y-x)|}{|y-x|}=0.
\end{equation}
\end{itemize}
\end{theorem}
\begin{remark}
Meaning of \eqref{Eq13} is that
$\forall\ \eps>0$  $\exists\  \delta>0$ $\forall \  x\in K$ $\forall \ y\in\bbbr^n$
$$
0<|x-y|<\delta \implies \frac{|d(f(x),f(y))-\md(f,x)(y-x)|}{|y-x|}<\eps.
$$
\end{remark}
\begin{remark}
Part (b)  can be regarded as a version of the Scorza-Dragoni theorem (cf.\ \cite[Theorem~3.8]{dacorogna}, \cite[Theorem~6.35]{FL}).
\end{remark}
\begin{proof}
From Lemma~\ref{T9} and Lusin's theorem (Theorem~\ref{TH5}), there is a set $F_\eps\subset D$,
which is closed as a subset of $\bbbr^n$ such that
$\H^n(D\setminus F_\eps)<\eps$ and
$$
\Psi_f:F_\eps\to C(\Sph^{n-1}\times [0,1])
\quad
\text{and}
\quad
\Phi_f=\pi\circ\Psi_f:F_\eps\to C(\Sph^{n-1})
$$
are continuous. This proves (a). Now we show that (a) implies (b) and (c).
It is easy to see that (b) follows from continuity of $\Phi_f:F_\eps\to C(\Sph^{n-1})$ and from \eqref{Eq12}.

It remains to prove (c). Suppose to the contrary that (c) is not true. Then there are $\eps>0$ and sequences $x_k\in K$, $y_k\in\bbbr^n$, $0<|x_k-y_k|<\frac{1}{k}$, such that
\begin{equation}
\label{MDeq3}
\frac{|d(f(x_k),f(y_k))-\md(f,x_k)(y_k-x_k)|}{|y_k-x_k|}\geq\eps.
\end{equation}
If we write $y_k=x_k+t_kv_k$, $|v_k|=1$, $0<t_k<\frac{1}{k}$, \eqref{MDeq3} is equivalent to
$$
|\Psi_f(x_k)(v_k,t_k)-\Psi_f(x_k)(v_k,0)|=\left|\frac{d(f(x_k),f(x_k+t_kv_k))}{t_k}-\md(f,x_k)(v_k)\right|\geq \eps.
$$
By taking subsequences, we may further assume that $x_k\to x$ and $v_k\to v$. We have
\begin{equation*}
\begin{split}
\eps
&\leq
|\Psi_f(x_k)(v_k,t_k)-\Psi_f(x_k)(v_k,0)| \leq
|\Psi_f(x_k)(v_k,t_k)-\Psi_f(x)(v_k,t_k)|\\
&+
|\Psi_f(x)(v_k,t_k)-\Psi_f(x)(v_k,0)|
+
|\Psi_f(x)(v_k,0)-\Psi_f(x_k)(v_k,0)|\\
&\leq
\bar{d}_\infty (\Psi_f(x_k),\Psi_f(x))+
|\Psi_f(x)(v_k,t_k)-\Psi_f(x)(v_k,0)|+
d_\infty(\Phi_f(x),\Phi_f(x_k))\\
&=
A_k+B_k+C_k.
\end{split}
\end{equation*}
We used here the fact that
$$
\Psi_f(x)(v_k,0)=\Phi_f(x)(v_k)
\quad
\text{and}
\quad
\Psi_f(x_k)(v_k,0)=\Phi_f(x_k)(v_k).
$$
Clearly, $A_k,C_k\to 0$ as $k\to\infty$ by continuity of
$$
\Psi_f:F_\eps\supset K\to C(\Sph^{n-1}\times [0,1])
\quad
\text{and}
\quad
\Phi_f=\pi\circ\Psi_f:F_\eps\supset K\to C(\Sph^{n-1}).
$$
Finally, $B_k\to 0$ by continuity of
$\Psi_f(x)\in C(\Sph^{n-1}\times [0,1])$ at $(v,0)$. This however, contradicts the fact that
$A_k+B_k+C_k\geq\eps$.
\end{proof}

\subsection{Strong metric differentiability}
\label{SMD}
This section is devoted to the proof of Theorem~\ref{INT1} which we state again:
\begin{theorem}
\label{T2}
If $f:\bbbr^n\supset\Omega\to X$ is Lipschitz, then for almost all $x\in\Omega$, we have
\begin{equation}
\label{Eq15}
\lim_{\Omega\times\Omega\ni (y,z)\to (x,x)}
\frac{d(f(z),f(y))-\md(f,x)(z-y)}{|z-x|+|y-x|}=0.
\end{equation}
\end{theorem}

The next result is now an immediate consequence of Theorem~\ref{T2} and Proposition~\ref{MDT2}.
\begin{corollary}
\label{MDT4}
Let $f:\bbbr^n\supset\Omega\to\ell^\infty$, $f=(f_1,f_2,\ldots)$ be a Lipschitz mapping defined on an open set. Then $f$ is strongly metrically differentiable a.e., and
$$
\md (f,x)(v)=\Vert D_{\rm c}f(x)v\Vert_\infty=\sup_{i\in\bbbn}|\nabla f_i(x)\cdot v|.
$$
for almost all $x\in\Omega$ and all $v\in\bbbr^n$.
\end{corollary}
\begin{proof}[Proof of Theorem~\ref{T2}]
By a standard embedding and extension argument (Theorem~\ref{BT3} and Corollary~\ref{BT2}) we can assume that $f:\bbbr^n\to\ell^\infty$.

For $i=1,2,3,\ldots$ let $F_{1/i}\subset\bbbr^n$ be a closed subset as in Theorem~\ref{T11}.
Let $\tilde{F}_{1/i}$ be the set of density points of $F_{1/i}$. Since $F_{1/i}$ is closed, it follows that $\tilde{F}_{1/i}\subset {F}_{1/i}$. Let $E=\bigcup_{i=1}^\infty \tilde{F}_{1/i}$. Clearly, $\H^n(\bbbr^n\setminus E)=0$.
It suffices to show that \eqref{Eq15} is true for all $x\in E$.

Let $x\in E$. Since $y$ and $z$ play a symmetric role in \eqref{Eq15}, it suffices to show that if
\begin{equation}
\label{Eq17}
0<|y_k-x|\to 0
\quad
\text{and}
\quad
|z_k-x|\leq |y_k-x|,
\end{equation}
then
$$
\frac{d(f(z_k),f(y_k))-\md(f,x)(z_k-y_k)}{|y_k-x|}\to 0
$$
or that
\begin{equation}
\label{Eq18}
d(f(z_k),f(y_k))-\md(f,x)(z_k-y_k)=o(|y_k-x|).
\end{equation}
Since $x\in E$, there is $i\in\bbbn$ such that $x\in F_{1/i}$, and $x$ is a density point of $F_{1/i}$. It easily follows from the definition of the density point that there is $\tilde{y}_k\in F_{1/i}$  such that
\begin{equation}
\label{Eq19}
\frac{|\tilde{y}_k-y_k|}{|y_k-x|}\to 0 \ \text{as $k\to\infty$}
\quad
\text{and hence}
\quad
|\tilde{y}_k-y_k|\leq |y_k-x|\ \text{for $k\geq k_o$.}
\end{equation}
We have
\begin{equation*}
\begin{split}
& |d(f(z_k),f(y_k))-\md(f,x)(z_k-y_k)|
\leq
|d(f(z_k),f(y_k))-d(f(z_k),f(\tilde{y}_k))|\\
&+
|d(f(z_k),f(\tilde{y}_k))-\md (f,\tilde{y}_k)(z_k-\tilde{y}_k)|
+
|\md (f,\tilde{y}_k)(z_k-\tilde{y}_k)-\md (f,\tilde{y}_k)(z_k-{y}_k)|\\
&+
|\md (f,\tilde{y}_k)(z_k-{y}_k)-\md (f,x)(z_k-{y}_k)|
=
A_k+B_k+C_k+D_k.
\end{split}
\end{equation*}
If follows from the triangle inequality and from \eqref{Eq19} that
$$
A_k\leq d(f(\tilde{y}_k),f(y_k))\leq L|\tilde{y}_k-y_k|=o(|y_k-x|).
$$
Inequality \eqref{Eq11} yields
$$
C_k\leq L|\tilde{y}_k-y_k|=o(|y_k-x|).
$$
Now, Theorem~\ref{T11}(c) implies that
\begin{equation}
\label{Eq20}
\frac{B_k}{|z_k-\tilde{y}_k|}\to 0
\quad
\text{as $k\to\infty$},
\end{equation}
because $\tilde{y}_k\in F_{1/i}$ as a convergent sequence, is contained in a compact subset of $F_{1/i}$.
Inequalities \eqref{Eq17} and \eqref{Eq19} along with the triangle inequality imply that
$|z_k-\tilde{y}_k|\leq 3|y_k-x|$ for $k\geq k_o$ and hence $B_k=o(|y_k-x|)$, by \eqref{Eq20}.

It remains to estimate $D_k$. Since \eqref{Eq17} yields $|z_k-y_k|\leq 2|y_k-x|$, we have
$$
D_k\leq |z_k-y_k|\,d_\infty(\Phi_f(\tilde{y}_k),\Phi_f(x))=o(|y_k-x|),
$$
because $\Phi_f(\tilde{y}_k)\to\Phi_f(x)$ in $C(\Sph^{n-1})$ by continuity of $\Phi_f$ on $F_{1/i}$.
This proves \eqref{Eq18} and completes the proof of the theorem.
\end{proof}

\subsection{Approximate metric derivative}
\label{AD1}

Let us recall the classical definition of the approximate derivative.
\begin{definition}
Let $f:A\to\bbbr$ be a measurable function defined on a measurable set $A\subset\bbbr^n$. We say that
$f$ is {\em approximately differentiable} at $x\in A$ if there is a linear function $L:\bbbr^n\to\bbbr$
such that for any $\eps>0$ the set
$$
\Big\{ y\in A:\, \frac{|f(y)-f(x)-L(y-x)|}{|y-x|} <\eps \Big\}
$$
has $x$ as a density point.
\end{definition}
This definition is equivalent to another condition that is easier to work with.
\begin{proposition}
\label{ADT1}
A measurable function $f:A\to\bbbr$ defined in a measurable set $A\subset\bbbr^n$ is approximately
differentiable at $x\in A$ if and only if there is a measurable set $A_x\subset A$
and a linear function $L:\bbbr^n\to\bbbr$ such that $x$ is a density point of $A_x$  and
$$
\lim_{A_x\ni y\to x} \frac{|f(y)-f(x)-L(y-x)|}{|y-x|} = 0.
$$
\end{proposition}
The proof is a nice exercise; a complete proof can be found in \cite[Proposition~5.2]{GH2}. We will use the condition from Proposition~\ref{ADT1} to define approximate metric derivative. Note that whenever $x$ is a density point of $A_x\subset A$, it is also a density point of $A$, so checking for approximate differentiability makes sense only in the density points of the domain of the function.
\begin{definition}
\label{def1}
A measurable function $f:A\to X$ defined in a measurable set $A\subset\bbbr^n$
is {\em approximately metrically differentiable} at $x\in A$, if there exists a measurable set $A_x \subset A$ and a seminorm $\sigma_x$ on $\bbbr^n$ such that $x$ is a density point of $A_x$, and
\begin{equation}
\label{eq34}
    \lim_{A_x \ni y\to x} \frac{d(f(y),f(x)) - \sigma_x(y-x)}{|y-x|} = 0 \, .
\end{equation}
We call $\sigma_x$ the \emph{approximate metric derivative} of $f$ at $x$ and denote it by $\apmd(f,x)$.
\end{definition}
Note that since $x$ is a density point of $A$, the seminorm $\sigma_x$ is unique and in particular it does not depend on the choice of $A_x$.

Clearly if $f:\bbbr^n\supset\Omega\to X$ is metrically differentiable at $x$, then $\apmd(f,x)=\md(f,x)$. In fact, a stronger result is true:
\begin{proposition}
\label{ADT2}
Let $\Omega \subset \bbbr^n$ be open and $f:\bbbr^n\supset\Omega\to X$ be Lipschitz. Then $f$ is metrically differentiable at $x_o\in \Omega$ if and only if it is approximately metrically differentiable at $x_o$, in which case, $\md(f,x_o)=\apmd(f,x_o)$.
\end{proposition}
\begin{proof}
Clearly, metric differentiability implies approximate metric differentiability. Assume now that $f$ is approximately metrically differentiable at $x_o\in\Omega$, \eqref{eq34} is satisfied along a set $A_{x_o}$, and $x_o$ is a density point of $A_{x_o}$. It is easy to check that if $f$ is $L$-Lipschitz, then $\sigma_{x_o}$ is $L$-Lipschitz (by modifying the proof of \eqref{Eq11}).
If $x_o\neq y\to x_o$, then we can find $\tilde{y}\in A_{x_o}$ such that $|y-\tilde{y}|=o(|y-x_o|)$.
Note that
$$
|d(f(y),f(x_o))-d(f(\tilde{y}),f(x_o))|\leq L|y-\tilde{y}|=o(|y-x_o|)
$$
and
$$
|\sigma_{x_o}(y-x_o)-\sigma_{x_o}(\tilde{y}-x_o)|\leq L|y-\tilde{y}|=o(|y-x_o|).
$$
Now, \eqref{eq34} applied to $A_{x_o}\ni\tilde{y}\to x_o$ along with the above estimates (triangle inequality)
easily lead to the result. We leave details to the reader.
\end{proof}
\begin{remark}
\label{R2}
It turns out that results about metric differentiability can easily be extended to corresponding results about approximate metric differentiability. Indeed, assume that $f:\bbbr^n\supset A\to X$ is a Lipschitz map defined on a measurable set $A\subset\bbbr^n$. Embed $f(A)$ isometrically to $\ell^\infty$ (Theorem~\ref{BT3}) and
$$
\text{extend }\quad f:\bbbr^n\supset A\to\ell^\infty
\qquad
\text{to a Lipschitz}
\qquad
F:\bbbr^n\to\ell^\infty
$$
(Corollary~\ref{BT2}). If $F$ is metrically differentiable at $x_o\in A$ and $x_o$ is a density point of $A$, then $f$ is approximately metrically differentiable at $x_o$, and $\apmd(f,x_o)=\md(F,x_o)$. Thus, if a statement involving $F$ and $\md(F,x)$ holds for almost every $x\in\bbbr^n$, then the corresponding statement involving $f$ and $\apmd(f,x)$ holds for almost every $x\in A$.
\end{remark}

In particular, according to the Kirchheim-Rademacher Theorem~\ref{INT1}, $F$ is strongly metrically differentiable a.e. Since almost every point of $A$ is a density point we obtain that $f$ is a.e.\ strongly approximately metrically differentiable. We proved
\begin{theorem}
\label{thm:apKR}
If $f:\bbbr^n\supset A\to X$ is a Lipschitz mapping defined on a measurable set, then $f$ is strongly approximately metrically differentiable a.e.\ in $A$. More precisely, for almost every $x\in A$,
\begin{equation}
\label{ADeq1}
\lim_{A\times A\ni (y,z)\to(x,x)}\frac{d(f(z),f(y))-\apmd(f,x)(z-y)}{|z-x|+|y-x|}=0.
\end{equation}
\end{theorem}
Indeed, \eqref{ADeq1} holds at the density points of $A$ that are points of strong metric differentiability of $F$.

The same technique applies to many of the results that are presented in the subsequent sections. For example it applies to the metric Sard theorem, the metric implicit function theorems, and the metric area and co-area formulas, see Theorems~\ref{CT3}, \ref{AFT1} and~\ref{thm:gCoa}.

Using this technique, we can write down the following lemma, essentially a chain rule for approximately metrically differentiable maps. The proof is essentially the same as for the standard chain rule.
\begin{lemma}
\label{lem:chain}
Assume that $A\subset \bbbr^n$ is measurable, $X$ is a metric space,
and $f:A\to X$ is Lipschitz. Let $\Phi:\bbbr^n\to\bbbr^n$ be a $C^1$
diffeomorphism. If $f$ is approximately metrically differentiable at $p$, then $f\circ \Phi$ is approximately
metrically differentiable at $q=\Phi^{-1}(p)$ and
$$
\apmd(f\circ \Phi,q)=\apmd(f,p)\circ D\Phi(q).
$$
\end{lemma}

\begin{proof}
As in Remark~\ref{R2}, we may assume that $X\subset \ell^\infty$ and
extend $f$ to a Lipschitz map $F:\bbbr^n\to \ell^\infty$.

Recall that approximate metric differentiability of $f$ at $p$ implies that $p$ is a density point of $A$. Since $f$ is approximately metrically differentiable at $p$ and $F=f$ on $A$, also $F$ is approximately metrically differentiable at $p$, with the same metric derivative. Then Proposition~\ref{ADT2} yields that $F$ is metrically differentiable at $p$ and $\md(F,p)=\apmd(f,p)$.

First we note that $q$ is a density point of $\Phi^{-1}(A)$. Indeed,
since $\Phi$ is a $C^1$ diffeomorphism, it is locally $L$-bi-Lipschitz for some $L\geq 1$. Hence there is $r_0>0$ such that, for $0<r<r_0$,
$\Phi(B(q,r))\subset B(p,Lr)$,
and, for every measurable $E\subset B(q,r_0)$,
$|E|\leq L^n |\Phi(E)|$.

Therefore, for $0<r<r_0$,
$$
\frac{|B(q,r)\setminus \Phi^{-1}(A)|}{|B(q,r)|}
\leq
L^n
\frac{|B(p,Lr)\setminus A|}{\omega_n r^n}
=
L^{2n}
\frac{|B(p,Lr)\setminus A|}{|B(p,Lr)|}\stackrel{r\to 0}{\longrightarrow} 0,
$$
so $q$ is a density point of $\Phi^{-1}(A)$.

As $f\circ \Phi$ coincides with $F\circ \Phi$ on $\Phi^{-1}(A)$,
it remains to prove that $F\circ \Phi$ is metrically differentiable at
$q$ and that $\md(F\circ\Phi,q)=\md(F,p)\circ D\Phi(q)$.

Set $\sigma=\md(F,p)$. By Taylor's formula,
$$
\Phi(z)-p = D\Phi(q)(z-q)+\rho(z),
\qquad \text{where } |\rho(z)|=o(|z-q|).
$$
Writing $y=\Phi(z)$, we have
\begin{equation*}
\begin{split}
\left|\right. d((F\circ\Phi)(z)&,(F\circ\Phi)(q))
-\sigma(D\Phi(q)(z-q))\left.\right|  =
\left|
d(F(y),F(p))-\sigma(D\Phi(q)(z-q))
\right| \\
&\leq
\left|d(F(y),F(p))-\sigma(y-p)\right|+
\left|\sigma(y-p)-\sigma(D\Phi(q)(z-q))\right| \\
& \leq
\left|d(F(y),F(p))-\sigma(y-p)\right|+\sigma(\rho(z)),
\end{split}
\end{equation*}
thus
\begin{equation*}
\begin{split}
&\frac{\left|d((F\circ\Phi)(z),(F\circ\Phi)(q))
-\sigma(D\Phi(q)(z-q))
\right|}{|z-q|}\\
&\qquad\qquad\leq
\frac{
\left|
d(F(y),F(p))-\sigma(y-p)
\right|
}{|y-p|}
\frac{|y-p|}{|z-q|}
+\sigma\Big(\frac{\rho(z)}{|z-q|}\Big).
\end{split}
\end{equation*}
The first term tends to zero with $z\to q$ by the metric differentiability of $F$ at
$p$, because $y=\Phi(z)\to p$ and $|y-p|\leq L|z-q|$ for $z$ close to $q$.
The second term tends to zero by continuity of the seminorm $\sigma$. Hence $F\circ\Phi$ is metrically differentiable
at $q$, with
$$
\md(F\circ\Phi,q)=\sigma\circ D\Phi(q)
=\md(F,p)\circ D\Phi(q).
$$
Since $q$ is a density point of $\Phi^{-1}(A)$, this gives
$$
\apmd(f\circ \Phi,q)
=
\apmd(f,p)\circ D\Phi(q),
$$
as claimed.
\end{proof}

\section{Metric Sard Theorem}
\label{MST}
Throughout this section $X$ will be a metric space and $\Omega\subset\bbbr^n$ will be open.
\subsection{The classical Sard theorem}
The following version of Sard’s theorem applies to Lipschitz mappings between Euclidean spaces.
\begin{theorem}[Sard]
\label{T15}
If $f:\bbbr^n\supset\Omega\to\bbbr^m$ is Lipschitz and
$$
\operatorname{Crit}(f):=\{x\in\Omega:\, \rank Df(x)<n\},
$$
then $\H^n(f(\operatorname{Crit}(f)))=0$.
\end{theorem}
The aim of this section is to generalize this result to the case of Lipschitz mappings into metric spaces and in particular into $\ell^\infty$, see Theorems~\ref{T17} and~\ref{T18}. The first issue is that we have to define the rank of the metric derivative, but that is easy, see Definition~\ref{gum13}.
The second issue is, however, much more essential. The classical proof of Sard's Theorem~\ref{T15} (presented below) employs the Fr\'echet differentiability of $f$ in an essential way, but Lipschitz mappings into $\ell^\infty$ need not be Fr\'echet differentiable at any point, see Proposition~\ref{T5}. To overcome this difficulty, we will use strong metric differentiability, Theorem~\ref{INT1}.

\begin{proof}[Proof of Theorem~\ref{T15}]
By Theorem~\ref{BT1} we may assume that $\Omega=\bbbr^n$.

It suffices to show that $\H^n(f(Z))=0$, where $Z:=Q\cap\operatorname{Crit}(f)$ and $Q$ is an open cube of side length $1$. Let $L$ be the Lipschitz constant of $f$ and fix $\eps\in (0,L)$.

The Fr\'echet differentiability of $f$ at the points of $Z$ implies that for every $x\in Z$, there is $r_x>0$ such that $B(x,r_x)\subset Q$ and
$$
|f(y)-f(x)-Df(x)(y-x)|\leq\eps r_x
\quad
\text{for all } y\in B(x,5r_x).
$$
Hence,
$$
\dist (f(y),W_x)\leq\eps r_x
\quad
\text{for all } y\in B(x,5r_x),
$$
where $W_x:=f(x)+Df(x)(T_x\bbbr^n)$ is an affine subspace of $\bbbr^m$ passing through $f(x)$. Clearly,
\begin{equation}
\label{eq33}
f(B(x,5r_x))\subset B(f(x),5Lr_x)\cap\{z\in\bbbr^m:\, \dist (z,W_x)\leq\eps r_x\}.
\end{equation}
Since $x\in\operatorname{Crit}(f)$, we have that $\dim W_x=k\leq n-1$. We claim that
$$
\H^n_\infty\big(f(B(x,5r_x))\big)\lesssim_n \eps L^{n-1}r_x^n.
$$
Indeed,
since $\eps<L$, every point
$$
z\in B(f(x),5Lr_x)
   \cap\{w\in\bbbr^m:\dist(w,W_x)\leq\eps r_x\}
$$
lies at distance at most $\eps r_x$ from a point of
$$
B(f(x),6Lr_x)\cap W_x.
$$
The $k$-dimensional ball $B(f(x),6Lr_x)\cap W_x$ can be covered by
$$
C(k)\left(\frac{L}{\eps}\right)^k
    \lesssim_n \left(\frac{L}{\eps}\right)^k
$$
balls of radius $\eps r_x$ with centers in $W_x$.
Therefore the set on the right-hand side of \eqref{eq33} can be
covered by the same number of balls of radius $2\eps r_x$.
Hence
\begin{equation}
\label{eq35}
\H^n_\infty\big(f(B(x,5r_x))\big)\lesssim_n\Big(\frac{L}{\eps}\Big)^{n-1}(2\eps r_x)^n\lesssim_n\eps L^{n-1}r_x^n.
\end{equation}
From the covering $Z\subset\bigcup_{x\in Z} B(x,r_x)$ we can select a countable family of pairwise disjoint balls $\{B(x_i,r_{x_i})\}_{i\in I}$ (see Theorem~\ref{AT7}) such that $Z\subset\bigcup_{i\in I}B(x_i,5r_{x_i})$ and hence
$$
\H^n_\infty(f(Z))\leq\sum_{i\in I} \H^n_\infty\big(f(B(x_i,5r_{x_i}))\big)\lesssim_n
\eps L^{n-1}\sum_{i\in I} r_{x_i}^n
\lesssim_n\eps L^{n-1}.
$$
The last inequality follows from the fact that the balls $B(x_i,r_{x_i})$ are pairwise disjoint and contained in $Q$.
Letting $\eps\to 0^+$ yields $\H^n_\infty(f(Z))=0$ and hence $\H^n(f(Z))=0$.
\end{proof}

\subsection{Rank of the metric derivative}
In this subsection, we will define the rank of the metric derivative. Since the metric derivative is a seminorm in $\bbbr^n$, more generally we will define the rank of a seminorm.

If $\sigma$ is a seminorm on $\bbbr^n$, then its kernel, i.e., $N_\sigma:=\{v\in\bbbr^n\colon \sigma(v)=0\}$ is a linear subspace of $\bbbr^n$. Recall that $N_\sigma^\perp$ is the orthogonal complement of $N_\sigma$ in $\bbbr^n$.
\begin{definition}
\label{gum13}
The {\em rank of a seminorm} $\sigma$ on $\bbbr^n$ is
$$
\rank\sigma:=n-\dim N_\sigma=\dim N_\sigma^\perp.
$$
In particular, the {\em rank  of the metric derivative} of $f$ at $x$ is defined as the rank of the seminorm $\md(f,x)$.
\end{definition}
Clearly, $\rank\md(f,x)=n$ if and only if $\md(f,x)$ is a norm on $\bbbr^n$.

If $f:\bbbr^n\to\bbbr^m$ is Fr\'echet differentiable at $x\in\bbbr^n$, then it is metrically differentiable at $x$ with $\md(f,x)(v)=|Df(x)v|$, so in that case $\rank \md(f,x)=\rank Df(x)$.

Note that a seminorm $\sigma$ on $\bbbr^n$ restricted to $N_\sigma^\perp$, is a norm and $\rank\sigma$ can be characterized as the maximum of dimensions of linear subspaces $V\subset\bbbr^n$ such that $\sigma|_V$ is a norm.

Later we will also need
\begin{lemma}
\label{T16}
If $\pi:\bbbr^n\to N_\sigma^\perp$ is the orthogonal projection (so, $\ker\pi=N_\sigma$), then
$$
\sigma(v)=\sigma(\pi(v)),
\quad
\text{for all $v\in\bbbr^n$.}
$$
\end{lemma}
\begin{proof}
Since $v=\pi(v)+(v-\pi(v))$, we have
$
\sigma(v)\leq\sigma(\pi(v))+\sigma(v-\pi(v))=\sigma(\pi(v)).
$
Similarly, $\pi(v)=v+(\pi(v)-v)$ yields
$
\sigma(\pi(v))\leq\sigma(v)+\sigma(\pi(v)-v)=\sigma(v).
$
\end{proof}
The main (but simple) result of this subsection, Theorem~\ref{T13}, characterizes the rank of the metric derivative of a Lipschitz map into $\ell^\infty$.

If  $f=(f_1,f_2,\ldots)\colon\bbbr^n\to\ell^\infty$ is Lipschitz, then according to Proposition~\ref{MDT2}, for almost every $x\in \bbbr^n$, all $f_i$ are differentiable at $x$, $f$ is metrically differentiable at $x$, and
$$
\md(f,x)(v)=\Vert D_{\rm c}f(x)v\Vert_{\ell^\infty}=\sup_{i\in\bbbn}|\nabla f_i(x)\cdot v|
\quad
\text{for all $v\in\bbbr^n$.}
$$
Recall that the linear map $D_{\rm c}f(x)\colon\bbbr^n\to\ell^\infty$ is represented by an $\infty\times n$ matrix, see Definition~\ref{D1}.

Before we state Theorem~\ref{T13}, let us prove that for $\infty\times n$ matrices the row-rank equals the column-rank. While this is a well known fact for finite dimensional matrices, we want to make sure that the result is still true in the infinite dimensional case.

Let
$$
A=
\left\lceil
\begin{array}{cccc}
a_{11} & a_{12} & \ldots & a_{1n}\\
a_{21} & a_{22} & \ldots & a_{2n}\\
a_{31} & a_{32} & \ldots & a_{3n}\\
\vdots & \vdots & \ldots & \vdots
\end{array}
\right\rceil
$$
be an $\infty\times n$ (real) matrix. The span of rows is a linear subspace of $\bbbr^n$ and the span of columns is a linear subspace of $\bbbr^\infty$, the space of all real sequences. The {\em row-rank} and the {\em column-rank} of the matrix $A$ are defined, respectively, as dimensions of these linear subspaces.
\begin{lemma}
\label{gum8}
If $A$ is an $\infty\times n$ matrix, then the row-rank equals the column-rank.
\end{lemma}
\begin{proof}
Denote the row-rank of $A$ by $r$. Suppose, without loss of generality, that the first $r$ rows of $A$ are linearly independent. The finite $r\times n$ matrix composed of the first $r$ rows has column rank equal to $r$. Again, without loss of generality, we may assume that the first $r$ columns are linearly independent, so $\det[a_{ij}]_{1\leq i,j\leq r}\neq 0$.
This implies that the first $r$ columns of the original matrix $A$ are linearly independent and so the column rank of $A$ is at least $r$. To show that it is actually equal to $r$, it suffices to show that any $k$-th column of $A$, $k>r$, is a linear combination of the first $r$ columns. Fix $k>r$.

Let $A_r:=[a_{ij}]_{1\leq i,j\leq r}$ and $a_k:=[a_{ik}]_{1\leq i\leq r}$. Since the matrix $A_r$ is invertible, the equation $A_r\alpha=a_k$ has a unique solution $\alpha\in\bbbr^r$ i.e., there are unique coefficients $\alpha_1,\ldots,\alpha_r$, such that
\begin{equation}
\label{Eq30}
\sum_{j=1}^r \alpha_j a_{ij}=a_{ik}
\quad
\text{for $i=1,2,\ldots,r$.}
\end{equation}
For $i'>r$, the row-rank of the matrix
\begin{equation}
\label{Eq31}
\left[
\begin{array}{cccc}
a_{11} & \ldots & a_{1r} & a_{1k}\\
\vdots & \ldots & \vdots & \vdots \\
a_{r1} & \ldots & a_{rr} & a_{rk}\\
a_{i'1} & \ldots & a_{i'r} & a_{i'k}
\end{array}
\right]
\end{equation}
equals $r$. Since it is a finite matrix, the column-rank also equals $r$, and hence the last column in \eqref{Eq31} is a linear combination of the first $r$ columns. Uniqueness of coefficients $\alpha_i$ satisfying \eqref{Eq30} shows that the coefficients in such a linear combination must be equal to $\alpha_1,\ldots,\alpha_r$. Therefore, by looking at the last row we have that
$$
\sum_{j=1}^r \alpha_j a_{i'j}=a_{i'k}.
$$
Thus, we have proved that for any $i\in\bbbn$,
$$
\sum_{j=1}^r \alpha_j a_{ij}=a_{ik}.
$$
This shows that the (infinite) $k$-th column of $A$ is a linear combination of the first $r$ columns with coefficients $\alpha_1,\ldots,\alpha_r$. The proof is complete.
\end{proof}
\begin{theorem}
\label{T13}
If $f=(f_1,f_2,\ldots)\colon \bbbr^n\to\ell^\infty$ is Lipschitz, then at almost all $x\in \bbbr^n$,
$$
\rank\md(f,x) = \operatorname{row-rank\ of} D_{\rm c}f(x).
$$
In other words, almost everywhere, $\rank\md(f,x)$ is the maximal $r\geq 0$ with the property that there exist indices $i_1<i_2<\ldots<i_r$ such that the vectors
\begin{equation*}
\nabla f_{i_1}(x),\nabla f_{i_2}(x),\ldots,\nabla f_{i_r}(x)
\quad
\text{are linearly independent.}
\end{equation*}
\end{theorem}
\begin{proof}
According to Proposition~\ref{MDT2}, for almost all $x\in \bbbr^n$, and all $v\in\bbbr^n$, we have
$\md(f,x)(v)=\Vert D_{\rm c}f(x)v\Vert_{\infty}$.
Clearly,
$$
N:=\{v\in\bbbr^n:\, \md(f,x)(v)=0\}=\ker D_{\rm c}f(x).
$$
The decomposition $\bbbr^n=N\oplus N^\perp$ implies that $D_{\rm c}f(x)\bbbr^n=D_{\rm c}f(x) N^\perp$. Since, $D_cf(x)$ is a monomorphism on $N^\perp$,
$$
\dim (D_{\rm c}f(x)\bbbr^n)=\dim N^\perp = \rank\md(f,x).
$$
Obviously, $D_{\rm c}f(x)\bbbr^n$ equals the span of the vectors $D_{\rm c}f(x)e_i$, where $e_1,\cdots,e_n$ are the unit vectors in the canonical basis of $\bbbr^n$. Since $D_{\rm c}f(x)e_i$ is the $i$'th column of $D_{\rm c}f(x)$, it follows that $\dim (D_{\rm c}f(x)\bbbr^n)$ equals the column-rank of the matrix $D_{\rm c}f(x)$. However, by Lemma~\ref{gum8}, the column-rank is equal to the row-rank, which is equal, in turn, to the maximal number of linearly independent rows $\nabla f_i(x)$.
\end{proof}

\subsection{Sard theorem for Lipschitz mappings into metric spaces}
The next result is the main result of this section. Another, stronger version of Sard's theorem will be proved later, see Theorem~\ref{CT3}.
\begin{theorem}[Metric Sard Theorem I]
\label{T17}
If $f:\mathbb{R}^n\supset\Omega\to X$ is Lipschitz and
$$
\operatorname{Crit}(f):=\{x\in\Omega:\, \rank\md (f,x)<n\},
$$
then $\mathcal{H}^n(f(\operatorname{Crit}(f)))=0$.
\end{theorem}
In the case of Lipschitz mappings into $\ell^\infty$ the theorem reads as
\begin{theorem}
\label{T18}
If $f:\bbbr^n\supset\Omega\to\ell^\infty$ is Lipschitz, then
$$
\H^n(f(\{x\in\Omega:\, \rank D_{\rm c}f(x)<n\}))=0.
$$
\end{theorem}
For a short (but not easy) proof of Theorem~\ref{T18} that does not use the Kirchheim-Rademacher theorem see \cite[Theorem~2.2]{hajlaszma}.

The key argument in the proof of Theorem~\ref{T17}
is the next result which gives quantitative estimates for the covering of the image of the critical set by balls
(cf.\ \cite[Lemma~2.7]{hajlaszma} and \cite[Lemma~3.7]{HZ}).

\begin{proposition}
\label{T14}
Let $f:\bbbr^n\supset A \to X$ be an $L$-Lipschitz map defined on a measurable set $A\subset \bbbr^n$ and let
\begin{equation}
\label{eq37}
E_k:=\{x\in A\colon \rank\apmd(f,x)=k\},
\quad
0\leq k\leq n.
\end{equation}
Then almost every point $x\in E_k$ has the following property:

For every integer $m\geq 1$, there is $r_{x,m}>0$ such that for every $0<r\leq r_{x,m}$,  $f(A\cap B(x,r))$ can be covered by $m^k$ balls, each of radius
$3\sqrt{k}Lr/m$, in the case of $k>0$, and by one ball of radius $r/m$, in the case of $k=0$.
\end{proposition}
Comparing this to the proof of Theorem~\ref{T15},
the proposition plays the role of the estimate of number of balls used to prove
\eqref{eq35}. Then Theorem~\ref{T17} will follow from Proposition~\ref{T14} and the covering argument in a manner similar to that used in the proof of Theorem~\ref{T15}.

The idea of the proof of Proposition~\ref{T14} is as follows.
By covering $B(x,r)$ by a cube and dividing the cube into $m^n$ identical cubes,
we can divide $B(x,r)$ into $m^n$ many sets of diameter of order $r/m$. Then, $f$ maps these sets to $m^n$ sets of diameter bounded by $Lr/m$. However, the point of the proposition is that when rank is $k<n$, there is much more efficient way of covering the image of $B(x,r)$: by $m^k$ many rather than $m^n$ many. The proof below is based on the idea that in directions parallel to the kernel of the derivative, (on small scale) $f$ significantly shrinks down the diameters, so, in subdividing $B(x,r)$ we can have coarser meshes along those directions and still have the desired small diameter for the images of the subdivisions.
\begin{proof}[Proof of Proposition~\ref{T14}]
According to Remark~\ref{R2} we can assume that $A=\bbbr^n$, $X=\ell^\infty$ and
$$
E_k:=\{x\in \bbbr^n\colon \rank\md(f,x)=k\}.
$$

First, assume $k>0$. Let $\tilde{E}_k$ be the set of all points $x\in E_k$ such that
\begin{equation}
\label{Eq27}
d(f(y),f(z))-\md(f,x)(y-z)=o(|x-y|+|x-z|).
\end{equation}
By Theorem~\ref{T2}, $\H^n(E_k\setminus\tilde{E}_k)=0$ and we will show that the covering property in the statement of the proposition is true for all $x\in\tilde{E}_k$.
Fix $x\in\tilde{E}_k$. Let
$$
N_x=\{v\in\bbbr^n:\, \md(f,x)(v)=0\}
\quad
\text{so}
\quad
\dim N_x=n-k.
$$
By translating and rotating the coordinate system, we may assume that $x=0$ is the origin, and that
$$
N_x^\perp=\operatorname{span}\{ e_1,\ldots,e_k\}
\quad
\text{and}
\quad
N_x=\operatorname{span}\{e_{k+1},\ldots,e_n\}.
$$
If $r>0$, then the ball $B(x,r)=B(0,r)$ is contained in the cube $Q=[-r,r]^n$ so
$$
B(0,r)\cap N_x^\perp\subset [-r,r]^k\times\{0\}
\quad
\text{and}
\quad
B(0,r)\cap N_x\subset\{ 0\}\times [-r,r]^{n-k}.
$$
Given an integer $m\geq 1$, divide the cube $[-r,r]^k$ into the lattice of $m^k$ congruent cubes of edge length $2r/m$. Denote them by $\{ Q_\nu\}_{\nu=1}^{m^k}$. Then
$$
B(0,r)\subset [-r,r]^k\times [-r,r]^{n-k}=\bigcup_{\nu=1}^{m^k} \big(Q_\nu\times [-r,r]^{n-k}\big).
$$
The sets $Q_\nu\times [-r,r]^{n-k}$ are thin and long. We have
$$
f(B(0,r))\subset\bigcup_{\nu=1}^{m^k}f(Q_\nu\times [-r,r]^{n-k}).
$$
So far, all of this is true for any $r>0$. Now it suffices to show that there is $r_{x,m}>0$ such that if $0<r\leq r_{x,m}$, then
\begin{equation}
\label{Eq29}
\diam(f(Q_\nu\times [-r,r]^{n-k}))< 3\sqrt{k}Lr/m\, , \quad \text{for $\nu=1,\cdots,m^k$.}
\end{equation}
This, however, easily follows from \eqref{Eq27}. Since $x=0$, \eqref{Eq27} implies that there is $r_{x,m}>0$ such that for any $0<r\leq r_{x,m}$,
\begin{equation}
\label{Eq28}
|d(f(y),f(z))-\md(f,0)(y-z)|< \frac{\sqrt{k}Lr}{m}
\quad
\text{for all $y,z\in [-r,r]^n$.}
\end{equation}
Indeed, $|y|+|z|\lesssim r$ and on the right hand side of \eqref{Eq28} we can put $\eps r$ in place of $\sqrt{k}Lr/m$.
However, in order to have an estimate compatible with \eqref{eq36} below,
it is convenient to take this specific constant, $\sqrt{k}Lr/m$, instead of $\eps r$.

In particular, estimate \eqref{Eq28} is true for $y,z\in Q_\nu\times [-r,r]^{n-k}$.

Let $\pi:\bbbr^n\to N_x^\perp$ be the orthogonal projection. Then, Lemma~\ref{T16} and \eqref{Eq11} yield
$$
\md(f,0)(v)=\md(f,0)(\pi(v))\leq L|\pi(v)|.
$$
On the other hand, if $y,z\in Q_\nu\times [-r,r]^{n-k}$, then $|\pi(y-z)|\leq \diam Q_\nu = 2\sqrt{k}r/m$, and hence
\begin{equation}
\label{eq36}
\md(f,0)(y-z)\leq L|\pi(y-z)|\leq 2\sqrt{k}Lr/m\, ,
\end{equation}
which together with \eqref{Eq28} gives
$d(f(y),f(z))< 3\sqrt{k}Lr/m$. This proves \eqref{Eq29} and completes the proof in the case $k>0$.

When $k=0$, $\md(f,x)=0$ for every $x\in E_0$. Directly from the definition of metric derivative, there is $r_{x,m}>0$ such that
$$
d(f(y),f(x))<\frac{|y-x|}{m}<\frac{r}{m}
\quad
\text{for all $y\in B(x,r), \, 0< r \leq r_{x,m}$.}
$$
But this means that $f(B(x,r))\subset B(f(x),r/m)$.
\end{proof}
\begin{proof}[Proof of Theorem~\ref{T17}]
It suffices to show that for any open cube $Q$ of side length $1$,
$$
\H^n\big(f(\operatorname{Crit}(f)\cap Q)\big)=0.
$$
Since $\operatorname{Crit}(f)=\bigcup_{k=0}^{n-1} E_k$, where the sets $E_k$ are defined by \eqref{eq37}, it suffices to show that $\H^n\big(f(E_k\cap Q)\big)=0$ for all $0\leq k\leq n-1$. Fix $0\leq k\leq n-1$.

Let $E_k=\tilde{E}_k\cup Z_k$, where $\H^n(Z_k)=0$ and the covering property from Proposition~\ref{T14} holds for all $x\in \tilde{E}_k$. Since $\H^n(f(Z_k))=0$, it suffices to prove that $\H^n\big(f(\tilde{E}_k\cap Q)\big)=0$.

Fix an integer $m\geq 1$. For each $x\in\tilde{E}_k\cap Q$, let $r_x>0$ be such that
$$
0<5r_x\leq r_{x,m}
\quad
\text{and}
\quad
B(x,r_x)\subset Q.
$$
According to Theorem~\ref{AT7} from the covering
$$
\tilde{E}_k\cap Q\subset\bigcup_{x\in\tilde{E}_k\cap Q}B(x,r_x)
$$
we can select a countable family of pairwise disjoint balls $\{ B(x_i,r_{x_i})\}_{i\in I}$ such that
$$
\tilde{E}_k\cap Q\subset\bigcup_{i\in I} B(x_i,5r_{x_i}),
\quad
\text{so}
\quad
f(\tilde{E}_k\cap Q)\subset\bigcup_{i\in I} f(B(x_i,5r_{x_i})).
$$
Since $5r_{x_i}\leq r_{x_i,m}$, $f(B(x_i,5r_{x_i}))$ can be covered by $m^k$ balls of radius
$C(k,L) r_{x_i}m^{-1}$. Therefore,
$$
\H^n_\infty\big(f(B(x_i,5r_{x_i}))\big)\lesssim_{n,L} m^k\Big(\frac{r_{x_i}}{m}\Big)^n
$$
and hence
$$
\H^n_\infty\big(f(\tilde{E}_k\cap Q)\big)\lesssim_{n,L}m^{k-n}\sum_{i\in I}r_{x_i}^n\lesssim_nm^{k-n},
$$
because the balls $B(x_i,r_{x_i})$ are pairwise disjoint and contained in $Q$. Since $k<n$, letting $m\to\infty$ yields
$\H^n_\infty\big(f(\tilde{E}_k\cap Q)\big)=0$ and hence
$\H^n\big(f(\tilde{E}_k\cap Q)\big)=0$. The proof is complete.
\end{proof}

\section{Metric implicit function theorem}
\label{MIFT}

Throughout this section $n\geq m$ are positive integers; we will identify
 $$
 \bbbr^{n}=\bbbr^m\times\bbbr^{n-m}=\{(x,y)\colon x\in\bbbr^m, y \in \bbbr^{n-m}\} \, .
 $$


Let us recall the following variant of the classical Implicit Function Theorem for $C^1$ maps.
\begin{lemma}
\label{TFU}
Suppose $\Omega\subset\bbbr^n$ is open and $f:\Omega\to \bbbr^m$ is of class $C^1$. If for some $p\in \Omega$ we have $\rank Df(p)=m$, then there exist an open neighborhood $U$ of $p$ and a~$C^1$-diffeomorphism $G:\bbbr^n\to\bbbr^n$
such that
$$
(f\circ G^{-1})(x,y)=x \quad \text{ for all }(x,y)\in G(U).
$$
\end{lemma}
In other words: After a change of coordinates in the domain, in the neighborhood of $p$ the function $f$ becomes a projection onto the first $m$ coordinates.
\begin{remark}
Usually the Implicit Function Theorem claims the existence of a \emph{local} diffeomorphism $G:U\to G(U)$. However, using Palais result, Theorem~\ref{Pal2}, (and possibly shrinking $U$) we can extend $G$ to a global diffeomorphism of $\bbbr^n$, which simplifies the formulation of the lemma.
\end{remark}
\begin{remark}
\label{IR1}
If $n=m$ and $f$ is orientation preserving or if $n>m$, we can guarantee, that $G=\operatorname{id}$ outside a compact set.
\end{remark}
\begin{proof}[Sketch of the proof]
We can assume, possibly after relabeling coordinates in $\bbbr^n$, that $\rank D_x f(p)=m$. Let $\pi:\bbbr^n=\bbbr^m\times \bbbr^{n-m}\to \bbbr^m$ be the orthogonal projection, $\pi(x,y)=x$. Note that $F:\Omega\to\bbbr^n$, $F(x,y)=(f(x,y),y)$, has $\rank DF(p)=n$, thus $F$ is a diffeomorphism on some neighborhood $\tilde U$ of $p$. Let $\eps>0$ be such that $\bar{B}(p,\eps)\subset \tilde U$ and denote $U=B(p,\eps)$. Then $F$ satisfies the assumptions of Theorem~\ref{Pal2} on $U$ and we can find a global diffeomorphism $G:\bbbr^n\to\bbbr^n$ which agrees with $F$ on $U$. Finally, for all $(x,y)\in G(U)$
$$
(f\circ G^{-1})(x,y)=(\pi\circ F\circ G^{-1})(x,y)=\pi(x,y)=x.
$$
To prove the claim from Remark~\ref{IR1}, observe that if $F$ is orientation preserving, $G$ is identity outside a compact set by Theorem~\ref{Pal2}. If $m<n$ and $F$ is orientation reversing, we can replace $F$ in the above argument by
$$
\widetilde{F}(x,y)=(f(x,y),-y_1,y_2,\ldots,y_{n-m})
$$
which is orientation preserving.
\end{proof}

For a Lipschitz map $f:\bbbr^n\to\bbbr^m$, the picture near a point $p$ where the rank of its derivative is maximal (i.e., $\rank Df(p)=m$) need not be this simple. To start, $\rank Df(p)=m$ is no more an open condition, so for Lipschitz maps we shall study the behavior of $f$ on \emph{positive measure sets} where $\rank Df=m$.

\begin{proposition}
\label{6:prop1}
Suppose $f:\Omega\subset \bbbr^n\to\bbbr^m$ is Lipschitz and $A\subset\Omega$ is measurable, with $\H^n(A)>0$. If $\rank Df(p)=m$ for every $p\in A$, then there exist a compact set $K\subset A$, $\H^n(K)>0$, and a $C^1$-diffeomorphism $G:\bbbr^n\to\bbbr^n$ such that
\begin{equation}
\label{6:1}
(f\circ G^{-1})(x,y)=x \quad \text{ for all }(x,y)\in G(K).
\end{equation}
\end{proposition}

\begin{proof}
By the $C^1$-Lusin property of Lipschitz maps (Theorem~\ref{BT5}), $f$ coincides, outside a set of small measure, with a $C^1$ map: there is a $C^1$ map $g:\bbbr^n\to\bbbr^m$ and $A'\subset A$, $\H^n(A')>0$, such that $f(p)=g(p)$ and $Df(p)=Dg(p)$ for all $p\in A'$. Note that $\rank Dg(p)=m$ for all $p\in A'$.

Assume now that $p_o\in A'$ is a density point of $A'$, thus $\rank Dg(p_o)=m$. By Lemma~\ref{TFU} there is an open $U\ni p_o$ and a $C^1$ diffeomorphism $G:\bbbr^n\to\bbbr^n$ such that
$$
(g\circ G^{-1})(x,y)=x \quad \text{ for all }(x,y)\in G(U).
$$
Now, note that $\tilde K:=A'\cap U$ has positive measure. If $(x,y)\in G(\tilde K)\subset G(U)$, then $G^{-1}(x,y)\in A'$ and thus
$$
(f\circ G^{-1})(x,y)=(g\circ G^{-1})(x,y)=x,
$$
so \eqref{6:1} holds for all $(x,y)\in G(\tilde K)$. Finally, we find a compact $K\subset \tilde K$ of positive measure, which concludes the proof.
\end{proof}
\begin{remark}
If $f:\bbbr^n\to\bbbr^N$, $N>m$, and $\rank Df(p)=m$ for all $p\in A$, then we can compose $f$ with an orthogonal projection $\pi$ onto an $m$-dimensional coordinate subspace of $\bbbr^N$, so that $\rank D(\pi\circ f)(p)=m$ on a subset of $A$ of positive measure. Then we can apply Proposition \ref{6:prop1} to $\pi\circ f$. We will see a similar idea in Theorem~\ref{6:metric TFU} below.
\end{remark}

\begin{corollary}
\label{6:corr1}
Suppose $f:\bbbr^n\to \bbbr^m$ is Lipschitz and let $A\subset \bbbr^n$ be a measurable set such that $\rank Df(p)=m$ for a.e.\ $p\in A$ and $\H^n(A)>0$. Then there is a countable family $\{K_i\}_i$ of pairwise disjoint compact subsets of $A$, $\H^n(K_i)>0$, such that $\H^n(A\setminus \bigcup_i K_i)=0$ and for each $i$ there exists a diffeomorphism $G_i:\bbbr^n\to\bbbr^n$ such that
\begin{equation}
\label{6:corr1 eq}
(f\circ G_i^{-1})(x,y)=x \quad \text{ for all }(x,y)\in G_i(K_i).
\end{equation}
\end{corollary}
\begin{proof}
Without loss of generality we can assume that $\H^n(A)<\infty$.

Let $\mathcal{E}$ denote the family of all subsets $E\subset A$, which are, up to a null set,
unions of countable families
of disjoint compact sets, i.e., $E=\bigcup_i K_i\cup Z$, $\H^n(K_i)>0$, $\H^n(Z)=0$, such that \eqref{6:corr1 eq} holds for some diffeomorphism $G_i$, for each $i$ (in other words, the subsets of $A$ on which Corollary~\ref{6:corr1} holds). By Proposition \ref{6:prop1}, $\mathcal{E}$ is non-empty.

Since any measurable set is, up to a null set, a countable union of disjoint compact sets of positive measure, it is easy to see that a countable union of sets in $\mathcal{E}$ is again in $\mathcal{E}$.

Let  $\alpha=\sup\{\H^n(E)~:~E\in \mathcal{E}\}$. There is a sequence $E_i\in \mathcal{E}$ such that $\H^n(E_i)\to \alpha$. Then $E=\bigcup_i E_i\in\mathcal {E}$, satisfies $\H^n(E)=\alpha$.

We claim that $\H^n(A)=\alpha$, i.e., $\H^n(A\setminus E)=0$, and thus $A\in\mathcal{E}$ (which concludes the proof). Suppose otherwise, then we can apply Proposition \ref{6:prop1} to $A\setminus E$ in place of $A$, finding a compact set $K\subset A\setminus E$ with $\H^n(K)>0$ and a diffeomorphism $G:\bbbr^n\to\bbbr^n$ for which \eqref{6:corr1 eq} holds. Then $E\cup K\in \mathcal{E}$, while $\H^n(E\cup K)=\alpha+\H^n(K)>\alpha$, which is a contradiction.
\end{proof}

For Lipschitz maps with values in a metric space the `rank of the derivative' condition must be expressed in terms of the metric derivative -- and, if their domain is no longer an open subset of $\bbbr^n$, we need to resort to approximate metric derivatives introduced in Section \ref{AD1}.

The key result of this subsection is the following local version of the implicit function theorem for functions with values into metric spaces:

\begin{theorem}
\label{6:metric TFU}
Assume $X$ is a metric space, $A\subset\bbbr^n$ has positive measure and ${f:A\to X}$ is Lipschitz. Suppose $\rank \apmd(f,p)\geq m$ a.e.\ in $A$. Then
\begin{itemize}
\item[(A)] $\H^m(f(A))>0$;
\item[(B)] there is a compact set $K\subset A$ with $\H^n(K)>0$, a $C^1$ diffeomorphism $G:\bbbr^n\to\bbbr^n$ and a $\sqrt{m}$-Lipschitz map $\pi:X\to \bbbr^m$ such that if we denote $F=f\circ G^{-1}$, then
$$
(\pi\circ F)(x,y)=x \quad \text{ for all }(x,y)\in G(K),
$$
i.e., $\pi\circ F:G(K)\to\bbbr^m$ is a projection onto the first $m$ coordinates;
\item[(C)] $F^{-1}(F(x,y))\cap G(K)
\subset \{x\}\times \bbbr^{n-m}$ for all $(x,y)\in G(K)$;
\item[(D)] for all $y\in\bbbr^{n-m}$ the maps $F|_{(\bbbr^m\times\{y\})\cap G(K)}$ are bi-Lipschitz, with a uniform bi-Lipschitz constant (i.e., independent of $y$).
\end{itemize}
\end{theorem}
\begin{remark}
\label{6:rem1}
As before, the diffeomorphism $G$ should be understood as a coordinate change in the domain of $f$, which reveals the structure of $f$. After this coordinate change, $f$ and $K$ become $F$ and $G(K)$. By (C), `vertical' sections $(\{x\}\times \bbbr^{n-m})\cap G(K)$ are mapped by $F$, for varying $x$, into disjoint sets in $X$. By (D), the `horizontal' sections $(\bbbr^m\times\{y\})\cap G(K)$ are mapped by $F$ into (not necessarily disjoint) bi-Lipschitz `surfaces'.

The $\sqrt{m}$ Lipschitz constant of $\pi$ is a result of the identification of $\bbbr^m$ with $\ell^\infty_m$; $\pi$ considered as a map from $X$ to $\ell^\infty_m$ is $1$-Lipschitz.
\end{remark}

As an immediate consequence of (A), Theorem~\ref{6:metric TFU}, we get
\begin{corollary}
\label{IT1}
If $X$ is a metric space with $\H^{m+1}(X)=0$ (in particular if the Hausdorff dimension of $X$ is less than $m+1$) and $f\colon \bbbr^n \supset A \to X$ is a Lipschitz mapping, then $\rank \apmd(f,x) \leq m$ at a.e.\ $x \in A$.
\end{corollary}



Exactly as in Corollary~\ref{6:corr1}, we can exhaust the set $\{p\in A~:~\rank \apmd(f,p)\geq m\}$, up to a subset of $\H^n$-measure zero, by a countable (possibly finite) family of compact sets $\{K_i\}$, each with properties as in  Theorem~\ref{6:metric TFU}:

\begin{theorem}[Metric Implicit Function Theorem I]
\label{6:corr2}
Assume $A$ and $f$ are as in Theorem~\ref{6:metric TFU}.
Then there is a countable family $\{K_i\}$ of disjoint compact subsets of $A$ such that
$\H^n(A\setminus\bigcup_i K_i)=0$ and for each $i$ there exist a diffeomorphism $G_i:\bbbr^n\to\bbbr^n$ and a $\sqrt{m}$-Lipschitz $\pi_i:X\to\bbbr^m$ such that the properties (A)-(D) of Theorem~\ref{6:metric TFU} hold for $K_i$, $G_i$, $\pi_i$ in place of $K$, $G$ and $\pi$.
\end{theorem}

\begin{proof}[Proof of Theorem~\ref{6:metric TFU}]
Without loss of generality, we can assume that $X\subset \ell^\infty$, then  $f$ can be extended to a Lipschitz $\tilde{f}:\bbbr^n\to \ell^\infty$, with $\apmd (f,p)=\md(\tilde{f},p)$ for a.e.\ $p\in A$.
Since $\rank \md(\tilde{f},p)\geq m$ for a.e.\ $p\in A$, by Theorem~\ref{T13},  for such $p\in A$ there are indices $i_1(p)<i_2(p)<\cdots<i_m(p)$ such that the gradients of coordinate functions $\nabla \tilde{f}_{i_1(p)}(p),\,\nabla \tilde{f}_{i_2(p)}(p),\ldots, \nabla \tilde{f}_{i_m(p)}(p) $ are linearly independent.

Note that there is only a countable number of possible choices of sets of $m$ indices, so there is a particular choice of indices $i_1<i_2<\cdots<i_m$ such that the set
$$
A_1=\{p\in A~:~\nabla \tilde{f}_{i_1}(p),\,\nabla \tilde{f}_{i_2}(p),\ldots, \nabla \tilde{f}_{i_m}(p) \text{ are linearly independent} \}
$$
has positive measure.
We can also assume that $\tilde{f}(p)=f(p)$ for $p\in A_1$.
Then the projection $\pi:\ell^\infty\to\bbbr^m$, $\pi(x_1,x_2,\ldots)=(x_{i_1},x_{i_2},\ldots,x_{i_m})$ is $\sqrt{m}$-Lipschitz, if $\bbbr^m$ is equipped with its standard Euclidean norm (and $1$-Lipschitz, if we equip it with the $\ell^\infty_m$-norm).

Now, $f_1:=\pi\circ \tilde{f}:\bbbr^n\to\bbbr^m$ is Lipschitz with $\rank Df_1(p)=m$ at all $p\in A_1$, so we can apply Proposition \ref{6:prop1} to $f_1$, which gives a compact $K\subset A_1\subset A$ with $\H^n(K)>0$ and a $C^1$ diffeomorphism $G:\bbbr^n\to\bbbr^n$ such that for all $(x,y)\in G(K)$ we have
\begin{equation}\label{6:eq2}
(f_1\circ G^{-1})(x,y)=x.
\end{equation}
Then \eqref{6:eq2} yields, for $(x,y)\in G(K)$,
$$
x=(f_1\circ G^{-1})(x,y)=(\pi\circ \tilde{f}\circ G^{-1})(x,y)=(\pi\circ f\circ G^{-1})(x,y)=(\pi\circ F)(x,y)\,,
$$
where $F=f\circ G^{-1}$, which proves (B).

Next, note that since $\H^n(K)>0$, also $\H^n(G(K))>0$, and by Fubini's theorem,  the projection of $G(K)$ onto the first $m$ coordinates has positive $\H^m$ measure:
$$
\H^m((\pi\circ F)(G(K)))>0.
$$
However, 
the $\sqrt{m}$-Lipschitz map $\pi$ cannot increase measure more than by the factor of $(\sqrt{m})^m=m^{m/2}$, so we have
\begin{equation*}
\begin{split}
0<&\H^m((\pi\circ F)(G(K)))\leq m^{m/2}\H^m(F(G(K)))=m^{m/2} \H^m\big((f\circ G^{-1})(G(K))\big)\\=&m^{m/2}\H^m(f(K))\leq m^{m/2} \H^m(f(A)),
\end{split}
\end{equation*}
which proves (A).

To see (C), assume that $(x',y')\in F^{-1}(F(x,y))\cap G(K)$ for some $(x,y)\in G(K)$. Then $(x',y')\in G(K)$ and $F(x',y')=F(x,y)$, so obviously $x'=(\pi\circ F)(x',y')=(\pi\circ F)(x,y)=x$ and $(x',y')\in \{x\}\times \bbbr^{n-m}$.

Finally, to prove (D), denote the Lipschitz constant of $F$ on $G(K)$ by $\Lambda$ and take $(x,y), (x',y)\in G(K)$. We have
\begin{equation}\label{6:2}
m^{-1/2}|x-x'|\leq \|x-x'\|_\infty=\|\pi(F(x,y))-\pi(F(x',y))\|_\infty.
\end{equation}
Recall that $\pi:X\to (\bbbr^m,\|\cdot\|_\infty)$ is $1$-Lipschitz and $F:G(K)\to X$ is $\Lambda$-Lipschitz, so
\begin{equation}
\label{6:3}
\|\pi(F(x,y))-\pi(F(x',y))\|_\infty\leq d_X(F(x,y),F(x',y))\leq \Lambda |x-x'|.
\end{equation}
Combining \eqref{6:2} and \eqref{6:3} yields (D).
\end{proof}
As a corollary we obtain the following elegant result about bi-Lipschitz decomposition of a Lipschitz map. A quantitative version of this result will be used in the proof of the area formula; see the proof of Proposition~\ref{AFT3}.
\begin{corollary}[cf. {\cite[Lemma 4]{kir}}]
Assume $X$ is a metric space, $A\subset \bbbr^n$ a measurable set of positive $\H^n$ measure and $f:A\to X$ is Lipschitz. Then there is a countable family $\{K_i\}$ of pairwise disjoint compact subsets of $A$ such that $\H^n(f(A\setminus \bigcup_i K_i))=0$ and for each $i$ the map $f:K_i\to f(K_i)$ is bi-Lipschitz.
\end{corollary}
\begin{proof}
As before, we can assume that $X\subset \ell^\infty$.
Let $A_1=\{p\in A~~:~~\rank \apmd(f,p)=n\}$. Applying Theorem~\ref{6:corr2} to $f$, $m=n$ and $A_1$ in place of $A$ we get the countable family of compact, disjoint $K_i\subset A_1$, Lipschitz $\pi_i:X\to\bbbr^m$ and $C^1$-diffeomorphisms $G_i:\bbbr^n\to\bbbr^n$ such that (by condition (B) of Theorem~\ref{6:metric TFU})
$$
(\pi_i\circ f\circ G_i^{-1})(x)=x \quad \text{ for all }x\in G_i(K_i),
$$
or, equivalently, taking $x=G_i(y)$,
$$
(G_i^{-1}\circ \pi_i\circ f) (y)=y \quad \text{ for all  }y\in K_i.
$$
In other words, $G_i^{-1}\circ \pi_i|_{f(K_i)}$ is the inverse map to $f$ on $K_i$, $\pi_i$ is Lipschitz on $f(K_i)$ and $G_i^{-1}$ is Lipschitz on  $G_i(K_i)$, so $f$ is bi-Lipschitz on $K_i$.

Now, $A\setminus \bigcup_i K_i=(A\setminus A_1)\cup (A_1\setminus \bigcup_i K_i)$. The second summand, $A_1\setminus \bigcup_i K_i$, has $\H^n$ measure zero, so also $\H^n(f(A_1\setminus \bigcup_i K_i))=0$, and the image of the first summand, $f(A\setminus A_1)$, has $\H^n$ measure zero by the Metric Sard Theorem (Theorem~\ref{T17}). This completes the proof.
\end{proof}

\subsection{Implicit Function Theorem with $\H^m$-$\sigma$-finite target}
If additionally, $X$ is $\mathcal{H}^m$-$\sigma$-finite (i.e., if the $\H^m$-measure is $\sigma$-finite on $X$), the Metric Implicit Function Theorem~\ref{6:corr2} can be substantially improved leading to a version in line with the Euclidean one (Lemma~\ref{TFU}).



\begin{theorem}[Metric Implicit Function Theorem II]
\label{6:thm2}
Assume $X$ is a $\H^m$-$\sigma$-finite metric space and let  $A\subset \bbbr^n$ be measurable, with $\H^n(A)>0$. Let $f\colon A\to X$ be a Lipschitz map such that $\rank \apmd (f,p)=m$ for a.e.\ $p\in A$. Denote by $\pi:\bbbr^n=\bbbr^{m}\times\bbbr^{n-m}\to\bbbr^m$ the orthogonal projection, $\pi(x,y)=x$.
Then there is a countable family $\{K_i\}$ of pairwise disjoint compact subsets of $A$,  $\H^n(A\setminus \bigcup_i K_i)=0$, such that for each $i$ we have
\begin{itemize}
\item a diffeomorphism $G_i:\bbbr^n\to\bbbr^n=\bbbr^m\times\bbbr^{n-m}$
\item and a bi-Lipschitz $\phi_i:\pi(G_i(K_i))\to X$
\end{itemize}
such that
$$
F_i(x,y):=(f\circ G_i^{-1})(x,y)=
(\phi_i\circ\pi)(x,y)=
\phi_i(x)
\quad
\text{for all }(x,y)\in G_i(K_i).
$$
\end{theorem}
\begin{proof}
The key idea of the proof is the fact that by Theorem~\ref{6:corr2} we can find $K_i$, $G_i$ and $F_i$ such that, for all $y\in \bbbr^{n-m}$,  $F_i$ maps the `horizontal cuts' $C_i:=(\bbbr^m\times\{y\})\cap G_i(K_i)$ into $X$ in a bi-Lipschitz way. Since $X$ is $\H^m$-$\sigma$-finite, we can choose countably many such slices so that, for each $i$, the full inverse images of their images cover $G_i(K_i)$ up to an $\H^n$-null set. If $m<n$, the slices themselves have $\H^n$ measure zero. However, if $E_i:=F_i^{-1}(F_i(C_i))$, then $F_i(E_i)=F_i(C_i)$, and the sets $E_i$ may have positive $\H^n$ measure. After applying $G_i^{-1}$, we subdivide the resulting sets into compact, pairwise disjoint subsets that cover $A$ up to a set of measure zero.

The rest of the proof is a careful implementation of this idea.

Without loss of generality we may assume that $\mathcal{H}^m(X)<\infty$. Indeed, assume $X=\bigcup_{i=1}^\infty X_i$, where the sets $X_i$ are pairwise disjoint and have finite measure, $\mathcal{H}^m(X_i)<\infty$. By the Borel regularity of the Hausdorff measure we can also assume that $X_i$ are Borel. Then $A=\bigcup_{i=1}^\infty f^{-1}(X_i)$ is a decomposition of $A$ into disjoint and measurable pieces and we may apply the result to each piece $f^{-1}(X_i)$ separately, whenever $\H^n(f^{-1}(X_i))>0$. Note that $\apmd(f|_{f^{-1}(X_i)},p)=\apmd(f,p)$ at a.e. $p\in f^{-1}(X_i)$ (namely, at every density point $p$ of $f^{-1}(X_i)$ in which $f$ is approximately metrically differentiable).



By Theorem~\ref{6:corr2}, there exist countably many compact, disjoint $K_i'$ such that $\H^n(A\setminus \bigcup_i K_i')=0$ and for every $i$ there is $C^1$-diffeomorphism $G_i:\bbbr^n\to\bbbr^n=\bbbr^m\times \bbbr^{n-m}$  such that $F_i:=f\circ G_i^{-1}$ restricted to every `horizontal slice' $(\bbbr^m\times \{y\})\cap G_i(K_i')$ is $c_i$-bi-Lipschitz, with $c_i$ independent of $y$.

Subdividing the sets $K_i'$ into smaller pieces, if necessary, we may assume that each of the sets $\tilde{K}_i:=G_i(K_i')$ is contained in a unit cube.

We shall define the sets $A_{ij}\subset \tilde{K}_i$, $j=1,2,\ldots$ and $B_{ij}\subset F_i(\tilde{K}_i)$, $j=0,1,2,\ldots$ by induction.

Let $A_{i1}=\tilde{K}_i$ and $B_{i0}=\varnothing$. Suppose now that the sets $A_{ij}$ and $B_{ij-1}$ have been defined.







Since $A_{ij}\subset\tilde{K}_i$ is contained in a unit cube,
by Fubini's theorem, one can find $y_{ij}$ such that $\H^m((\bbbr^m\times\{y_{ij}\})\cap A_{ij})>\frac{1}{2}\H^n(A_{ij})$; pick such $y_{ij}$ and denote $B_{ij}=F_i((\bbbr^m\times\{y_{ij}\})\cap A_{ij})$. Then $\H^m(B_{ij})\geq 2^{-1}c_i^{-m}\H^n(A_{ij})$, because $F_i$ is $c_i$-bi-Lipschitz on $(\bbbr^m\times\{ y_{ij}\})\cap A_{ij}$.
Finally, \begin{itemize}
    \item if $\H^n(A_{ij}\setminus F_i^{-1}(B_{ij}))>0$, we set $A_{ij+1}=A_{ij}\setminus F_i^{-1}(B_{ij})$,
    \item if $\H^n(A_{ij}\setminus F_i^{-1}(B_{ij}))=0$, we terminate the inductive construction, setting

  $$
  A_{ik}=A_{ij}\setminus F_i^{-1}(B_{ij}), \qquad B_{ik}=\varnothing \qquad \text{ for all }k>j. $$
\end{itemize}

Note that all the sets $A_{ij}$ and $B_{ij}$ arising in the construction are Borel. Indeed, this follows by induction. 

The set $A_{i1}=\tilde K_i$ is compact. Suppose that $A_{ij}$ is Borel and put $S_{ij}:=(\bbbr^m\times\{y_{ij}\})\cap A_{ij}$, $T_{ij}:=(\bbbr^m\times\{y_{ij}\})\cap \tilde K_i$. Then $S_{ij}$ is a Borel subset of the compact set $T_{ij}$. Since $F_i$ is bi-Lipschitz on the horizontal slice $T_{ij}$, the restriction $F_i|_{T_{ij}}$ is a homeomorphism of $T_{ij}$ onto the compact set $F_i(T_{ij})\subset X$. Therefore $B_{ij}=F_i(S_{ij})$
is Borel in $F_i(T_{ij})$, and hence Borel in $X$, so
$A_{i,j+1}=A_{ij}\setminus F_i^{-1}(B_{ij})$
is Borel as well. This proves, in particular, the measurability of all the sets $A_{ij}$ and $B_{ij}$ and justifies the use of Fubini's theorem.

From the construction it follows that
$$
B_{ij}\subset F_i(A_{ij}),
\quad
B_{ij}\cap F_i(A_{ij+1})=\varnothing
\quad
\text{and}
\quad
A_{ik}\subset A_{ij+1}
\text{ for } k>j.
$$
This implies that the sets $B_{ij}$ are pairwise disjoint. Indeed, for $k>j$ we have $B_{ik}\subset F_i(A_{ik})$ and hence
$$
B_{ij}\cap B_{ik}\subset B_{ij}\cap F_i(A_{ik})\subset B_{ij}\cap F_i(A_{ij+1})=\varnothing.
$$
Clearly, the sequence $\H^n(A_{ij})$ is non-increasing. Also,
$\lim_{j\to\infty}\H^n(A_{ij})=0$. Indeed, assume that for some $\eps>0$ and all $j$ we have $\H^n(A_{ij})>\eps$, then for all $j$ we have
$$
\H^m(B_{ij})\geq 2^{-1}c_i^{-m}\H^n(A_{ij})>2^{-1}c_i^{-m}\eps>0.
$$
This leads to a contradiction, because $B_{ij}$ are disjoint subsets of $X$, which has finite $\H^m$ measure.

Note that
\begin{equation}
\label{Ieq1}
G_i(K_i')=\tilde{K}_i=Z_i\cup\bigcup_{j=1}^\infty(A_{ij}\cap F_i^{-1}(B_{ij})),
\end{equation}
where $Z_i:=\bigcap_{j=1}^\infty A_{ij}$, thus  $\H^n(Z_i)=0$ and the sets $A_{ij}\cap F_i^{-1}(B_{ij})$ are pairwise disjoint.

Indeed, since $(A_{ij})_j$ is a decreasing sequence of sets with the intersection $Z_i$ of measure zero,
and $A_{i1}=\tilde{K}_i$, we have
$$
\tilde{K}_i=Z_i\cup\bigcup_{j=1}^\infty(A_{ij}\setminus A_{ij+1}),
\quad
\H^n(Z_i)=0,
$$
and it remains to observe that
$$
A_{ij}\setminus A_{ij+1}=A_{ij}\setminus(A_{ij}\setminus F_i^{-1}(B_{ij}))=A_{ij}\cap F_i^{-1}(B_{ij}).
$$
We may assume that the sets $A_{ij}\cap F_i^{-1}(B_{ij})$ have positive $\H^n$ measure, as otherwise we can omit those sets $A_{ij}\cap F_i^{-1}(B_{ij})$ which have zero $\H^n$ measure, enlarging the null set $Z_i$ by their union.

Observe that
$$
\text{If } (x,y)\in A_{ij}\cap F_i^{-1}(B_{ij}), \text{ then } (x,y_{ij}) \in A_{ij}\cap F_i^{-1}(B_{ij}) \text{ and }
F_i(x,y)=F_i(x,y_{ij}).
$$
Indeed,
$$
F_i(x,y)\in F_i(A_{ij})\cap B_{ij}=F_i\big((\bbbr^m\times\{ y_{ij}\})\cap A_{ij}\big),
$$
so $F_i(x,y)=F_i(x',y_{ij})$ for some $(x',y_{ij})\in A_{ij}$, but property (C) implies that $x'=x$ so
$F_i(x,y)=F_i(x,y_{ij})$.

Thus, if we define
$$
\phi_{ij}:\pi(A_{ij}\cap F_i^{-1}(B_{ij}))\to X
\quad
\text{by}
\quad
\phi_{ij}(x)=F_i(x,y_{ij}),
$$
then $\phi_{ij}$ is $c_i$-bi-Lipschitz and
$$
F_i(x,y)=\phi_{ij}(x)
\quad
\text{for }
\quad
(x,y)\in A_{ij}\cap F_i^{-1}(B_{ij}).
$$
Note that by \eqref{Ieq1}, the sets $A_{ij}\cap F_i^{-1}(B_{ij})$ decompose $G_i(K_i')$ into disjoint sets (up to a set of measure zero).

Finally, each $G_i^{-1}(A_{ij}\cap F_i^{-1}(B_{ij}))\subset K_i'$ is, up to a $\H^n$-null set, a countable union of disjoint compact sets $K_{ijk}$, which, after relabeling, give the desired compact sets $K_i$.
\end{proof}

\section{Seminorms and Jacobians}
\label{SJ}
Recall that $\sigma:\bbbr^n\to [0,\infty)$ is a \emph{seminorm} if $\sigma(\lambda v)=|\lambda|\sigma(v)$ and $\sigma(v+w)\leq \sigma(v)+\sigma(w)$ for all $\lambda\in\bbbr$ and $v,w\in\bbbr^n$.

Every seminorm $\sigma$ on $\bbbr^n$ is continuous and it is uniquely defined by its restriction $\sigma|_{\Sph^{n-1}}$ to the unit sphere $\Sph^{n-1}=\{x\in\bbbr^n~:~|x|=1\}$. Therefore, we can consider the space of all seminorms on $\bbbr^n$ as the subspace of the space $C(\Sph^{n-1})$ of continuous functions on $\Sph^{n-1}$. The latter space is separable by the Stone-Weierstrass theorem, which yields the following observation.

\begin{lemma}
\label{lem:semi sep}
The space of all seminorms on $\bbbr^n$, endowed with the metric
$$
d_\infty(\sigma,\tau):=\sup_{|v|=1} |\sigma(v)-\tau(v)|,
$$
is separable.
\end{lemma}
Recall the definition of the Jacobian of the seminorm (see, Definition~\ref{INd2}).
\begin{definition}
Let $\sigma$ be a seminorm on $\bbbr^n$.
Let $1\leq m\leq n$. If $\rank\sigma<m$, we set $|J_m(\sigma)|=0$. If $\rank\sigma=m$, we define
\begin{equation}
\label{SJeq1}
|J_m(\sigma)|=\frac{\omega_m}{\H^m(\{v\in N_\sigma^\perp:\, \sigma(v)\leq 1\})}\, .
\end{equation}
\end{definition}

\begin{proposition}
\label{SJT7}
If $\sigma$ is a seminorm on $\bbbr^n$ satisfying $\operatorname{rank}\sigma\leq m$, then
\begin{equation}
\label{SJeq4}
|J_m(\sigma)|
=
\sup_{V\in \mathrm{Gr}(m,n)} |J_m(\sigma|_V)|,
\end{equation}
where $\mathrm{Gr}(m,n)$ is the Grassmannian of $m$-dimensional linear
subspaces of $\mathbb{R}^n$.
\end{proposition}
\begin{proof}
If $\operatorname{rank}\sigma<m$, then both sides of \eqref{SJeq4} are
equal to zero, because the restriction of $\sigma$ to every
$m$-dimensional subspace is degenerate.

Assume now that $\rank\sigma=m$. Then $\sigma|_{N_\sigma^\perp}$ is a norm and
by the definition of the Jacobian
$$
|J_m(\sigma)|=\big|J_m\big(\sigma|_{N_\sigma^\perp}\big)\big|.
$$
Hence the left hand side of \eqref{SJeq4} is less than or equal to the right hand side and it remains to show the opposite inequality.

Let  $V\in \mathrm{Gr}(m,n)$. If $V\cap N_\sigma\neq\{ 0\}$, then $\sigma|_V$ is degenerate and hence
$$
|J_m(\sigma|_V)|=0\leq |J_m(\sigma)|.
$$
If $V\cap N_\sigma=\{ 0\}$,
then the orthogonal projection $\pi:\bbbr^n\to N_\sigma^\perp$ restricted to $V$ is an isomorphism and by Lemma~\ref{T16}, $\sigma(v)=\sigma(\pi(v))$ for every $v\in V$.
Hence
\[
\pi\bigl(\{v\in V:\sigma(v)\leq 1\}\bigr)
=
\pi\bigl(\{v\in V:\sigma(\pi (v))\leq 1\}\bigr)
=
\{w\in N_\sigma^\perp:\sigma(w)\leq 1\}.
\]
Since the orthogonal projection
$\pi|_V:V\to N_\sigma^\perp$ is $1$-Lipschitz we have
\[
\mathcal H^m\bigl(\{w\in N_\sigma^\perp:\sigma(w)\leq 1\}\bigr)
\leq
\mathcal H^m\bigl(\{v\in V:\sigma(v)\leq 1\}\bigr).
\]
Therefore, the definition of the Jacobian yields
\[
|J_m(\sigma|_V)|
=
\frac{\omega_m}
{\mathcal H^m\bigl(\{v\in V:\sigma(v)\leq 1\}\bigr)}
\leq
\frac{\omega_m}
{\H^m\bigl(\{w\in N_\sigma^\perp:\sigma(w)\leq 1\}\bigr)}
=
|J_m(\sigma)|.
\]
The proof is complete.
\end{proof}
If $\sigma$ is a norm on $\bbbr^n$, we write $\bbbr^n_\sigma$ or $(\bbbr^n,\sigma)$ for $\bbbr^n$ equipped with the norm metric
$d(x,y)=\sigma(x-y)$.
Similarly, $B^n_\sigma(x,r)$ and $\bar{B}^n_\sigma(x,r)$ denote the open and closed balls in $\bbbr^n_\sigma$, respectively. For any set $A\subset \bbbr^n_\sigma$, we define
$\diam_\sigma(A)=\sup\{\sigma(x-y):x,y\in A\}$.
We write $\H^s_\sigma$ for the Hausdorff measure on the metric space $\bbbr^n_\sigma$, while $\H^s$ denotes the Hausdorff measure with respect to the Euclidean metric.

\begin{remark}
\label{SJR1}
As a sign of warning, note that
if $L:\bbbr^n_\sigma\to \bbbr^n_\tau$ is an isometry of normed spaces, then in general $|J_n(\sigma)|\neq |J_n(\tau)|$. Let $\sigma$ be the Euclidean norm and $\tau=\tfrac{1}{2}\sigma$. Then $L:\bbbr^n_\sigma\to\bbbr^n_\tau$, $Lx=2x$ is an isometry, but $|J_n(\sigma)|=1$, while $|J_n(\tau)|=2^{-n}$.
\end{remark}


The next result will be very important (see \cite[Lemma~6]{kir}).
\begin{proposition}
\label{SJT1}
If $\sigma$ is a norm on $\bbbr^n$, then for any set $A\subset\bbbr^n$,
\begin{equation}
\label{eq:jac2}
\H^n_\sigma(A) = |J_n(\sigma)| \mathcal{H}^n(A),
\quad
\text{and hence}
\quad
|J_n(\sigma)|=\H^n_\sigma([0,1]^n).
\end{equation}
\end{proposition}
\begin{remark}
In the case of a norm, \eqref{SJeq1} yields
\begin{equation}
\label{SJeq3}
|J_n(\sigma)|=\frac{\omega_n}{\H^n(\{v\in\bbbr^n:\, \sigma(v)\le 1\})}.
\end{equation}
This formula seems more convenient for computing or estimating the Jacobian than
\[
|J_n(\sigma)|=\H^n_\sigma([0,1]^n),
\]
because it involves the Lebesgue measure, whereas the Hausdorff measure $\H^n_\sigma$ is generally more difficult to handle. Nevertheless, \eqref{eq:jac2} will play an important role in the proof of the area formula, Theorem~\ref{AFT1}.
\end{remark}

\begin{proof}[Proof of Proposition~\ref{SJT1}]
According to Corollary~\ref{AT12}, there is a constant $c(\sigma)\in (0,\infty)$ such that $\H^n_\sigma(A)=c(\sigma)\H^n(A)$ for all sets $A\subset\bbbr^n$ and it remains to show that $c(\sigma)=|J_n(\sigma)|$. This, however, follows immediately from Proposition~\ref{AT5} since
for the unit ball $\Bbbb^n_\sigma$ we have
$$
|J_n(\sigma)|=\frac{\omega_n}{\H^n(\Bbbb_\sigma^n)}=\frac{\H^n_\sigma(\Bbbb^n_\sigma)}{\H^n(\Bbbb^n_\sigma)}=c(\sigma).
$$
The proof is complete.
\end{proof}

\begin{lemma}
\label{lem: J cont}
The map $\sigma \to |J_n(\sigma)|$ is continuous on the space of seminorms on $\bbbr^n$ equipped with the distance $d_\infty(\sigma,\sigma')=\sup_{|v|=1}|\sigma(v)-\sigma'(v)|$.
\end{lemma}
\begin{proof}
Let $\sigma$ be a seminorm on $\bbbr^n$; we claim that $|J_n(\cdot)|$ is continuous at $\sigma$.

Let us first assume that $\rank\sigma=n$, i.e.,  $\sigma$ is a norm. Let $a=\inf_{|v|=1}\sigma(v)$. Then, for any $0<\eps<a$ and any seminorm $\tau$ such that $d_\infty(\tau,\sigma)<\eps$, we have $\sigma(v)=|v|\sigma(v/|v|)\geq a|v|$ and
$$
|\sigma(v)-\tau(v)|=|v|\big|\sigma(v/|v|)-\tau(v/|v|)\big|\leq \eps |v|\leq \frac{\eps}{a}\sigma(v),
$$
so
$$
\left(1-\frac{\eps}{a}\right)\sigma(v)\leq \tau(v)\leq \left(1+\frac{\eps}{a}\right)\sigma(v)
$$
(in particular, $\tau$ is also a norm). Thus, for any $A\subset\bbbr^n$,
$$
\left(1-\frac{\eps}{a}\right)\diam_\sigma(A)\leq \diam_\tau(A)\leq \left(1+\frac{\eps}{a}\right)\diam_\sigma(A),
$$
and therefore, for any $A\subset\bbbr^n$,
$$
\left(1-\frac{\eps}{a}\right)^n\H^n_\sigma(A)\leq \H^n_\tau(A)\leq \left(1+\frac{\eps}{a}\right)^n\H^n_\sigma(A).
$$
Setting $A=[0,1]^n$ we get, by Proposition~\ref{SJT1}, that
$$
\left(1-\frac{\eps}{a}\right)^n |J_n(\sigma)|\leq |J_n(\tau)|\leq \left(1+\frac{\eps}{a}\right)^n |J_n(\sigma)|,
$$
which proves the claim in the case when $\sigma$ is a norm.

Assume now that $\rank \sigma<n$, so there is $e\in\bbbr^n$, $|e|=1$, such that $\sigma(e)=0$. Denote $C:=d_\infty(\sigma,0)=\sup_{|v|=1}\sigma(v)$.

Let $\tau_j$ be a sequence of seminorms such that $d_\infty(\tau_j,\sigma)\xrightarrow{j\to\infty}0$, $d_\infty(\tau_j,\sigma)<1$.
Our aim is to show that $|J_n(\tau_j)|\xrightarrow{j\to\infty}|J_n(\sigma)|=0$; since $|J_n(\tau_j)|=0$ whenever $\tau_j$ is not a norm, it suffices to consider the case when all $\tau_j$ are norms.

Let $\eps_j:=\tau_j(e)$; obviously $\eps_j\xrightarrow{j\to\infty}\sigma(e)=0$ and since $\tau_j$ are norms, $\eps_j>0$.

Also,
$$
\sup_{|v|=1}\tau_j(v)=d_\infty(\tau_j,0)\leq d_\infty(\tau_j,\sigma)+d_\infty(\sigma,0)\leq 1+C.
$$
Assume now that $t\in\bbbr$, $|t|\leq (2\eps_j)^{-1}$ and $v\in \bbbr^n$, $v\perp e$, $|v|\leq (2(C+1))^{-1}$. Then
$$
\tau_j(te+v)\leq |t|\tau_j(e)+|v|\tau_j(v/|v|)\leq (2\eps_j)^{-1}\eps_j+(2(C+1))^{-1} (C+1)=1.
$$
Therefore the unit ball $\Bbbb^n_{\tau_j}$ contains the cylinder
$$
E_j=\{te+v~~:~~|t|\leq (2\eps_j)^{-1},\,v\perp e,\, |v|\leq (2(C+1))^{-1}\}
$$
and hence $\H^n(\Bbbb^n_{\tau_j})\geq \H^n(E_j)=\omega_{n-1}(2(C+1))^{-(n-1)} \eps_j^{-1}\xrightarrow{j\to\infty}\infty$.
However, since $\tau_j$ is a norm, \eqref{SJeq3} yields $|J_n(\tau_j)|=\omega_n/\H^n(\Bbbb^n_{\tau_j})$, thus
$|J_n(\tau_j)|\xrightarrow{j\to\infty}0$, as required. This concludes the proof.
\end{proof}

\begin{corollary}
\label{SJT6}
Let $1\leq m\leq n$. The map $\sigma\mapsto |J_m(\sigma)|$, defined on the subspace of
seminorms satisfying $\operatorname{rank}\sigma\leq m$, is continuous.
\end{corollary}
In the proof we will need a simple lemma.
\begin{lemma}
\label{SJT8}
Let $X$ and $K$ be metric spaces, with $K$ compact, and let
$\Phi\colon X\times K\to \mathbb R$ be continuous. Define
$F\colon X\to\mathbb R$ by
\[
F(x)=\sup_{y\in K}\Phi(x,y).
\]
Then the supremum is attained for every $x\in X$, and the function $F$ is continuous.
\end{lemma}
\begin{proof}
Since $K$ is compact and $y\mapsto \Phi(x,y)$ is continuous, the supremum is
attained for each fixed $x\in X$.
We prove continuity.
Suppose to the contrary that $F$ is not continuous at $x\in X$. Then, there is $\eps>0$ and a sequence $x_j\to x$ such that
\begin{equation}
\label{ziuta}
|F(x_{j})-F(x)|\geq \eps.
\end{equation}
Choose $y_j$ such that $F(x_j)=\Phi(x_j,y_j)$. Since $K$ is compact, a subsequence $y_{j_k}\to y$ is convergent. Hence
$$
F(x_{j_k})=\Phi(x_{j_k},y_{j_k})\to\Phi(x,y)\leq F(x).
$$
On the other hand, choose $y_o\in K$ such that $F(x)=\Phi(x,y_o)$. Then
$$
F(x_{j_k})\geq\Phi(x_{j_k},y_o)\to\Phi(x,y_o)=F(x).
$$
Hence $F(x_{j_k})\to F(x)$, but that contradicts \eqref{ziuta}.
\end{proof}

\begin{proof}[Proof of Corollary~\ref{SJT6}]
In view of Proposition~\ref{SJT7}, it suffices to prove continuity of the right hand side of \eqref{SJeq4} with respect to $\sigma$ satisfying $\rank\sigma\leq m$. Define
\[
\Phi(\sigma,V)=|J_m(\sigma|_V)|,
\qquad V\in \mathrm{Gr}(m,n).
\]
Note that $\Phi(\sigma,V)$ is defined for all seminorms and hence the right hand side of \eqref{SJeq4} is defined for all seminorms $\sigma$, but the equality in \eqref{SJeq4} is true for $\sigma$ with $\rank\sigma\leq m$.

We claim that $(\sigma,V)\mapsto\Phi(\sigma,V)$ is continuous as a function defined on the product of the space of all seminorms and the Grassmannian.

Let $\sigma_j\to \sigma$ with respect to $d_\infty$, and let
$V_j\to V$ in $\mathrm{Gr}(m,n)$. Choose an orthonormal basis
$e_1,\ldots,e_m$ of $V$
and choose orthonormal bases
$e_1^j,\ldots,e_m^j$ of $V_j$ such that
$e_i^j\to e_i$ for $i=1,\ldots,m$.
Define seminorms $\widetilde\sigma_j$ and $\widetilde\sigma$ on
$\mathbb{R}^m$ by
\[
\widetilde\sigma_j(a)
=
\sigma_j\Big(\sum_{i=1}^m a_i e_i^j\Big),
\qquad
\widetilde\sigma(a)
=
\sigma\Big(\sum_{i=1}^m a_i e_i\Big),
\qquad
\text{where } a=(a_1,\ldots,a_m)\in\mathbb{R}^m.
\]
Then
$\widetilde\sigma_j\to\widetilde\sigma$
uniformly on the unit sphere in $\mathbb{R}^m$. Hence, by Lemma~\ref{lem: J cont} applied in dimension $m$,
$|J_m(\widetilde\sigma_j)|\to |J_m(\widetilde\sigma)|$.
Since
\[
|J_m(\widetilde\sigma_j)|=|J_m(\sigma_j|_{V_j})|,
\qquad
|J_m(\widetilde\sigma)|=|J_m(\sigma|_V)|,
\]
we obtain $\Phi(\sigma_j,V_j)\to \Phi(\sigma,V)$. Thus $\Phi$ is continuous.

Recall that, by Proposition~\ref{SJT7},
$$
|J_m(\sigma)|=\sup_{V\in \mathrm{Gr}(m,n)} \Phi(\sigma,V).
$$
Since the Grassmannian $\mathrm{Gr}(m,n)$ is compact, continuity of $\sigma\mapsto |J_m(\sigma)|$
on the space of seminorms $\sigma$ on $\mathbb{R}^n$ with $\operatorname{rank}\sigma\leq m$ follows from Lemma~\ref{SJT8}.
\end{proof}

Important examples of seminorms on $\bbbr^n$ arise as pull-backs of norms by linear maps.

\begin{lemma}
\label{SJT5}
Assume $n\leq m$, $L:\bbbr^n\to\bbbr^m$ is an injective linear map and $\sigma$ is a norm on $\bbbr^m$. Then $\sigma\circ L$ is a norm on $\bbbr^n$ and
\begin{equation}
\label{eq:pullsemi}
|J_n(\sigma\circ L)|=\sqrt{\det(L^T L)}\,\,|J_n(\sigma|_{L(\bbbr^n)})|.
\end{equation}
\end{lemma}

\begin{proof}
Checking that $\sigma\circ L$ is a seminorm is straightforward; also, $\rank (\sigma\circ L)=n$ and thus $\sigma\circ L$ is a norm on $\bbbr^n$.

Now, note that $L:(\bbbr^n,\sigma\circ L)\to (L(\bbbr^n),\sigma|_{L(\bbbr^n)})\subset (\bbbr^m,\sigma)$ is a linear isometry, thus for any $A\subset \bbbr^n$ we have $\H^n_{\sigma\circ L}(A)=\H^n_{\sigma|_{L(\bbbr^n)}}(L(A))$. Applying it to $A=[0,1]^n$ we get, by Proposition~\ref{SJT1},
and Lemma~\ref{AT8}
\[
\begin{split}
&
|J_n(\sigma\circ L)|
=
\H^n_{\sigma\circ L}([0,1]^n)=\H^n_{\sigma|_{L(\bbbr^n)}}(L([0,1]^n))\\
&=
|J_n(\sigma|_{L(\bbbr^n)})|\H^n(L([0,1]^n))=|J_n(\sigma|_{L(\bbbr^n)})|\sqrt{\det L^TL}.
\end{split}
\]
The proof is complete.
%
\end{proof}

\begin{remark}
In the above lemma, if $L$ is not injective, then one immediately checks that $\sigma\circ L$ is still a seminorm in $\bbbr^n$.  In this case, however, $\ker L$ is non-trivial. Since $\sigma$ is a norm, $\ker L=\ker (\sigma\circ L)$, so $\rank (\sigma\circ L)=\rank L<n$, hence $|J_n(\sigma\circ L)|=0$.

At the same time $\ker L= \ker (L^T L)$, so $L^T L$ is not invertible and thus $\det(L^T L)=0$. However, $\dim L(\bbbr^n)<n$, so $|J_n(\sigma|_{L(\bbbr^n)})|$ is not defined.
\end{remark}

\begin{lemma}
\label{SJT4}
Assume $n\geq m$, $L:\bbbr^n\to\bbbr^m$ is linear and $\sigma$ is a norm on $\bbbr^m$. Then $\sigma\circ L$ is a seminorm on $\bbbr^n$ with $\rank (\sigma\circ L)\leq m$ and
$$
|J_m(\sigma\circ L)|=\sqrt{\det(LL^T)}|J_m(\sigma)|.
$$
\end{lemma}
\begin{remark}
Immediately from the definition of the $m$-Jacobian, for any seminorm $\tau$ of rank $m$ on $\bbbr^n$ we have $|J_m(\tau)|=|J_m(\tau|_{N_\tau^\perp})|$ and $\tau$ is a norm on  $N_\tau^\perp$.
\end{remark}
\begin{proof}
If $L$ is not surjective, $\rank (\sigma\circ L)=\rank L<m$, so $|J_m(\sigma\circ L)|=0$, but also $LL^T:\bbbr^m\to\bbbr^m$ is not surjective, so $\det (LL^T)=0$.

Assume thus that $\rank L=m$ and denote $M=(\ker L)^\perp$. This is an $m$-dimensional linear subspace of $\bbbr^n$, so we have a linear isometry $U:\bbbr^m\to M$.

Then $\sigma\circ L\circ U$ is a norm on $\bbbr^m$ and
$$
|J_m(\sigma\circ L)|=|J_m(\sigma\circ L|_{M})|=|J_m(\sigma\circ L\circ U)|=\H^m_{\sigma\circ L\circ U}([0,1]^m).
$$
Also, $L\circ U:(\bbbr^m,\sigma\circ L\circ U)\to (\bbbr^m, \sigma)$ is an isometry, so
\begin{equation*}
\begin{split}
|J_m(\sigma\circ L)|
&=
\H^m_{\sigma\circ L\circ U}([0,1]^m)=\H^m_\sigma((L\circ U)([0,1]^m))=|J_m(\sigma)|\H^m((L\circ U)([0,1]^m))\\
&=|J_m(\sigma)||\det(LU)|=|J_m(\sigma)|\sqrt{\det\left((LU)(LU)^T\right)}=|J_m(\sigma)|\sqrt{\det LL^T}.
\end{split}
\end{equation*}
\end{proof}
\begin{lemma}
\label{lem:semi11}
Assume $N\geq m$, $L:\bbbr^N\to\bbbr^N$ is a linear isomorphism and $\sigma$ is a seminorm on $\bbbr^N$, $\rank \sigma\leq m$. Then $\sigma\circ L$ is a seminorm on $\bbbr^N$ and
\begin{equation}
\label{eq:lab}
|J_m(\sigma)||\det L|=|J_m(\sigma\circ L)||\det (L|_{N_{\sigma\circ L}})|.
\end{equation}
\end{lemma}
\begin{remark}
Here $|\det (L|_{N_{\sigma\circ L}})|$ is the determinant of the linear isomorphism $L|_{N_{\sigma\circ L}}:N_{\sigma\circ L}\to N_\sigma$. It has to be taken with the absolute value due to the lack of canonical orientations in the subspaces.
\end{remark}
\begin{proof}
Since $\rank (\sigma\circ L)=\rank\sigma$, if $\rank\sigma<m$, both sides of \eqref{eq:lab} are equal zero. Thus we may assume that $\rank\sigma=m$.
Denote by $\pi:\bbbr^N\to\bbbr^N$ the orthogonal projection onto $N_\sigma^\perp$. Observe that $N_{\sigma\circ L}=L^{-1}(N_\sigma)$.

Set $T=\pi\circ L:N_{\sigma\circ L}^\perp \to N_\sigma^\perp$.
Note that $T$ is an isometry between $(N_{\sigma\circ L}^\perp, \sigma\circ L)$ and $(N_\sigma^\perp,\sigma)$. Indeed,
$\dim N_{\sigma\circ L}^\perp=\dim N_\sigma^\perp=m$ and
for $v\in N_{\sigma\circ L}^\perp$, Lemma~\ref{T16} yields
$$
\sigma(Tv)=\sigma ((\pi\circ L)v)=\sigma(Lv)=(\sigma\circ L)(v).
$$
In particular, if we denote the unit balls in the normed spaces by
$$
B_{\sigma\circ L}=\{v\in N_{\sigma\circ L}^\perp~:~(\sigma\circ L)(v)\leq 1\}
\quad
\text{and}
\quad
B_\sigma=\{v\in N_\sigma^\perp~:~\sigma(v)\leq 1\},
$$
then $T(B_{\sigma\circ L})=B_\sigma$, thus
$\H^m(B_\sigma)=|\det T|\H^m(B_{\sigma\circ L})$ and hence
$$
|J_m(\sigma\circ L)|=\frac{\omega_m}{\H^m(B_{\sigma\circ L})}=|\det T|\frac{\omega_m}{\H^m(B_\sigma)}=|\det T| |J_m(\sigma)|
$$
Therefore, it suffices to show that
\begin{equation}
\label{xyw3}
|\det L|=|\det T|\, |\det (L|_{N_{\sigma\circ L}})|.
\end{equation}
Introduce in $\bbbr^N$ two orthonormal bases: one adapted to the splitting $\bbbr^N=N_{\sigma\circ L}^\perp\oplus N_{\sigma\circ L}$ (in the domain of $L$), the other to $\bbbr^N=N_{\sigma}^\perp\oplus N_{\sigma}$ (in the target). In these bases the matrix of $L$ takes the square block form:
$$
L=
\left[
\begin{array}{ c | c}
  ~~T ~~\rule[-1.5ex]{0pt}{4.5ex} & 0 \\
\hline
*\rule[-1.5ex]{0pt}{4.5ex} & L|_{N_{\sigma\circ L}}
\\
\end{array}
\right],
$$
and hence \eqref{xyw3} follows.
\end{proof}

\section{Area Formula}
\label{AF}

This section is centered around the Metric Area Formula of Kirchheim \cite[Theorem~7 and Corollary~8]{kir} (cf. Theorem~\ref{INT4}).

\begin{theorem}[Metric area formula]
\label{AFT1}
Suppose that $A\subset \bbbr^n$ is measurable, $X$ is a metric space,
and $f:A\to X$ is Lipschitz. Then for every measurable function
$g:A\to [0,\infty)$ we have
$$
\int_A g(x)\, |J_n(\apmd(f,x))|\, d\H^n(x)
=
\int_X
\Big(
\sum_{x\in f^{-1}(y)\cap A} g(x)
\Big)
\,d\H^n(y).
$$
\end{theorem}
\begin{remark}
The function $y\mapsto \sum_{x\in f^{-1}(y)\cap A}g(x)$, i.e., the sum of values of $g$ over the fiber $f^{-1}(y)\cap A$, is by definition equal to zero if $y\not \in f(A)$. It may also be equal to $+\infty$ for some $y$.
\end{remark}

\begin{remark}
Since the image $f(A)\subset X$ is separable, we may assume, after an
isometric embedding, that $f(A)\subset \ell^\infty$. Then $f$ admits a
Lipschitz extension
\[
F:\bbbr^n\to \ell^\infty
\qquad
\text{and}
\qquad
\md(F,x)=\apmd(f,x) \text{ for a.e.\ } x\in A.
\]
Thus it suffices to prove the following proposition.
\end{remark}

\begin{proposition}
\label{AFT2}
Suppose that $f:\bbbr^n\to \ell^\infty$ is Lipschitz. If
$g:A\to [0,\infty)$ is a measurable function defined on a measurable set
$A\subset \bbbr^n$, then
\begin{equation}
\label{AFeq2}
\int_A g(x)\, |J_n(\md(f,x))|\, d\H^n(x)
=
\int_{\ell^\infty}
\Big(
\sum_{x\in f^{-1}(y)\cap A} g(x)
\Big)
\,d\H^n(y).
\end{equation}
\end{proposition}

For a map $f:\bbbr^n\to X$ and a set $A\subset \bbbr^n$, we define the
multiplicity function, also called the {\em Banach indicatrix}, by
\[
N(f,A,y):=\card\big(f^{-1}(y)\cap A\big)
        =\H^0\big(f^{-1}(y)\cap A\big).
\]
In the particular case $g=\chi_A$, Proposition~\ref{AFT2}
reduces to the following statement.

\begin{proposition}
\label{AFT3}
If $f:\bbbr^n\to \ell^\infty$ is Lipschitz, then for every measurable set
$A\subset \bbbr^n$ we have
\begin{equation}
\label{AFeq3}
\int_A |J_n(\md(f,x))|\, d\H^n(x)
=
\int_{\ell^\infty} N(f,A,y)\, d\H^n(y).
\end{equation}
\end{proposition}

One can derive Proposition~\ref{AFT2} and hence Theorem~\ref{AFT1} from Proposition \ref{AFT3}, as well, using the standard technique of approximating measurable $g$ by simple functions. Before we recall this argument, let us first address the question of measurability of integrands in \eqref{AFeq2}.

Since $g$ is assumed to be $\H^n$-measurable, the measurability of the left side integrand of \eqref{AFeq2} follows from the following lemma. For later use in the proof of the co-area formula, we state the result in greater generality than is needed here; at this point, for the area formula, only the case $m=n$ will be used.
\begin{lemma}\label{AFT4}
Assume $A\subset\bbbr^n$ is $\H^n$ measurable, $X$ is a metric space and let $f:A\to X$ be Lipschitz and such that for some $m\in\bbbn$, $1\leq m\leq n$,
$$
\rank \apmd(f,x)\leq m\quad \text{for $\H^n$-a.e. }x\in A.
$$
Then $x\mapsto |J_m(\apmd(f,x))|$ is $\H^n$-measurable on $A$.
\end{lemma}

\begin{proof}
Embed the separable space $f(A)$ isometrically into $\ell^\infty$ and let
$$
F=(F_1,F_2,\ldots):\bbbr^n\to\ell^\infty
$$
be a Lipschitz extension of $f$.

By Remark \ref{R2}, for $\H^n$-a.e. $x\in A$ we have $\apmd(f,x)=\md(F,x)$. Also, the coordinate functions $F_1,F_2,\ldots$ are differentiable at $\H^n$-a.e. point of $A$ and by Proposition \ref{MDT2} for \mbox{$\H^n$-a.e.} $x\in\bbbr^n$ and every $v$ we have $\md(F,x)(v)=\sup_i|\nabla F_i(x)\cdot v|$. Let $E\subset A$ be a subset of $A$ of full $\H^n$ measure in which all of these a.e. conditions hold, i.e., $E$ consists of these $x\in A$ in which
\begin{itemize}
    \item $\rank \md (F,x)\leq m$,
    \item $\apmd(f,x)=\md(F,x)$,
    \item all $F_i$ are differentiable at $x$,
    \item $\md(F,x)(v)=\sup_i|\nabla F_i(x)\cdot v|$.
\end{itemize}
For $j=1,2,\ldots$ and $x\in E$ define the seminorms
$$
\tau_j^x(v):=
\max_{1\leq i\leq j} |\nabla F_i(x)\cdot v|,\qquad v\in\bbbr^n.
$$
Equivalently, $\tau_j^x(v)=\|L_j(x)v\|_\infty$, where
$$
L_j(x):\bbbr^n\to\bbbr^j,
\qquad
L_j(x)v
=
\left(
\nabla F_1(x)\cdot v,\ldots,\nabla F_j(x)\cdot v
\right).
$$
Obviously, the map $x\mapsto L_j(x)\in \mathrm{Lin}(\bbbr^n,\bbbr^j)$ is $\H^n$-measurable.

By Proposition~\ref{MDT2},
$$
\md(F,x)(v)=\sup_{i\in\bbbn}|\nabla F_i(x)\cdot v|,
\qquad v\in\bbbr^n,
$$
in particular, $\tau_j^x(v)\leq\md(F,x)(v)$ for all $v\in\bbbr^n$, and thus $N_{\md(F,x)} \subset N_{\tau_j^x}$. Therefore, $\rank \tau_j^x\leq \rank \md (F,x)\leq m$ for every $x\in E$ and every $j$.

The seminorms $\tau_j^x$ form a sequence which converges pointwise to $\md(F,x)$ for $x\in E$.
Since for every $v$ the sequence $\tau_j^x(v)$ is non-decreasing and the sphere $\{|v|=1\}$ is compact, by the Dini Theorem the convergence $\tau_j^x(v)\to \md(F,x)(v)=\apmd(f,x)(v)$ is uniform on the unit sphere, so $d_\infty(\tau_j^x,\apmd(f,x))\to 0$. Then, by Corollary~\ref{SJT6}, $|J_m(\apmd(f,x))|=\lim_{j\to\infty}|J_m(\tau_j^x)|$ for all $x\in E$.

Since $E$ is a subset of $A$ of full $\H^n$ measure, to prove the measurability of  $x\mapsto |J_m(\apmd(f,x))|$ on $A$ it remains to show that the functions $x\mapsto |J_m(\tau_j^x)|$ are $\H^n$ measurable on $E$. Recall that $\tau_j^x(v)=\|L_j(x)v\|_{\infty}$.

For a linear map $L\colon \mathbb R^n\to\mathbb R^j$, define the seminorm
$\sigma_L$ on $\mathbb R^n$ by $\sigma_L(v)=\|Lv\|_\infty$.

Clearly, $d_\infty(\sigma_L,\sigma_{L'})\leq \|L-L'\|$, so
$L\mapsto \sigma_L$ is continuous as a map from
$\operatorname{Lin}(\mathbb R^n,\mathbb R^j)$ to the space of seminorms, and consequently, it is continuous on the closed subset $\mathcal{L}_m:=\{L\in \operatorname{Lin}(\mathbb R^n,\mathbb R^j)~:~\rank L\leq m\}$. By Corollary~\ref{SJT6}, the map $L\mapsto |J_m(\sigma_L)|$ is continuous on $\mathcal{L}_m$. Since
$x\mapsto L_j(x)$ is $\H^n$ measurable on $E$, it follows that
$$
x\mapsto |J_m(\sigma_{L_j(x)})|=|J_m(\tau_j^x)|
$$
is $\H^n$ measurable on $E$.
\end{proof}


The next lemma addresses the measurability of the right-hand-side integrand of \eqref{AFeq3}. We shall address the $\H^n$-measurability of the right-hand-side integrand of \eqref{AFeq2} in Remark \ref{AF:rem1}.

We say that $A$ is {\em $\sigma$-compact} if $A$ can be written as a countable union of compact sets.
For a much deeper version of the next lemma, see \cite{evseev}.
\begin{lemma}
\label{AF:lem1}
Assume that $f\colon \mathbb R^n\to X$ is continuous, where $X$ is a metric space.
Then the following hold.

(a) If $A$ is $\sigma$-compact, then $y\mapsto N(f,A,y)$ is Borel.

(b) If $f$ is Lipschitz and $A$ is $\mathcal H^n$-measurable, then
$y\mapsto N(f,A,y)$ is $\mathcal H^n$-measurable.
\end{lemma}

\begin{proof}
We first prove (a) in the case when $A=K$ is compact.
It suffices to prove that the sets
\[
N_k:=\{y\in X:N(f,K,y)\geq k\},
\qquad
k\in\mathbb{N},
\]
are Borel. Indeed, then
$$
N(f,K,y)=\sum_{k=1}^\infty  \chi_{N_k}(y)
\qquad
\text{is Borel.}
$$
Fix $k\in\mathbb N$. For $\ell\in\mathbb N$, let
\[
E_{k,\ell}=\{(x_1,\ldots,x_k)\in K^k: |x_i-x_j|\geq 1/\ell
\text{ for } i\neq j,\ f(x_1)=\cdots=f(x_k)\}.
\]
The set $E_{k,\ell}$ is compact. Indeed, it is closed in the compact set $K^k$.
Let
\[
B_{k,\ell}=(f\circ\pi_1)(E_{k,\ell}),
\]
where $\pi_1(x_1,\ldots,x_k)=x_1$. Then $B_{k,\ell}$ is compact.

We claim that
\[
\{y\in X:N(f,K,y)\geq k\}=\bigcup_{\ell=1}^{\infty}B_{k,\ell}.
\]
Indeed, if $y\in B_{k,\ell}$, then there are $k$ distinct points
$x_1,\ldots,x_k\in K$ such that $f(x_1)=\cdots=f(x_k)=y$, and hence
$N(f,K,y)\geq k$. Conversely, if $N(f,K,y)\geq k$, then there are distinct points
$x_1,\ldots,x_k\in K$ such that $f(x_1)=\cdots=f(x_k)=y$. Since the points are
distinct, there is $\ell\in\mathbb N$ such that $|x_i-x_j|\geq 1/\ell$ for all
$i\neq j$. Thus $(x_1,\ldots,x_k)\in E_{k,\ell}$ and $y\in B_{k,\ell}$.

Therefore, $\{y\in X:N(f,K,y)\geq k\}$ is Borel, since it is a countable union of compact sets.

Now suppose that $A$ is $\sigma$-compact. Write
$A=\bigcup_{j=1}^{\infty}K_j$, where $K_j$ are compact and
$K_1\subset K_2\subset\ldots$ Then, for every $k\in\mathbb N$,
\[
\{y\in X:N(f,A,y)\geq k\}
=
\bigcup_{j=1}^{\infty}\{y\in X:N(f,K_j,y)\geq k\}.
\]
The sets on the right-hand side are Borel by the compact case. Hence
$y\mapsto N(f,A,y)$ is Borel. This proves (a).

We now prove (b). Since $A$ is $\mathcal H^n$-measurable in $\mathbb R^n$, by
inner regularity we can find a $\sigma$-compact set $K\subset A$ such that
$\mathcal H^n(A\setminus K)=0$.

Let $Z=A\setminus K$. Since $f$ is Lipschitz,
$\H^n(f(Z))\leq \operatorname{Lip}(f)^n\mathcal H^n(Z)=0$.
For $y\notin f(Z)$ we have
\[
N(f,A,y)=N(f,K,y).
\]
By part (a), the function $y\mapsto N(f,K,y)$ is Borel. Thus
$y\mapsto N(f,A,y)$ agrees with a Borel function outside the
$\mathcal H^n$-null set $f(Z)$. Hence $y\mapsto N(f,A,y)$ is
$\mathcal H^n$-measurable.
\end{proof}

\begin{proof}[Proof of Proposition~\ref{AFT2}, assuming Proposition~\ref{AFT3}]
Since Proposition~\ref{AFT2} implies Theorem~\ref{AFT1}, this will also complete
the proof of Theorem~\ref{AFT1}.

Assume Proposition~\ref{AFT3} is proved, i.e., \eqref{AFeq3} holds for any $\H^n$-measurable set in $\bbbr^n$. Let $A\subset\bbbr^n$ be measurable. Then by linearity of the integral, for any simple function
$$
g=\sum_{j=1}^N a_j\chi_{E_j},\qquad a_j\geq 0,
$$
Proposition~\ref{AFT3} applied to $E_j\cap A$ yields
\[
\begin{split}
\int_A g(x) |J_n(\md(f,x))|\,d\H^n(x)=\int_{\bbbr^n} \chi_A(x)\sum_{j=1}^N a_j \chi_{E_j}(x) |J_n(\md(f,x))|\,d\H^n(x)\\
=\sum_{j=1}^N a_j \int_{E_j\cap A}|J_n(\md(f,x))|\,d\H^n(x)
=\sum_{j=1}^N a_j \int_{\ell^\infty} N(f,E_j\cap A,y)\,d\H^n(y)
\end{split}
\]
Note, however, that for $g(x)=\sum_{j=1}^N a_j\chi_{E_j}(x)$,
\begin{equation}
\label{AF:eq3}
\sum_{j=1}^N a_j N(f,E_j\cap A,y)=\sum_{j=1}^N \sum_{x\in f^{-1}(y)\cap A}a_j\chi_{E_j}(x)=\sum_{x\in f^{-1}(y)\cap A}g(x),
\end{equation}
so
$$
\int_A g(x) |J_n(\md(f,x))|\,d\H^n(x)=\int_{\ell^\infty} \sum_{x\in f^{-1}(y)\cap A}g(x)\,d\H^n(y),
$$
i.e., \eqref{AFeq2} holds for any simple function $g$.

For a general $\H^n$-measurable, nonnegative $g:A\to [0,\infty)$, we can find an increasing sequence of simple functions $g_j$ converging pointwise to $g$. Then
\eqref{AFeq2} holds for all $g_j$, and since obviously $\sum_{x\in f^{-1}(y)\cap A} g_j(x) \nearrow \sum_{x\in f^{-1}(y)\cap A}g(x) $, by the monotone convergence theorem \eqref{AFeq2} holds for $g$, as well.
\end{proof}
\begin{remark}\label{AF:rem1}
By tracking the above proof we can also derive the $\H^n$-measurability of $G(y):=\sum_{x\in f^{-1}(y)\cap A}g(x)$, i.e., the right-hand-side integrand of \eqref{AFeq2}. Namely, if $(g_j)$ is the increasing sequence of simple functions that converge pointwise to $g$, we have $\sum_{x\in f^{-1}(y)\cap A} g_j(x) \nearrow G(y)$, and a pointwise limit of $\H^n$-measurable functions is $\H^n$-measurable, so it suffices to prove that $G(y)$ is $\H^n$-measurable in the case when $g$ is a non-negative simple function, $g(x)=\sum_{j=1}^N a_j\chi_{E_j}(x)$. Then, however,  by \eqref{AF:eq3},
$G(y)=\sum_{x\in f^{-1}(y)\cap A}g(x)=\sum_{j=1}^N a_j N(f,E_j\cap A,y)$ and $\H^n$-measurability of $G$ follows from Lemma \ref{AF:lem1}.
\end{remark}

\begin{proof}[Proof of Proposition \ref{AFT3}]
Recall that if $f:\bbbr^n\to\ell^\infty$ is Lipschitz, then for $\H^n$-a.e.\ $x\in\bbbr^n$, we have
\begin{equation}
\label{AF:md-limit}
\lim_{y\to x}
\frac{\Vert f(y)-f(x)\Vert_\infty-\md(f,x)(y-x)}{|y-x|}=0.
\end{equation}

Let
\begin{itemize}
\item $\Omega_0=\{x\in\bbbr^n~~:~~\text{\eqref{AF:md-limit} holds and }\md(f,x)\text{ is \emph{not} a norm}\}$,
\item $\Omega_1=\{x\in\bbbr^n~~:~~\text{\eqref{AF:md-limit} holds and }\md(f,x)\text{ is a norm}\}$
\item $Z=\{x\in\bbbr^n~~:~~\text{\eqref{AF:md-limit} does not hold}\}$
\end{itemize}

Let us now proceed with the proof of the formula \eqref{AFeq3}.

Set $A_1:=\Omega_1 \cap A$. Note that $A\setminus A_1\subset \Omega_0\cup Z$ and $f(A\setminus A_1)\subset f(\Omega_0)\cup f(Z)$. Also,
\begin{itemize}
\item by definition, $|J_n(\md(f,x))|=0$ for $x\in \Omega_0$,
\item $f(\Omega_0)$ is, by the  Metric Sard Theorem (Theorem~\ref{T17}),  a $\H^n$-null set,
\item $Z$ is a $\H^n$-null set,  by the Kirchheim-Rademacher Theorem, and since $f$ is Lipschitz, also ${\H^n(f(Z))=0}$.
\end{itemize}
(Note that the proof of Theorem~\ref{T17} does not use any form of the area formula.)

Therefore $N(f,A\setminus A_1,y)$, which is non-zero only if $y\in f(A\setminus A_1)$, vanishes for $\H^n$-almost every $y$,
\begin{equation}
\label{eqnull}
\int_{A\setminus A_1} |J_n(\md(f,x))| \, d\H^n(x) = 0=\int_{\ell^\infty} N(f,A\setminus A_1,y) \, d\H^n(y) \,,
\end{equation}
and since $N(f,A,y)=N(f,A_1,y)+N(f,A\setminus A_1,y)$, it suffices to prove \eqref{AFeq3} with $A_1$ in place of $A$.

Since the space of all norms on $\bbbr^n$ is separable (Lemma~\ref{lem:semi sep}),  we can choose a dense countable subset $\{\sigma_1,\sigma_2,\ldots\}$.
Fix $\lambda>1$ and $\eps > 0$ small enough so that $\lambda^{-1}+\eps < 1 < \lambda - \eps$. For $j=1,2,\ldots$, let $B_j$ be the set of all $x \in A_1$ such that
\begin{equation}
\label{ma3}
(\lambda^{-1}+\eps)  \sigma_j(v) \leq \md(f,x)(v) \leq (\lambda - \eps) \sigma_j(v), \quad \text{for all $v \in \bbbr^n$.}
\end{equation}
Then $B_j$ are measurable. We claim that
$$
A_1 = \bigcup_{j=1}^\infty B_j.
$$
Indeed, let $x\in A_1$. We need to find $\sigma_j$ that satisfies \eqref{ma3}. Since $\md(f,x)$ is a norm, $m:=\inf_{|v|=1}\md(f,x)(v)>0$. Let $\eta>0$ be so small that
$$
\lambda^{-1}+\eps<\frac{1}{1+\eta}<\frac{1}{1-\eta}<\lambda-\eps.
$$
By the density, we can find $\sigma_j$ such that
$$
\sup_{|v|=1}|\sigma_j(v)-\md(f,x)(v)|<m\eta,
$$
and it easily follows that $\sigma_j$ satisfies \eqref{ma3}. One first verifies it for $|v|=1$, and then the inequality follows for all $v$ by homogeneity. We leave details to the reader.

For each $j=1,2,\ldots$, $\sigma_j$ is a norm, so there is a constant $c_j >0$ such that
\begin{equation}\label{eq:ucmo2}
     \sigma_j(v) \ge c_j |v| \quad \text{for all $v \in \bbbr^n$}
\end{equation}

Finally, for every $x \in B_j$ there exists a $\delta >0$ such that
\begin{equation}\label{eq:ucmo1}
\big|\, \Vert f(x)-f(y)\Vert_\infty - \md(f,x)(y-x)\big| \leq \eps c_j|y-x|, \; \text{for all $y \in B(x,\delta)$.}
\end{equation}
Keeping $j$ fixed, for $k=1,2,\ldots$ let $ D^j_k $ be the set of all $x \in B_j$ such that $\delta=1/k$ satisfies \eqref{eq:ucmo1}, i.e.\
\begin{equation}
    \label{ma4}
    \big|\, \Vert f(x)-f(y)\Vert_\infty - \md(f,x)(y-x)\big| \leq \eps c_j|y-x|, \; \text{for all $y \in B(x,{1}/{k})$.}
\end{equation}
Then $D^j_k$ are measurable and clearly $B_j = \bigcup_k D^j_k$, so $A_1 =\bigcup_k\bigcup_j D^j_k$.

Fix any measurable $E\subset D^j_k$ with $\diam E < 1/k$. For any $x,y \in E$, since $x \in D^j_k$ and $y \in B(x,{1}/{k})$, all the inequalities in \eqref{ma3} and \eqref{ma4} hold. So, by the triangle inequality
\begin{align*}
    \Vert f(x)-f(y)\Vert_\infty &\leq \md(f,x)(y-x)+\eps c_j|y-x| \\
    &\le (\lambda-\eps)\sigma_j(y-x)+\eps c_j|y-x|\\
    &= \lambda \sigma_j(y-x) +\eps (c_j|y-x|-\sigma_j(y-x))\\
    &\le \lambda \sigma_j(y-x),
\end{align*}
where in the last inequality we used \eqref{eq:ucmo2}.

Analogously, we calculate
\begin{align*}
    \Vert f(x)-f(y)\Vert_\infty &\ge \md(f,x)(y-x)-\eps c_j|y-x| \\
    &\ge (\lambda^{-1}+\eps)\sigma_j(y-x)-\eps c_j|y-x|\\
    &= \lambda^{-1} \sigma_j(y-x) +\eps (\sigma_j(y-x)-c_j|y-x|)\\
    &\ge \lambda^{-1} \sigma_j(y-x),
\end{align*}
Together they yield
$$
\lambda^{-1} \sigma_j(y-x) \le \Vert f(x)-f(y)\Vert_\infty \le \lambda \sigma_j(y-x), \; \text{for all $x$ and $y$ in $E$}.
$$
In other words,
$$
f\colon (E,\sigma_j) \to \ell^\infty
$$
is $\lambda$-bi-Lipschitz. In particular, it is one-to-one and
\begin{equation}
\label{nord1}
 \lambda^{-n} \H^n_{\sigma_j}(E) \leq  \H^n(f(E)) \leq \lambda^n\H^n_{\sigma_j}(E) \, .
\end{equation}
On the other hand, since $E \subset B_j$,  for every $x \in E$ we have, by \eqref{ma3},
$$
\{v~\in \bbbr^n~:~\lambda^{-1}\sigma_j(v)\leq 1\}\supset \{v\in\bbbr^n~:~\md(f,x)(v)\leq 1\}\supset  \{v~\in \bbbr^n~:~\lambda\sigma_j(v)\leq 1\}
$$
thus, by the definition of the $n$-Jacobian,
\begin{equation}
\label{nord2}
\lambda^{-n} |J_n(\sigma_j)| \leq |J_n(\md(f,x))| \leq \lambda^n |J_n(\sigma_j)|.
\end{equation}

Then, by Proposition~\ref{SJT1},  for every $j$,
$$
\H^n_{\sigma_j}(E) = |J_n(\sigma_j)| \H^n(E) = \int_{E} |J_n(\sigma_j)| \, d\H^n(x) \, .
$$
Thus, from \eqref{nord1}, \eqref{nord2} and the last equality, we get
\begin{equation}
\label{nord3}
 \lambda^{-2n} \int_{E} |J_n(\md(f,x))| \, d\H^n(x)  \leq  \H^n(f(E))  \leq \lambda^{2n}\int_{E} |J_n(\md(f,x))| \, d\H^n(x)  \, .
\end{equation}
By representing each of the sets $D_k^j$ as a union of measurable sets of diameter less than $1/k$ we can write
$$
A_1=\bigcup_{k=1}^\infty\bigcup_{j=1}^\infty D_k^j=\bigcup_{\ell=1}^\infty E_\ell,
$$
where the sets $E_\ell$ are measurable and for each $\ell$ there are $k$ and $j$ such that $E_\ell\subset D_k^j$, $\diam E_\ell<1/k$. Using a standard argument we may assume that the sets $E_\ell$ are pairwise disjoint. The map $f$ is one-to-one on each $E_\ell$ and \eqref{nord3} holds.

Since,
$$
N(f,A_1,y)=\sum_{\ell=1}^\infty N(f,E_\ell,y)=\sum_{\ell=1}^\infty\chi_{f(E_\ell)}(y),
$$
summing \eqref{nord3} over all $\ell$ yields
\begin{align*}
   \lambda^{-2n}\int_{A_1} |J_n(\md(f,x))| \, d\H^n(x) &\leq  \int_{\ell^\infty} N(f,A_1,y)\, d\H^n(y)
   &\leq  \lambda^{2n} \int_{A_1} |J_n(\md(f,x))| \, d\H^n(x) \, .
\end{align*}
Since $\lambda>1$ was arbitrary, by letting $\lambda \to 1^+$ we obtain the area formula on $A_1$, i.e.\
$$
\int_{A_1} |J_n(\md(f,x))| \, d\H^n(x) = \int_{\ell^\infty} N(f,A_1,y)\, d\H^n(y),
$$
which together with \eqref{eqnull} yields
$$
\int_{A} |J_n(\md(f,x))| \, d\H^n(x) = \int_{\ell^\infty} N(f,A,y)\, d\H^n(y),
$$
as desired.
\end{proof}

As a final result, we present a useful change of variables formula. \begin{corollary}
\label{gum14}
Suppose $A\subset\bbbr^n$ is measurable, $X$ is a metric space, $f:A\to X$ is Lipschitz and $u:f(A)\to [0,\infty]$ is $\H^n$-measurable. Then
$$
\int_A u(f(x))|J_n(\apmd (f,x))|\,d\H^n(x)=\int_Xu(y) N(f,A,y)\,d\H^n(y).
$$
\end{corollary}
\begin{proof}
The corollary essentially follows by applying Theorem~\ref{AFT1} to $g=u\circ f$, but there is a subtle point: unless $u$ is Borel, $g=u\circ f$ need not be measurable.

Note that $f(A)$ is $\H^n$-$\sigma$-finite. Indeed, set $A_k=A\cap B(0,k)$, then $f(A)=\bigcup_{k=1}^\infty f(A_k)$ and $\H^n(f(A_k))\leq \rm{Lip}(f)^n k^n\omega_n<\infty$.

By Lemma~\ref{app:measborel}, there is  a Borel $v:f(A)\to [0,\infty]$ and a $\H^n$-null Borel set $Z\subset f(A)$ such that $u=v$ on $f(A)\setminus Z$. Applying the area formula \eqref{AFeq3} to $f^{-1}(Z)$ we get
$$
\int_{f^{-1}(Z)}|J_n(\apmd(f,x))|\,d\H^n(x)=\int_Z N(f,A,y)\,d\H^n(y)=0,
$$
so $|J_n(\apmd(f,x))|=0$ for $\H^n$-a.e. $x\in f^{-1}(Z)$ and
$$
u(f(x))|J_n(\apmd(f,x))|=v(f(x))|J_n(\apmd(f,x))|\quad \text{for $\H^n$-a.e. }x\in A.
$$
Now $v\circ f$ is measurable, so we can apply Theorem~\ref{AFT1} to $g=v\circ f$, which gives
\begin{equation*}
    \begin{split}
\int_A u(f(x))&|J_n(\apmd(f,x))|\,d\H^n(x)=\int_A v(f(x))|J_n(\apmd(f,x))|\,d\H^n(x)\\
&=\int_X\sum_{x\in f^{-1}(y)\cap A}v(f(x))\,d\H^n(y)=\int_X v(y)N(f,A,y)\,d\H^n(y)\\
&=\int_X u(y)N(f,A,y)\,d\H^n(y),
    \end{split}
\end{equation*}
as desired.
\end{proof}

\section{Co-area inequality and Sard's theorem}
\label{coarea}
If $(X,\mu)$ is a measure space, then for {\em any} function $f:X\to [0,\infty]$ we define the {\em upper integral} by
$$
\int^*_X f \ d\mu = \inf \, \int_X \phi \, d\mu \, ,
$$
where the infimum is taken over all {\em $\mu$-measurable} functions $\phi$ satisfying $0 \leq f(x) \leq \phi(x)$ for $\mu$-a.e.\ $x \in X$.

While in the definition of the upper integral we do not require measurability of $f$, if $\int_X^*f\,d\mu=0$, then $f=0$ $\mu$-a.e.

Clearly,
\begin{equation}
\label{Ceq5}
\int_X^*f\, d\mu=\int_{X\setminus Z}^* f\, d\mu
\quad
\text{if } \mu(Z)=0.
\end{equation}
It is easy to check that if $f_i:X\to [0,\infty]$, $i\in\bbbn$, then
$$
\int_X^*\sum_{i=1}^\infty f_i\, d\mu\leq \sum_{i=1}^\infty\int_X^* f_i\, d\mu.
$$
The following monotone convergence is also an easy exercise, see \cite[Lemma~2.2]{EH1}:
If $f_n\colon X\to [0,\infty]$ is a monotone sequence of functions $0\leq f_1\leq f_2\leq\ldots$ and $\lim_{n\to\infty} f_n(x)=f(x)$ for $\mu$-a.e.\ $x\in X$, then
\begin{equation}
\label{Ceq6}
\lim_{n\to\infty} \int_X^*f_n\, d\mu=\int_X^*f\, d\mu.
\end{equation}

The next result is a classical co-area inequality.
\begin{theorem}[Co-area inequality]
\label{CT1}
Let $X$ and $Y$ be arbitrary metric spaces, $0\leq m\leq n<\infty$ (any) real numbers and $A\subset X$ any subset.
Then, for any $L$-Lipschitz map $ f: X \to Y$ we have
$$
\int^*_Y \mathcal{H}^{n-m} \left(f^{-1}(y) \cap A\right)\,  d\mathcal{H}^m (y)  \leq \frac{\omega_{n-m} \omega_m}{\omega_{n}}L^m \,\mathcal{H}^{n}(A) \, .
$$
Moreover if $X$ is proper (i.e., bounded and closed sets in $X$ are compact), $A$ is $\H^n$\nobreakdash-measurable, and $\H^n(A)<\infty$, then the function
$$
y\mapsto \mathcal{H}^{n-m} \left(f^{-1}(y) \cap A\right)
$$
is $\H^m$-measurable and therefore, the upper integral  can be replaced with the usual integral.
\end{theorem}
This beautiful theorem is the culmination of work by N\"obeling, Szpilrajn, Eilenberg, Harold, Federer, and Davies, in chronological order, spanning a period of forty years.

The proof is highly nontrivial. For a self-contained, detailed, and relatively elementary treatment, see \cite{EH1}, which also includes an extensive historical overview.

Theorem~\ref{CT1} follows from a more general result, Theorem~\ref{CT2} below.
\begin{definition}
\label{CD1}
For an arbitrary map $f:X\to Y$ between metric spaces, $m,n\in [0,\infty)$, $\delta\in (0,\infty]$, and any $A\subset X$ we define
$$
\Phi^{m,n}_\delta(f,A):=\frac{\omega_m\omega_n}{2^{m+n}}\inf\sum_{i=1}^\infty (\diam f(A_i))^m(\diam A_i)^n,
$$
where the infimum is taken over all $\delta$-coverings  $\{A_i\}_{i=1}^\infty$ of $A$ (see Definition~\ref{HMD1} of the Hausdorff content; if $m=0$ or $n=0$, the same convention for $0^0$ is used as there). Furthermore, we define
$$
\Phi^{m,n}(f,A):=\lim_{\delta\to 0^+}\Phi^{m,n}_\delta(f,A).
$$
Note that the limit exists, because the function $\delta\mapsto\Phi_\delta^{m,n}$ is non-increasing.
\end{definition}
It is easy to see that
\begin{equation}
\label{Ceq1}
\Phi^{m,n}(f,A\cup B)\leq \Phi^{m,n}(f,A)+\Phi^{m,n}(f,B).
\end{equation}

The next result from \cite[Theorem~7.1]{EH1} is a generalization of Theorem~\ref{CT1}.
\begin{theorem}
\label{CT2}
If $f:X\to Y$ is a uniformly continuous map between metric spaces, $0\leq m\leq n<\infty$ are real numbers, and $A\subset X$, then
$$
\int_Y^* \H^{n-m}(f^{-1}(y)\cap A)\, d\H^m(y)\leq
\Phi^{m,n-m}(f,A).
$$
\end{theorem}
Indeed, since for an $L$-Lipschitz $f$ we have
$\diam f(A_i)\leq L\diam A_i$, it follows directly from the definitions of $\Phi^{m,n}$ and $\H^n$ that
\begin{equation}
\label{Ceq2}
\Phi^{m,n-m}(f,A)\leq L^m\frac{\omega_{n-m}\omega_m}{\omega_n}\, \H^n(A).
\end{equation}
Although Theorem~\ref{CT2} does not mention measurability of $y\mapsto \mathcal{H}^{n-m} \left(f^{-1}(y) \cap A\right)$, that part of Theorem~\ref{CT1} is easy (we present this reasoning in the case $X=\bbbr^n$ as Lemma \ref{lem:measfib} below).
The proof of Theorem~\ref{CT2} is long and difficult, and we will prove it here only under the additional assumption that $Y$ is $\H^m$-$\sigma$-finite. This special case is sufficient for our main application, which is the proof of the Metric Co-area Theorem~\ref{INT5}.
\begin{proof}[Proof of {Theorem~\ref{CT2}} when $Y$ is $\H^m$-$\sigma$-finite]
We can in fact assume that $\H^m(Y)<\infty$, because we can represent $Y$ as a union of an increasing sequence of sets $Y_i$ of finite $\H^m$ measure and pass to the limit applying \eqref{Ceq6}.
We can also assume that $\Phi^{m,n-m}(f,A)<\infty$ as otherwise the inequality is obvious.
Fix $\eps>0$.

Theorem~\ref{AT1} yields that there is a set $Z\subset Y$ with $\H^m(Z)=0$ such that
\begin{equation}
\label{Ceq3}
\forall y\in Y\setminus Z\ \exists\delta_y>0\ \forall F\subset Y
\ \Big(y\in F\subset\bar{B}(y,\delta_y)\  \Rightarrow \
\H^m(F)\leq (1+\eps)\frac{\omega_m}{2^m} (\diam F)^m\Big).
\end{equation}

For $j\in\bbbn$, let $W_j$ be the set of points $y\in Y\setminus Z$ such that
$$
\forall F\subset Y
\quad \Big(y\in F\subset\bar{B}(y,1/j)\  \Rightarrow \
\H^m(F)\leq (1+\eps)\frac{\omega_m}{2^m} (\diam F)^m\Big).
$$
It follows from \eqref{Ceq3} that
\begin{equation}
\label{Ceq7}
Y\setminus Z=\bigcup_{j=1}^\infty W_j,
\qquad
W_1\subset W_2\subset W_3\subset\ldots,
\qquad
\H^m(Z)=0.
\end{equation}
Therefore, it suffices to show that
\begin{equation}
\label{Ceq8}
\int_{W_j}^*\H^{n-m}(f^{-1}(y)\cap A)\, d\H^m(y)\leq
(1+\eps)\big(\Phi^{m,n-m}(f,A)+\eps\big).
\end{equation}
Indeed, \eqref{Ceq7} along with \eqref{Ceq5} and \eqref{Ceq6} will give
$$
\int_Y^*\cdots=\int_{Y\setminus Z}^*\cdots=\lim_{j\to\infty}
\int_{W_j}^*\H^{n-m}(f^{-1}(y)\cap A)\, d\H^m(y) \leq
(1+\eps)\big(\Phi^{m,n-m}(f,A)+\eps\big),
$$
and the result will follow upon letting $\eps\to 0^+$.

Fix $j\in\bbbn$. Since $f$ is uniformly continuous, there is
$\eta_j>0$ such that $\diam f(E)<1/j$ whenever $\diam E\leq \eta_j$.

Choose a decreasing sequence $\delta_k\to 0$ with $\delta_k\leq\eta_j$.
For each $k$, take a $\delta_k$-covering $\{A_{ik}\}_{i=1}^\infty$
of $A$ such that
$$
\frac{\omega_m\omega_{n-m}}{2^n}
\sum_{i=1}^\infty
(\diam f(A_{ik}))^m(\diam A_{ik})^{n-m}
<
\Phi^{m,n-m}_{\delta_k}(f,A)+\eps \leq
\Phi^{m,n-m}(f,A)+\eps.
$$
Since $\diam A_{ik}\leq\delta_k\leq\eta_j$, we have $\diam f(A_{ik})<\frac{1}{j}$.

If $i\in I_k:=\big\{i:\, W_j\cap\overline{f(A_{ik})}\neq\varnothing \big\}$, we can find $y\in W_j$ such that $y\in\overline{f(A_{ik})}\subset \bar{B}(y,1/j)$, and the definition of the set $W_j$ yields that
$$
\H^m\big(\overline{f(A_{ik})}\big)\leq
(1+\eps)\frac{\omega_m}{2^m}\big(\diam\overline{f(A_{ik})}\big)^m=
(1+\eps)\frac{\omega_m}{2^m}\big(\diam f(A_{ik})\big)^m.
$$
For any $y\in Y$, the sets $\{ A_{ik}:\, y\in f(A_{ik})\}$ form a $\delta_k$-covering of $f^{-1}(y)\cap A$ and hence
$$
\H^{n-m}_{\delta_k}(f^{-1}(y)\cap A)\leq
\frac{\omega_{n-m}}{2^{n-m}}\sum_{i=1}^\infty (\diam A_{ik})^{n-m}\chi_{\overline{f(A_{ik})}}(y).
$$
We could place $\chi_{f(A_{ik})}(y)$ on the right-hand side, but we use the characteristic function of the closure $\overline{f(A_{ik})}$ in order to ensure that the function on the right-hand side is measurable.

If $y\in W_j$, then by the definition of $I_k$,
$$
\H^{n-m}_{\delta_k}(f^{-1}(y)\cap A)\leq
\frac{\omega_{n-m}}{2^{n-m}}\sum_{i\in I_k} (\diam A_{ik})^{n-m}\chi_{\overline{f(A_{ik})}}(y).
$$
Integrating both sides of this inequality over $W_j$ yields
\begin{equation*}
\begin{split}
&
\int_{W_j}^*\H^{n-m}_{\delta_k}(f^{-1}(y)\cap A)\, d\H^m(y)\leq
\frac{\omega_{n-m}}{2^{n-m}}\sum_{i\in I_k} (\diam A_{ik})^{n-m} \H^{m}\big(\overline{f(A_{ik})}\big) \\
&\leq
(1+\eps)\frac{\omega_{n-m}\omega_m}{2^n}\sum_{i\in I_k}(\diam A_{ik})^{n-m}
(\diam f(A_{ik}))^m \leq
(1+\eps)\big(\Phi^{m,n-m}(f,A)+\eps\big).
\end{split}
\end{equation*}
Since $\delta_k\searrow 0$, for every $y\in Y$ we have $\H^{n-m}_{\delta_k}(f^{-1}(y)\cap A)\nearrow\H^{n-m}(f^{-1}(y)\cap A)$.
Therefore, applying \eqref{Ceq6} to the functions $\chi_{W_j}(y)\H^{n-m}_{\delta_k}(f^{-1}(y)\cap A)$
gives \eqref{Ceq8}. The proof is complete.
\end{proof}

\begin{lemma}
\label{lem:measfib}
Assume $X$ is an arbitrary metric space, $0\leq m\leq n$ are integers, $A\subset \bbbr^n$ is $\H^n$-measurable and $f:A\to X$ is Lipschitz. Then the function
$$
z\longmapsto \H^{n-m}(f^{-1}(z))=\H^{n-m}(f^{-1}(z)\cap A),\quad z\in X,
$$
is $\H^m$-measurable.
\end{lemma}
\begin{proof}
Let $K_1\subset K_2\subset\cdots\subset A$ be compact sets approximating $A$ from within, so that $Z=A\setminus \bigcup_{i=1}^\infty K_i$ is $\H^n$-null, $\H^n(Z)=0$. Then $A=\bigcup_{i=1}^\infty K_i \cup Z$ and for  any $z\in X$ we have
$$
\H^{n-m}(f^{-1}(z)\cap \bigcup_{i=1}^\infty K_i)\leq \H^{n-m}(f^{-1}(z)\cap A)\leq \H^{n-m}(f^{-1}(z)\cap\bigcup_{i=1}^\infty K_i)+\H^{n-m}(f^{-1}(z)\cap Z).$$
By the Coarea inequality (Theorem \ref{CT1}),
$$
\int_X^{*}
    \H^{n-m}\left(f^{-1}(z)\cap Z\right)\,d\H^m(z)
\leq
(\operatorname{Lip} f)^m\frac{\omega_{n-m}\omega_m}{\omega_n}\H^n(Z)=0,
$$
thus for $\H^m$-a.e.\ $z\in X$ we have $\H^{n-m}\left(f^{-1}(z)\cap Z\right)=0$
so, again for $\H^m$-a.e.\ $z\in X$,
$$
\H^{n-m}(f^{-1}(z)\cap A)=\H^{n-m}(f^{-1}(z)\cap \bigcup_{i=1}^\infty K_i)=\lim_{i\to\infty}  \H^{n-m}(f^{-1}(z)\cap K_i),
$$
where the last equality follows from the fact that $f^{-1}(z)\cap \bigcup_{i=1}^j K_i=f^{-1}(z)\cap K_j$ forms an increasing sequence of sets. Thus to show the measurability of
$$
z\longmapsto \H^{n-m}(f^{-1}(z))=\H^{n-m}(f^{-1}(z)\cap A)
$$
it suffices to verify that for any compact set $K\subset A$, the function
$$
X\ni z\longmapsto g_K(z)=\H^{n-m}(f^{-1}(z)\cap K)
$$ is $\H^m$-measurable. In fact, we shall prove that it is Borel.

Set
$$
g_j(z)=\H^{n-m}_{1/j,op}(f^{-1}(z)\cap K).
$$

By Lemma \ref{lem:hsop},
$$
g_K(z)=\H^{n-m}(f^{-1}(z)\cap K)=\lim_{j\to\infty} g_j(z)=\sup_j g_j(z),
$$
since $g_j(z)$ is non-decreasing in $j$.

We claim that each $g_j$ is upper semicontinuous. Fix $t\in \bbbr$ and suppose $g_j(z)<t$. By the definition of $\H^{n-m}_{1/j,op}$ (see Lemma \ref{lem:hsop}) we can choose an open $1/j$-cover $A_1,A_2,\ldots$ of $f^{-1}(z)\cap K$ such that $2^{m-n}\omega_{n-m}\sum_i(\diam A_i)^{n-m}<t.$

Set $U=\bigcup_i A_i$. We claim that $f^{-1}(y)\cap K\subset U$ whenever $y$ is sufficiently close to $z$.

Assume to the contrary that $f^{-1}(y)\cap K\not \subset U$, then we can find a sequence $y_k\to z$ such that
$$
x_k\in (f^{-1}(y_k)\cap K)\setminus U.
$$
The set $K\setminus U$ is compact, so after passing to a subsequence we can assume that $x_k\to x$ for some $x\in K\setminus U$, and by continuity of $f$,
$$
f(x)=\lim_{k\to\infty}f(x_k)=\lim y_k=z,
$$
so $x\in f^{-1}(z)\cap K$, while $x\not \in U$. However, by construction, $f^{-1}(z)\cap K\subset U$, which is a contradiction.

Hence $f^{-1}(y)\cap K\subset U$ for $y$ sufficiently close to $z$ and thus $\{A_i\}$ form an open $1/j$-cover of $f^{-1}(y)\cap K$, as well. This shows that $g_j(y)<t$ for $y$ close to $z$, hence the sets $\{g_j<t\}$ are open and $g_j$ are upper semicontinuous.

In particular, every $g_j$ is Borel measurable, and since $g_K=\sup_j g_j$, so is $g_K$.
\end{proof}

For a Lipschitz mapping $f:\bbbr^n\supset A\to X$ we define
$$
\operatorname{Crit}_m(f):=\{x\in A:\, \rank \apmd(f,x)<m\}.
$$
As a main application of Theorem~\ref{CT2} we will prove the following stronger version of the Metric Sard Theorem~\ref{T17}. We believe this result is new.
\begin{theorem}[Metric Sard Theorem II]
\label{CT3}
Let $0\leq m\leq n$, $m\in\bbbz$, $n\in\bbbn$, and let $X$ be any metric space. If $A\subset\bbbr^n$ is measurable, and $f:A\to X$ is Lipschitz, then
$$
\H^{n-m}\big(f^{-1}(z)\cap\operatorname{Crit}_m(f)\big)=0,
\qquad
\text{for } \H^m \text{ almost all } z\in X.
$$
\end{theorem}
\begin{remark}
Adding to $\operatorname{Crit}_m(f)$ the set $Z$ where $\apmd(f,x)$ is not defined does not change anything, because $\H^n(Z)=0$ and hence $\H^{n-m}(f^{-1}(z)\cap Z)=0$ for $\H^m$-almost every $z\in X$ by the co-area inequality.
\end{remark}
\begin{remark}
If $n=m$, then $\operatorname{Crit}_n(f)=\operatorname{Crit}(f)$, $\H^{n-m}=\H^0$ is the counting measure so $\H^{0}\big(f^{-1}(z)\cap\operatorname{Crit}(f)\big)=0$ means that $f^{-1}(z)\cap\operatorname{Crit}(f)=\varnothing$ for $\H^n$ almost all $z\in X$ which is equivalent to $\H^n(f(\operatorname{Crit}(f)))=0$. Thus the Metric Sard Theorem~\ref{T17} follows from Theorem~\ref{CT3}.
\end{remark}
\begin{proof}
If $m=0$, then $\operatorname{Crit}_m(f)=\varnothing$, since the rank of a seminorm is always nonnegative, hence the conclusion of the theorem holds for all $z\in X$. In what follows we assume $1\leq m\leq n$.

We may assume that the set $A$ is bounded, since $A$ can be decomposed into countably many bounded pieces.

For $k=0,1,2,\ldots,m-1$, let
$E_k:=\{x\in A:\, \rank \apmd(f,x)=k\}$, so
$\operatorname{Crit}_m(f)=\bigcup_{k=0}^{m-1} E_k$.
According to Theorem~\ref{CT2}
$$
\int_X^*\H^{n-m}(f^{-1}(z)\cap E_k)\, d\H^m(z)\leq \Phi^{m,n-m}(f,E_k),
$$
and hence it suffices to prove that $\Phi^{m,n-m}(f,E_k)=0$, because it will imply that
$\H^{n-m}(f^{-1}(z)\cap E_k)=0$ for $\H^m$ almost all $z\in X$.

In fact, it suffices to prove that $\Phi^{m,n-m}(f,A_k)=0$, for some
$A_k\subset E_k$ satisfying $\H^n(E_k\setminus A_k)=0$. Indeed, this, \eqref{Ceq1} and \eqref{Ceq2} will give
$$
\Phi^{m,n-m}(f,E_k)\leq \Phi^{m,n-m}(f,A_k)+\Phi^{m,n-m}(f,E_k\setminus A_k)=0.
$$
To this end, it suffices to show that for any $\delta>0$, $\Phi^{m,n-m}_\delta(f,A_k)=0$.
Fix $\delta>0$.

According to Proposition~\ref{T14}, almost every $x\in E_k$ has the following property:

For every integer $N\geq 1$, there is $0<r_{x,N}<\delta/2$ such that $f(A\cap B(x,r_{x,N}))$ can be covered by $N^k$ balls of radius $Cr_{x,N}N^{-1}$. Denote such balls by $\{B_j^{x,N}\}_{j=1}^{N^k}$.

Let $A_k\subset E_k$ be the set of points $x\in E_k$ satisfying the above property, so $\H^n(E_k\setminus A_k)=0$.

Applying Theorem~\ref{AT7} to the family of balls $\{B(x,r_{x,N}/5)\}_{x\in A_k}$ we find a countable (possibly finite) family of pairwise disjoint balls $\{B(x_i,r_{x_i,N}/5)\}_{i\in I}$ such that if we set
$$
U_\delta:=\bigcup_{x\in A_k} B(x,r_{x,N}/5),
$$
then
$$
A_k\subset U_\delta=\bigcup_{x\in A_k} B(x,r_{x,N}/5)\subset\bigcup_{i\in I} B(x_i,r_{x_i,N}).
$$
Let
$$
B_i=B(x_i,r_i):=B(x_i,r_{x_i,N}),
\quad
B_{ij}:=B_j^{x_i,N}
\quad
\text{and}
\quad
A_{ij}:=f^{-1}(B_{ij})\cap B_i.
$$
Note that
$$
\bigcup_{j=1}^{N^k} A_{ij} = \bigcup_{j=1}^{N^k} f^{-1}\big(B_j^{x_i,N}\big)\cap B(x_i,r_{x_i,N})
=A\cap B(x_i,r_{x_i,N})=A\cap B_i
$$
and hence
$$
A_k\subset A\cap \bigcup_{i\in I}B_i= \bigcup_{i\in I}\bigcup_{j=1}^{N^k} A_{ij}.
$$
Since
$\diam A_{ij}\leq\diam B_i\leq 2r_{x_i,N}<\delta$, $\{A_{ij}\}_{i,j}$ is a $\delta$-covering of $A_k$. Note also that
$$
\diam f(A_{ij})\leq \diam B_{ij}\leq Cr_{x_i,N}/N=Cr_i/N.
$$
Therefore,
\[
\begin{split}
&
\Phi_\delta^{m,n-m}(f,A_k)\lesssim\sum_{i,j}\big(\diam f(A_{ij})\big)^m(\diam A_{ij})^{n-m}
\lesssim
\sum_{i,j}\Big(\frac{r_i}{N}\Big)^m\, r_i^{n-m}\\
&=\sum_{i\in I}N^k\Big(\frac{r_i}{N}\Big)^mr_i^{n-m}
=
N^{k-m}5^n\sum_{i\in I}\Big(\frac{r_i}{5}\Big)^n
\lesssim
N^{k-m}\H^n(U_\delta)\stackrel{N\to\infty}{\longrightarrow} 0,
\end{split}
\]
since $k<m$, and the bound $\H^n(U_\delta)<\infty$ is uniform in $N$ (because $A$ is bounded). Therefore, $\Phi_\delta^{m,n-m}(f,A_k)=0$ and that completes the proof.
\end{proof}

The proof of Theorem~\ref{CT3} presented above is not self-contained, since it relies on the difficult Theorem~\ref{CT2}. Although we did not prove Theorem~\ref{CT2} in full generality, we did establish it in the case where $Y$ is $\H^m$-$\sigma$-finite. Hence, the argument above yields a complete proof of the following special case of Theorem~\ref{CT3}, which we record as a corollary. Corollary~\ref{CT4} will be used later in the proof of the Metric Co-area Theorem~\ref{INT5}.
\begin{corollary}
\label{CT4}
Let $0\leq m\leq n$, $m\in\bbbz$, $n\in\bbbn$, and let $X$ be a $\H^m$-$\sigma$-finite metric space. If $A\subset\bbbr^n$ is measurable, and $f:A\to X$ is Lipschitz, then
$$
\H^{n-m}\big(f^{-1}(z)\cap\operatorname{Crit}_m(f)\big)=0,
\qquad
\text{for } \H^m \text{ almost all } z\in X.
$$
\end{corollary}

As an application of the above results we prove rectifiability of fibers of Lipschitz maps.
\begin{definition}
Let $0\leq k\leq n$ be integers. We say that $E\subset\bbbr^n$ is {\em countably $\H^k$-rectifiable} if there are Lipschitz maps $f_i:\bbbr^k\supset A_i\to\bbbr^n$ such that $\H^k(E\setminus \bigcup_{i=1}^\infty f_i(A_i))=0$. In particular, countably $\H^0$-rectifiable sets are the same as countable sets.
\end{definition}

\begin{theorem}
\label{8.9}
Let $0\leq m\leq n$, $m\in\bbbz$, $n\in\bbbn$, and let
$f:\bbbr^n\supset A\to X$ be a Lipschitz map from a measurable set to an arbitrary metric space. Then, for
$\mathcal{H}^m$-almost every $z\in X$, the fiber $f^{-1}(z)$ is countably
$\mathcal{H}^{n-m}$-rectifiable.
\end{theorem}
\begin{proof}
The case $m=0$ is immediate. Assume that $1\leq m\leq n$, and set
\[
A_{<m}
=
\left\{
x\in A:
\rank\apmd(f,x)<m
\right\}
\text{ and }
A_{\geq m}
=
\left\{
x\in A:
\rank\apmd(f,x)\geq m
\right\}.
\]
Clearly, $A=A_{<m}\cup A_{\geq m}\cup N$, where $\H^n(N)=0$.
By the Metric Sard Theorem~\ref{CT3},
\[
\mathcal{H}^{n-m}
\bigl(f^{-1}(z)\cap A_{<m}\bigr)=0
\qquad
\text{for $\mathcal{H}^m$-almost every $z\in X$,}
\]
so for such $z$ that part of the set $f^{-1}(z)$ is countably $\H^{n-m}$-rectifiable. Similarly, the co-area inequality, Theorem~\ref{CT1}, implies that $\H^{n-m}(f^{-1}(z)\cap N)=0$ for $\H^m$ almost all $z\in X$. Thus, it remains to investigate fibers $f^{-1}(z)\cap A_{\geq m}$.

Applying the Metric Implicit Function Theorem~\ref{6:corr2} to $f|_{A_{\geq m}}$ we obtain pairwise disjoint compact sets $K_i\subset A_{\geq m}$ such that
\[
\mathcal{H}^n
\Big(
A_{\geq m}\setminus\bigcup_{i=1}^{\infty}K_i
\Big)=0,
\]
and, for every $i$,  a $C^1$-diffeomorphism
$G_i:\mathbb{R}^n\to
\mathbb{R}^m\times\mathbb{R}^{n-m}$
such that, writing $F_i=f\circ G_i^{-1}$,
\[
F_i^{-1}\bigl(F_i(x,y)\bigr)\cap G_i(K_i)
\subset
\{x\}\times\mathbb{R}^{n-m}
\qquad
\text{for all } (x,y)\in G_i(K_i).
\]
Since the fibers $f^{-1}(z)$ intersected with $K_i$ are contained in diffeomorphic images of $(n-m)$-dimensional  hyperplanes, they are rectifiable. We still have a subset $\tilde{N}\subset A_{\geq m}$, $\H^n(\tilde{N})=0$ that is not covered by the sets $K_i$, but as we have already noticed, $\H^{n-m}(f^{-1}(z)\cap\tilde{N})=0$ for $\H^m$ almost all $z\in X$. The proof is complete.
\end{proof}

\section{Metric co-area formula}
\label{MCOF}

In this section we prove and discuss the Metric Co-area Formula. Its general form, Theorem \ref{thm:gCoa}, follows easily from Proposition \ref{thm:Coa} in a similar way as the Metric Area Formula (Theorem \ref{AFT1}) follows from Proposition \ref{AFT2}.
\begin{theorem}\label{thm:gCoa}
Let $1\leq m\leq n$ be integers.
Suppose $X$ is an $\H^m$-$\sigma$-finite metric space, $A\subset\bbbr^n$ is $\H^n$-measurable, $f:A\to X$ is Lipschitz,  and $g:A\to [0,\infty]$ is $\H^n$-measurable. Then
\begin{equation}
\label{eq:gCoa}
\int_{A} g(x)|J_m(\apmd (f,x))|\,d\H^n(x)=\int_X\left( \int_{f^{-1}(z)}g(y)\,d\H^{n-m}(y)\right) d\H^m(z)
\end{equation}
(measurability of integrands is part of the claim).
\end{theorem}

\begin{proposition}
\label{thm:Coa}
Let $1\leq m\leq n$ be integers.
Suppose $X$ is an $\H^m$-$\sigma$-finite metric space, $A\subset \bbbr^n$ is $\H^n$-measurable and $f:A\to X$ is Lipschitz. Then
\begin{equation}\label{eq:mCoa}
    \int_A |J_m(\apmd(f,x))|\,d\H^n(x)=\int_X \H^{n-m}(f^{-1}(z)\cap A)\,d\H^m(z).
\end{equation}
(measurability of integrands is part of the claim).
\end{proposition}

As in the case of the Metric Area Formula, we shall first show how Theorem \ref{thm:gCoa} follows from Proposition \ref{thm:Coa}, and then concentrate on proving the latter.

\begin{proof}[Proof of Theorem \ref{thm:gCoa}, assuming Proposition \ref{thm:Coa} holds.]\mbox{}\\
The proof is standard: Proposition~\ref{thm:Coa} is Theorem \ref{thm:gCoa} with $g\equiv 1$. Proposition~\ref{thm:Coa} allows us to establish \eqref{eq:gCoa} for any simple function $g$, and then prove the general case by approximation of $g$ by simple functions and the monotone convergence theorem. Here are the details.

First choose a Borel set $A_0\subset A$ such that $\H^n(A\setminus A_0)=0$ and $g|_{A_0}$ is Borel. This is possible by Lemma~\ref{app:measborel} and the Borel regularity of $\H^n$. By the Co-area Inequality (Theorem~\ref{CT1}),
$$
\H^{n-m}\big(f^{-1}(z)\cap(A\setminus A_0)\big)=0
\quad\text{for $\H^m$-a.e.\ $z\in X$.}
$$
For every such $z$, the set $f^{-1}(z)\cap A_0$ is Borel, so the restriction of $g$ to $f^{-1}(z)$ is $\H^{n-m}$-measurable, and
$$
\int_{f^{-1}(z)}g(y)\,d\H^{n-m}(y)
=\int_{f^{-1}(z)\cap A_0}g(y)\,d\H^{n-m}(y).
$$
The inner integral over the original fiber is thus defined for $\H^m$-a.e.\ $z$; on the exceptional null set we assign it the value zero.

Also, $\apmd(f|_{A_0},x)=\apmd(f,x)$ for $\H^n$-a.e.\ $x\in A_0$, so the left-hand side of \eqref{eq:gCoa} does not change when $A$ is replaced by $A_0$. Hence, it suffices to prove the theorem when $A$ and $g$ are Borel, which we assume from now on. In this case all fibers are Borel, and restrictions of Borel functions to the fibers are $\H^{n-m}$-measurable for every $z$.

Let $g$ be an arbitrary non-negative simple function on $A$, $g=\sum_{j=1}^N a_j\chi_{A_j}$, $a_j\geq 0$, $A_j\subset A$ Borel.  For any $j$ we have $\apmd(f|_{A_j},x)=\apmd(f,x)$ for $\H^n$-a.e.\ $x\in A_j$, so applying Proposition~\ref{thm:Coa} to each of the restrictions  $f|_{A_j}:A_j\to X$  we obtain
\begin{equation*}\begin{split}
\int_{A}&g(x)|J_m(\apmd(f,x))|\,d\H^n(x)=\int_{A}\left(\sum_{j=1}^N a_j\chi_{A_j}(x)\right)|J_m(\apmd(f,x))|\,d\H^n(x)\\
    &=\sum_{j=1}^N a_j\int_{A_j} |J_m(\apmd(f,x))|\,d\H^n(x)=\sum_{j=1}^N a_j \int_X \H^{n-m}(f^{-1}(z)\cap A_j)\,d\H^m(z)\\
    &=\sum_{j=1}^N a_j\int_X
    \left(
        \int_{f^{-1}(z)}\chi_{A_j}(y)\,d\H^{n-m}(y)
    \right)d\H^m(z)\\
    &=\int_X\left(\int_{f^{-1}(z)}g(y)\,d\H^{n-m}(y)
    \right)d\H^m(z),
\end{split}
\end{equation*}
i.e., \eqref{eq:gCoa} holds for the arbitrary simple function $g$.

For a general non-negative Borel function $g:A\to[0,\infty]$, choose an increasing
sequence of non-negative Borel simple functions $g_k$ which converge pointwise to $g$. Then for any~$k$
\begin{equation}\label{eq:Coagk}
    \int_{A} g_k(x)|J_m(\apmd(f,x))|\,d\H^n(x)=\int_X\left(\int_{f^{-1}(z)}g_k(y)\,d\H^{n-m}(y)\right)d\H^m(z),
\end{equation}
and since the integrands of all the integrals in \eqref{eq:Coagk} converge with $k\to\infty$ pointwise, in an increasing way, to the corresponding integrands for $g$,  we obtain the desired formula \eqref{eq:gCoa} by the monotone convergence theorem. For that last step to be valid, we need to assert the measurability of the  integrands.

Since $X$ is $\H^m$-$\sigma$-finite, we have $\H^{m+1}(X)=0$. Thus, Corollary~\ref{IT1} gives $\rank \apmd (f,x)\leq m$ for $\H^n$-a.e. $x\in A$. Hence, by Lemma~\ref{AFT4}, $x\mapsto |J_m(\apmd(f,x))|$ is $\H^n$-measurable on $A$. Since $g$ is $\H^n$-measurable by assumption, the integrand on the left side of \eqref{eq:gCoa} is $\H^n$-measurable.

Now, for the right-hand side integrand of \eqref{eq:gCoa}, applying Lemma \ref{lem:measfib} to any Borel $A'\subset A$ and $f|_{A'}:A'\to X$ yields the $\H^m$-measurability of  $$z\mapsto \H^{n-m}(f^{-1}(z)\cap A')$$ (we also provide a slightly different proof within the proof of Proposition \ref{thm:Coa}). Thus
$$
z\mapsto \int_{f^{-1}(z)}g_k(y)\,d\H^{n-m}(y)
$$
is $\H^m$-measurable for any non-negative Borel simple function $g_k$. Let $g_k$ be, as before, the increasing sequence of simple functions converging to $g$. Then, as we already noted,
$$
\int_{f^{-1}(z)}g_k(y)\,d\H^{n-m}(y)~~\text{converge pointwise to~~}\int_{f^{-1}(z)}g(y)\,d\H^{n-m}(y),
$$
which proves the $\H^m$-measurability of
$$
z\mapsto \int_{f^{-1}(z)}g(y)\,d\H^{n-m}(y).
$$
\end{proof}

\begin{proof}[Proof of Proposition \ref{thm:Coa}.]\mbox{}\\
The idea of the proof is as follows: it is clear that both sides of the formula \eqref{eq:mCoa} are countably additive with respect to the set $A$ in the sense that if $A$ is decomposed into a countable union of measurable, pairwise disjoint sets $A=\bigcup_j A^j$ for which the fiber intersections $f^{-1}(z)\cap A^j$ are $\H^{n-m}$-measurable, and \eqref{eq:mCoa} holds for each of them in place of $A$, then it holds for $A$, as well. Thus, we begin by identifying negligible sets, for which both sides of \eqref{eq:mCoa} vanish:
\begin{itemize}
    \item[a)] the set where $\rank \apmd(f,\cdot)$ is less than $m$,
    \item[b)] certain $\H^n$-measure zero subsets of $A$, (e.g., the set where $f$ is not approximately metrically differentiable or where $\rank \apmd(f,x)>m$).
\end{itemize}
That the set a) is negligible, follows from the definition of $|J_m(\apmd(f,\cdot))|$ on the left-hand side and Corollary~\ref{CT4} on the right-hand side. As for b) it is obvious on the left-hand side, and requires the Co-area Inequality (Theorem~\ref{CT1}) on the right-hand side of \eqref{eq:mCoa}.

What we are left with is the set $\{x\in A~:~\rank \apmd(f,x)=m\}$, which, by the Metric Implicit Function Theorem II (Theorem~\ref{6:thm2}) decomposes, up to a set of $\H^n$ measure zero (again negligible), into a countable family of pairwise disjoint compact sets $K_j$, on which $f$ transforms, after a diffeomorphism change of variables, to a map depending only on a subset of variables, i.e., constant along the fibers. Then, the metric co-area formula \eqref{eq:mCoa} simplifies to Fubini's theorem (that allows us to integrate the $m$-Jacobian of $f$ along the fibers) and the change-of-variables formula resulting from the Area Formula (Corollary~\ref{gum14}) in the horizontal directions. Also, the fibers $f^{-1}(z)\cap K_j$ are compact and thus $\H^{n-m}$-measurable.

While this outline is short, details must be filled to make the proof rigorous.

Consider first the set $A_0=\{x\in A~~:~~\rank \apmd(f,x)<m\}$. Note that, by the definition of $|J_m(\cdot)|$,  $|J_m(\apmd(f,x))|=0$ whenever $\rank\apmd(f,x)<m$, and since $A_0=A\cap \mathrm{Crit_m} (f)$, by Corollary~\ref{CT4},
$\H^{n-m}(f^{-1}(z)\cap A_0)=0$ for $\H^m$ almost all $z\in X$.
Therefore, both sides of \eqref{eq:mCoa} vanish if we replace $A$ with $A_0$.

Let now $A_1$ be any subset of $A$ with $\H^n(A_1)=0$. Then obviously
$$\int_{A_1}|J_m(\apmd(f,x))|\,d\H^n(x)=0,$$ and by the Co-area Inequality,
$$\int^*_X \mathcal{H}^{n-m} \left(f^{-1}(z) \cap A_1\right)\,  d\mathcal{H}^m (z)  \leq \frac{\omega_{n-m} \omega_m}{\omega_{n}}L^m \,\mathcal{H}^{n}(A_1)=0,$$
so both sides of \eqref{eq:mCoa} vanish with $A_1$ in place of $A$.

We want to emphasize that here we need the Co-area Inequality for an $\H^m$-$\sigma$-finite target and in that case we provided a complete proof, see the proof of Theorem~\ref{CT2}.

This reasoning allows us in particular to neglect the sets
$$
\{x\in A~:~\apmd (f,x) \text{ does not exist }\}
\qquad
\text{and}
\qquad
\{x\in A~:~\rank \apmd(f,x)>m\},
$$
which by Theorem~\ref{thm:apKR} and Corollary~\ref{IT1} have measure zero (since $X$ is $\H^m$-$\sigma$-finite, $\H^{m+1}(X)=0$), and we are left with the set $$A_2=\{x\in A~:~\rank\apmd (f,x)=m\}.$$
If $\H^n(A_2)=0$, the set $A_2$ is negligible, both sides of \eqref{eq:mCoa} are zero and the proof is concluded. Assume thus that $\H^n(A_2)>0$.

Denote by $\pi:\bbbr^{n}=\bbbr^m\times\bbbr^{n-m}\to\bbbr^m$ the orthogonal projection onto the first $m$ coordinates. Theorem~\ref{6:thm2} provides us with a countable family of pairwise disjoint compact sets $K_i\subset A_2$, $i=1,2,\ldots$, such that $\H^n(A_2\setminus \bigcup_i K_i)=0$, and each of the $K=K_i$ has the following property:

There exist a $C^1$ diffeomorphism $G:\bbbr^n\to\bbbr^n=\bbbr^m\times\bbbr^{n-m}$ and a bi-Lipschitz $\phi:\pi(G(K)) \to X$ such that for all $(x,y)\in G(K)\subset \bbbr^m\times\bbbr^{n-m}$ we have $f\circ G^{-1}(x,y)=\phi(x)$. In other words, after a diffeomorphic coordinate change $G^{-1}$, the transformed mapping depends (in a bi-Lipschitz way) only on the first $m$ coordinates and it is constant on the fibers $G(K)\cap (\{x\}\times \bbbr^{n-m})$.

Since  $\H^n(A_2\setminus \bigcup_i K_i)=0$, also the set $Z:=A_2\setminus \bigcup_i K_i$ is negligible and it remains to verify \eqref{eq:mCoa} on $\bigcup_i K_i$. The formula \eqref{eq:mCoa} is countably additive with respect to a decomposition into a countable pairwise disjoint union of compact sets (compactness of $K_i$ guarantees compactness -- and thus $\H^{n-m}$-measurability -- of the fibers $f^{-1}(z)\cap K_i$), so it suffices now to verify that it holds on each compact piece $K=K_i$:
\begin{equation}\label{eq:coAonK}
\int_K |J_m(\apmd(f,x))|\,d\H^n(x)=\int_X \H^{n-m}(f^{-1}(z)\cap K)\,d\H^m(z).
\end{equation}
Fix $G$, $\phi$, and denote $\Psi:=G^{-1}$.
Recall that $f\circ\Psi(x,y)=\phi(x)$ for all $(x,y)\in G(K)$, and $\phi:\pi(G(K))\to X$ is bi-Lipschitz.

Set $h:=f\circ\Psi:G(K)\to X$. For $\H^n$-a.e. $q=(x,y)\in G(K)$, the point $q$ is a density point of $G(K)$, $\sigma:=\apmd(h,q)$ exists and has rank $m$, and
$\tau:=\apmd(\phi,x)$ exists. The last assertion follows from Theorem~\ref{thm:apKR} applied to $\phi$ and Fubini's theorem.

We claim that
$N_\sigma=\{0\}\times\bbbr^{n-m}$ and $\sigma|_{\bbbr^m\times\{0\}}=\tau.
$
To prove this, embed $X$ isometrically into $\ell^\infty$ and extend $h$ and $\phi$ to Lipschitz maps defined on $\bbbr^n$ and $\bbbr^m$, respectively. Since $h$ and $\phi$ are approximately metrically differentiable at $q$ and $x$, respectively, Proposition~\ref{ADT2} shows that the corresponding extensions are metrically differentiable there, with metric derivatives $\sigma$ and $\tau$.

Since $q$ is a density point of $G(K)$, for every fixed $a\in\bbbr^n$ and $t\to0$ we can choose $q_t=(x_t,y_t)\in G(K)$, $q_t=q+ta+o(|t|)$.

First take $a=(0,v)$. Then $x_t=x+o(|t|)$, and, since $h(x',y')=\phi(x')$ on $G(K)$ and $\phi$ is Lipschitz,
$d(h(q_t),h(q))=d(\phi(x_t),\phi(x))=o(|t|)$. On the other hand, metric differentiability of the extension of $h$ at $q$ gives
$$
d(h(q_t),h(q))=\sigma(q_t-q)+o(|q_t-q|)
=|t|\sigma(0,v)+o(|t|).
$$
Hence $\sigma(0,v)=0$ for every $v\in\bbbr^{n-m}$, so $
\{0\}\times\bbbr^{n-m}\subset N_\sigma$. Since $\rank\sigma=m$, both spaces have dimension $n-m$, and consequently $N_\sigma=\{0\}\times\bbbr^{n-m}$ and $
N_\sigma^\perp=\bbbr^m\times\{0\}$.

Next take $a=(u,0)$. Then $x_t=x+tu+o(|t|)$. Metric differentiability of the extensions of $h$ and $\phi$, together with $h(q_t)=\phi(x_t)$ and $h(q)=\phi(x)$, yields
$$
|t|\sigma(u,0)+o(|t|)
=d(h(q_t),h(q))
=d(\phi(x_t),\phi(x))
=|t|\tau(u)+o(|t|),
$$
so $\sigma(u,0)=\tau(u)$ for every $u\in\bbbr^m$. Hence
$\apmd(f\circ\Psi,(x,y))|_{\bbbr^m\times\{0\}}=\apmd(\phi,x)$ and 
$$
|J_m(\apmd(f\circ\Psi,(x,y)))|=|J_m(\apmd(\phi,x))|
\qquad
\text{for $\H^n$-a.e.\ }(x,y)\in G(K).
$$

By the chain rule (Lemma~\ref{lem:chain}), we know that $\apmd(f\circ\Psi,(x,y))=\apmd(f,\Psi(x,y))\circ D\Psi(x,y)$; if we set $\sigma:=\apmd(f,\Psi(x,y))$ and $L:=D\Psi(x,y)$, then, as observed above, $N_{\sigma\circ L}=\{0\}\times\bbbr^{n-m}$ and Lemma~\ref{lem:semi11} yields
\begin{equation*}
\begin{split}
&|J_m(\apmd(f,\Psi(x,y)))|\,|\det D\Psi(x,y)|\\
&=|J_m(\apmd(f\circ\Psi,(x,y)))|\,\big|\det (D\Psi(x,y)|_{\{0\}\times\bbbr^{n-m}})\big|\\
&=|J_m(\apmd(\phi,x))|\, \big|\det (D\Psi(x,y)|_{\{0\}\times\bbbr^{n-m}})\big|.
    \end{split}
\end{equation*}
for $\H^n$-a.e. $(x,y)\in G(K)$.

Thus, using the standard change of variables formula, and the Fubini theorem (in the last equality)
\[
\begin{split}
&
\int_K|J_m(\apmd(f,p))|\,d\H^n(p)
=
\int_{G(K)}|J_m(\apmd(f,\Psi(x,y)))|\,|\det D\Psi(x,y)|\,d\H^n(x,y)\\
&=
\int_{G(K)}|J_m(\apmd(\phi,x))|\, \big|\det (D\Psi(x,y)|_{\{0\}\times\bbbr^{n-m}})\big|\,d\H^n(x,y)\\
&=
\int_{\bbbr^m} |J_m(\apmd(\phi,x))|\left(\int_{G(K)\cap({\{x\}\times\bbbr^{n-m}})} \big|\det (D\Psi(x,y)|_{\{0\}\times\bbbr^{n-m}})\big|\,d\H^{n-m}(y)\right)\,d\H^m(x).
\end{split}
\]
Now, by the Area formula,
$$
\int_{G(K)\cap({\{x\}\times\bbbr^{n-m}})} \big|\det (D\Psi(x,y)|_{\{0\}\times\bbbr^{n-m}})\big|\,d\H^{n-m}(y)=
\H^{n-m}\big(\Psi(G(K)\cap (\{x\}\times \bbbr^{n-m}))\big)\, ,
$$
in particular the function $x\mapsto \H^{n-m}\big(\Psi(G(K)\cap (\{x\}\times \bbbr^{n-m}))\big)$ is, by Fubini's Theorem, $\H^m$-measurable.

(Note that here by $\det (D\Psi(x,y)|_{\{0\}\times\bbbr^{n-m}})$ we understand the $(n-m)$-dimensional determinant of the restriction of $D\Psi(x,y)$ to the vertical space ${\{0\}\times\bbbr^{n-m}}\approx \bbbr^{n-m}$.)

As sets,
$$
\Psi(G(K)\cap (\{x\}\times \bbbr^{n-m})) =
\begin{cases}
f^{-1}(\phi(x)) \cap {K} & \text{if } x\in \pi(G(K))\\
\varnothing &\text{if } x\in\bbbr^m\setminus \pi(G(K)),
\end{cases}
$$
so
\begin{equation*}
\begin{split}
\int_K|J_m(\apmd(f,p))|\,d\H^n(p)=
\int_{\pi(G(K))} \H^{n-m}(f^{-1}(\phi(x)) \cap {K})\, \big|J_m(\apmd(\phi,x))\big|\,d\H^{m}(x) \,,
\end{split}
\end{equation*}
and the function $x\mapsto \tilde{g}_K(x):=\H^{n-m}(f^{-1}(\phi(x)) \cap {K})$ is measurable, and so is
$$
g_K(z)=\H^{n-m}(f^{-1}(z) \cap {K})=\begin{cases} \tilde{g}_K\circ\phi^{-1}(z)&\text{ if }z\in f(K)\\
0& \text{ otherwise,}\end{cases}
$$
because $f(K)=\phi(\pi(G(K)))$ is compact and $\phi^{-1}$ is Lipschitz on $f(K)$.

Finally, by the change of variables formula (Corollary~\ref{gum14}), applied with ${\phi:\pi(G(K))\to X}$ in place of $f$ and $u(z)=g_K(z)=\H^{n-m}(f^{-1}(z)\cap K)$, we get
\[
\begin{split}
\int_K|J_m(\apmd(f,p))|\,d\H^n(p)
&=
\int_X u(z)N(\phi,\pi(G(K)),z)\,d\H^m(z)\\
&=
\int_X \H^{n-m}(f^{-1}(z)\cap K)\,d\H^m(z),
\end{split}
\]
since $\phi$ is bi-Lipschitz on $\pi(G(K))$ and thus $N(\phi,\pi(G(K)),z)=1$ for all $z\in\phi(\pi(G(K)))$, and whenever $z\not\in\phi(\pi(G(K)))$, $f^{-1}(z)\cap K=\varnothing$, so $\H^{n-m}(f^{-1}(z)\cap K)=0$.

This completes the proof of \eqref{eq:coAonK},  which, by the preceding remarks, completes the proof of the proposition,  except for the analysis of measurability of the integrands.

The measurability of $|J_m(\apmd(f,\cdot))|$ was already established in the course of the proof of Theorem~\ref{thm:gCoa}.
The measurability of the right-hand-side integrand has been addressed already in Lemma \ref{lem:measfib}. The arguments below are related, but slightly different.

Recall that we decomposed $A$ into disjoint sets $A=A_0\cup A_1\cup Z\cup\bigcup_i K_i$, where $A_0=A\cap\mathrm{Crit}_m f$ and $\H^n(A_1\cup Z)=0$. As we observed, by Corollary \ref{CT4} and the Co-area Inequality,
$$
\H^{n-m}(f^{-1}(z)\cap A_0)=\H^{n-m}(f^{-1}(z)\cap (A_1\cup Z))=0 \quad \text{for }\H^m\text{-a.e.\ }z\in X,
$$
thus $\H^{n-m}(f^{-1}(z)\cap A)=\sum_i \H^{n-m}(f^{-1}(z)\cap K_i)=\sum_i g_{K_i}(z)$ for $\H^m$-almost every $z$ in $X$. As we already noted in the course of the proof, for each $K_i$ the function $g_{K_i}$ is measurable, which completes the proof of measurability of 
$z\mapsto\H^{n-m}(f^{-1}(z)\cap A)$.
\end{proof}

The next example shows that the assumption that $X$ is $\H^m$-$\sigma$-finite is indeed necessary.
\begin{example}[c.f. {\cite[Example 6.1]{reichel}}]
\label{ex:4.11}
Let $C_1$ be a Cantor set in $[0,1]$ obtained by removing, in a standard ternary construction, one interval of length $\frac{1}{4}$, two intervals of length $\frac{1}{16}$, ..., $2^k$ intervals of length $4^{-(k+1)}$. Obviously, $\H^1(C_1)=\frac{1}{2}$. Removing intervals of the same length, again through a ternary construction, from the interval $[0,\frac{1}{2}]$, yields a zero $\H^1$-measure Cantor set $C_2$.

Let $\phi:\bbbr\to\bbbr$, $\phi(t)=\int_0^t \chi_{\bbbr\setminus C_1}(s)\, ds$. Then $\phi$ is $1$-Lipschitz, strictly increasing, maps $C_1$ onto $C_2$ and $\phi'(t)=1$ for all $t\not \in C_1$, $\phi'(t)=0$ for a.e.\ $t\in C_1$.

Fix now $A=C_1\times\bbbr\subset\bbbr^2$,  let $X=C_2\times \bbbr\subset \ell_2^\infty$, the plane endowed with the maximum norm $\|(x,y)\|_\infty=\max\{|x|,|y|\}$, and let $f:A\to X$ be given by the formula $f(x,y)=(\phi(x),y)$.

The space $X$ is not $\H^1$-$\sigma$-finite. Indeed, assume to the contrary that $X=\bigcup_{i=1}^\infty X_i$ with $\H^1(X_i)<\infty$. It easily follows from the Theorem \ref{CT1} (Coarea inequality) that $\H^0(X_i\cap (\bbbr\times\{y\}))<\infty$ for a.e. $y\in\bbbr$, i.e., $X_i$ meets a.e. horizontal line at a finite number of points, and thus $X$ would have to meet a.e. horizontal line at a countable set of points -- but this is not the case, as $X\cap (\bbbr\times\{y\})=C_2\times\{y\}$ is uncountable for any $y$.

For a.e.\ $(x,y)\in A$ we have $\frac{\partial}{\partial x}f(x,y)=(0,0)$,  $\frac{\partial}{\partial y}f(x,y)=(0,1)$, so $\rank \apmd f(x,y)=1$ a.e.\ in $A$ and $|J_1(\apmd(f,(x,y)))|=1$ a.e.\ in $A$.

Now, $\int_A |J_1(\apmd (f,(x,y)))|\,d\H^2=\int_{C_1\times\bbbr}d\H^2=\H^1(C_1)\cdot \H^1(\bbbr)=\infty$. On the other hand, if $(u,v)\in X$, then the fiber $f^{-1}(u,v)=\{(\phi^{-1}(u),v)\}$ consists of a single point, because $\phi$ is strictly increasing and $\phi(C_1)=C_2$, and $\H^1(f^{-1}(u,v)\cap A)=0$. Thus
$$
\int_X \H^1(f^{-1}(u,v)\cap A)\, d\H^1(u,v) =0
$$
and the formula \eqref{eq:mCoa} fails -- the conclusion of the Metric Co-area Formula does not hold in this case.
\end{example}

\section{Metric Differentiability Revisited}
\label{ap:sec3}

The purpose of this section is to provide better understanding and intuitions concerning the notion of metric and strong metric differentiability.

The principal tool to understand any kind of differentiability in normed spaces is the analysis of blow-up maps.
Let $(Y,\Vert\cdot\Vert)$ be any normed space, $\Omega\subset \bbbr^n$ an open set, $f:\Omega\to Y$.
\begin{definition}
    Let $p\in\Omega$ and set $R=\dist(p,\partial\Omega)$. For $0<r<R$ we define the blow-up maps of~$f$ at~$p$ as
    $$
    F_r(v):=\frac{f(p+rv)-f(p)}{r}\qquad \text{ for }v\in\bbbr^n,~~ |v|< R/r.
    $$
\end{definition}
\begin{remark}
The restriction $|v|< R/r$ is necessary to have $p+rv\in\Omega$. However, we are mostly interested in the limit behavior of $F_r$ as $r\to 0$, so for any fixed $p$ and any $v\in\bbbr^n$ the blow-up maps $F_r$ are well defined at $v$ for all $r$ small enough, i.e., $0<r<R/|v|$.

If $f$ admits an extension $\tilde f:\bbbr^n\to Y$, then one may
consider the blow-up maps
$$
\tilde F_r(v)=\frac{\tilde f(p+rv)-\tilde f(p)}{r},
\qquad v\in\bbbr^n .
$$
This is the case, for instance, for Lipschitz maps with values in
$\bbbr^m$ or $\ell^\infty$. One immediately checks that $F_r(v)=\tilde F_r(v)$ whenever $|v|<R/r$,
so the local limiting behaviour as $r\to 0$ of $\tilde{F}_r$ and $F_r$ is the same.
\end{remark}

\begin{remark}
\label{rem:md1}
It is well known that
\begin{itemize}
    \item[a)] $f$ is \emph{Gateaux differentiable} at $p$ (i.e., the directional derivative $D_vf(p)$ exists for all $v\in \bbbr^n$ and the map $v\mapsto D_vf(p)$ is linear) if and only if there is a linear map $L:\bbbr^n\to Y$ such that $F_r(v)\xrightarrow{~~r\to 0~~}Lv$, for every $v\in\bbbr^n$;
    \item[b)] $f$ is \emph{Fr\'echet differentiable} at $p$ if and only if there is a linear map $L:\bbbr^n\to Y$ such that $F_r(v)\xrightarrow{~~r\to 0~~}Lv$ and this convergence is locally uniform in $v$ (i.e., uniform on every compact subset of $\bbbr^n$),
\end{itemize}
and in both cases the linear map $L$ is the (Gateaux or Fr\'echet) differential of $f$ at $p$. Thus Gateaux differentiability corresponds to pointwise, and Fr\'echet -- to locally uniform convergence of $F_r$ as $r\to 0$.
\end{remark}
Our aim is to describe the metric and strong metric differentiability in similar terms. Since any separable metric space $X$ embeds in $\ell^\infty$, the case of $Y=\ell^\infty$ is of particular interest.
\begin{lemma}\label{lem:md l1.2}
 The map $f:\Omega\to Y$ is
\begin{itemize}
\item[a)]  \emph{metrically differentiable} at $p$ if and only if there is a seminorm $\sigma$ on $\bbbr^n$ such that $\Vert F_r(v)\Vert\xrightarrow{~~r\to 0~~}\sigma(v)$  uniformly in $v$ on the unit sphere $\{u\in\bbbr^n~:~|u|=1\}$;
\item[b)] \emph{strongly metrically differentiable} at $p$ if and only if  there is a seminorm $\sigma$ on $\bbbr^n$ such that $\Vert F_r(v)-F_r(w)\Vert\xrightarrow{~~r\to 0~~}\sigma(v-w)$ uniformly in $v,w$ on the unit ball $\{u\in\bbbr^n~:~|u|\leq 1\}$.
\end{itemize}
\end{lemma}

Recall that a normed space $(Y,\Vert\cdot\Vert)$ is {\em strictly convex} if for all $x,y\in Y$, $x\neq y$, with $\Vert x\Vert=\Vert y\Vert=1$, we have $\Vert tx+(1-t)y\Vert<1$ for all $t\in (0,1)$.

We find the following observation very instructive:

\begin{proposition}
\label{lem:md l1.3}
Assume $(Y,\Vert\cdot\Vert)$ is finite dimensional. If $f:\Omega\to (Y,\Vert\cdot\Vert)$ is strongly metrically differentiable at $p$, then from any sequence $r_k\to 0$ we can choose a subsequence (still denoted by $r_k$) such that $F_{r_k}$ converge with $k\to\infty$, locally uniformly, to a continuous map $F:\bbbr^n\to Y$, which is distance-preserving with respect to the pseudometric induced by $\md(f,p)$. If additionally either
\begin{itemize}
        \item[a)] $Y$ is strictly convex or
        \item[b)] $\rank \md(f,p)=\dim Y$,
    \end{itemize}
then the limit maps $F$ are linear.
\end{proposition}
\begin{corollary}
If $f:\bbbr^n\supset\Omega\to\bbbr^m$ is strongly metrically differentiable at $p$, then $\md(f,p)(v)=|F(v)|$ for some linear map $F:\bbbr^n\to\bbbr^m$ (i.e., the seminorm $\md(f,p)$ is a pullback of the Euclidean norm by a linear map).
\end{corollary}

\begin{proof}[Proof of Lemma~\ref{lem:md l1.2}]
For a), assume first that $f$ is metrically differentiable at $p$ and denote $\sigma=\md(f,p)$. Then for $x\in\Omega$ we have $\Vert f(x)-f(p)\Vert=\sigma(x-p)+o(|x-p|)$, thus for any $v\in\bbbr^n$ with $|v|=1$
\begin{equation*}
\Vert F_r(v)\Vert=\frac{\Vert f(p+rv)-f(p)\Vert}{r}=\frac{\sigma(rv)+o(r)}{r}=\sigma(v)+\frac{o(r)}{r}\xrightarrow{~r\to 0~}\sigma(v)
\end{equation*}
and the convergence is uniform in $v$.

Assume now that $\|F_r(v)\|$ converges with $r\to 0$ uniformly to a seminorm $\sigma(v)$ on the unit sphere $\{|v|=1\}$. Then for $x\in\Omega$ we set $r=|x-p|$, $v=(x-p)/|x-p|$; an easy calculation shows that
$$
\frac{\Vert f(x)-f(p)\Vert-\sigma(x-p)}{|x-p|}=\Vert F_r(v)\Vert-\sigma(v)
$$
and the right hand side converges to 0 uniformly in $v$, by assumption, thus $f$ is metrically differentiable at $p$.

For b), the arguments are very similar. Assume that $f$ is strongly metrically differentiable at $p$, with $\sigma=\md(f,p)$, that is for $x,y\in\Omega$ we have $\Vert f(x)-f(y)\Vert =\sigma(x-y)+o(|x-p|+|y-p|)$.
Then for any $v,w\in \bbbr^n$, $|v|, |w|\in [0,1]$, we have
\begin{equation*}
    \begin{split}
\Vert F_r(v)-F_r(w)\Vert &=\frac{\Vert f(p+rv)-f(p+rw)\Vert}{r}=\frac{\sigma(rv-rw)+o(|rv|+|rw|)}{r}\\
&=\sigma(v-w)+\frac{o(r(|v|+|w|))}{r}\xrightarrow{~~r\to 0~~}    \sigma(v-w)
    \end{split}
\end{equation*}

and since $|v|+|w|$ is bounded by 2, the convergence is uniform in $v,w$.

For the converse, let $x,y\to p$ and set
$$
    r=\max\{|x-p|,|y-p|\},\qquad v=\frac{x-p}{r},\qquad w=\frac{y-p}{r}.
$$
Then $|v|,|w|\leq 1$ and
$$
\Vert F_r(v)-F_r(w)\Vert-\sigma(v-w)=\frac{\Vert f(x)-f(y)\Vert-\sigma(x-y)}{r}.
$$
Since $r\le |x-p|+|y-p|\leq 2r$,
the uniform convergence in $v,w\in \bar{B}(0,1)$ implies
$$
\frac{\Vert f(x)-f(y)\Vert-\sigma(x-y)}{|x-p|+|y-p|} \to 0,
$$
so $f$ is strongly metrically differentiable at $p$.
\end{proof}

\begin{proof}[Proof of Proposition~\ref{lem:md l1.3}]
Let us write, for short, $\sigma=\md(f,p)$.

Fix $M>0$. For all sufficiently small $r>0$, the maps $F_r$ are defined on $B(0,M)$. Then by the strong metric differentiability of $f$,
$$
\Vert F_r(v)\Vert=\frac{\Vert f(p+rv)-f(p)\Vert}{|r|}=\frac{\sigma(rv)+o(r|v|)}{r}=\sigma(v)+\frac{o(r|v|)}{r}
$$
and the right hand side is bounded, for $v\in B(0,M)$, independently of $v$, so the family $\{F_r\}$ is uniformly bounded on $B(0,M)$. Similarly, using
b), Lemma~\ref{lem:md l1.2}, and the scaling property of the blow-up map:
$$
F_r(v)=MF_{rM}\Big(\frac{v}{M}\Big),
$$
we get immediately that $\Vert F_r(v)-F_r(w)\Vert\to \sigma(v-w)$ uniformly for $v,w\in B(0,M)$.

Let $D=\mathbb Q^n$ and fix any sequence $r_j\to 0$. For each $v\in D$, the sequence $F_{r_j}(v)$ is defined and bounded for all sufficiently large $j$. Since $Y$ is finite dimensional, a diagonal argument gives a subsequence, again denoted by $r_j$, such that $F_{r_j}(v)$ converges for every $v\in D$. We claim that this subsequence converges locally uniformly.

Fix $M>0$ and $\eps>0$. By the uniform convergence above, for all sufficiently large $j$ we have
$$
\Vert F_{r_j}(v)-F_{r_j}(w)\Vert\leq\sigma(v-w)+\eps
\quad\text{for all $v,w\in B(0,M)$.}
$$
Since $\sigma$ is continuous, we can choose finitely many points $v_1,\ldots,v_N\in D\cap B(0,M)$ such that for every $v\in B(0,M)$ there is $v_i$ with $\sigma(v-v_i)<\eps$. The convergence at these finitely many points implies that, for all sufficiently large $j,k$,
$$
\max_{1\leq i\leq N}\Vert F_{r_j}(v_i)-F_{r_k}(v_i)\Vert<\eps.
$$
Thus, for $v\in B(0,M)$ and $v_i$ chosen as above,
\begin{equation*}
\begin{split}
\Vert F_{r_j}(v)-F_{r_k}(v)\Vert
&\leq\Vert F_{r_j}(v)-F_{r_j}(v_i)\Vert
+\Vert F_{r_j}(v_i)-F_{r_k}(v_i)\Vert\\
&\qquad+\Vert F_{r_k}(v_i)-F_{r_k}(v)\Vert
<5\eps.
\end{split}
\end{equation*}
Hence the sequence is uniformly Cauchy on $B(0,M)$. Since $Y$ is complete and $M>0$ was arbitrary, $F_{r_j}$ converges locally uniformly to a map $F:\bbbr^n\to Y$.

Again by b), Lemma~\ref{lem:md l1.2}, for any fixed $v,w\in\bbbr^n$,
$\Vert F_{r_j}(v)-F_{r_j}(w)\Vert$ converges to $\md(f,p)(v-w)$, so passing with $r_j\to 0$ we get
$\Vert F(v)-F(w)\Vert=\md(f,p)(v-w)$. In particular, $F$ is continuous. This means that $F$ is a (pseudo)isometric `embedding' of $(\bbbr^n,\md(f,p))$ into $(Y,\Vert\cdot\Vert)$ -- it preserves the seminorm distance. Note, however, that $F$ need not be injective (thus the quotation marks in `embedding'), and it becomes injective only after factoring it through $\ker \md(f,p)$.

It remains to prove that $F$ is linear. This is essentially a variant of the Mazur-Ulam Theorem (\cite{MazurUlam, Vaisala}, see also \cite[Theorem 14.1.3]{AlbiacKalton}), stating that any surjective isometry between normed spaces is affine, and if the target space is strictly convex, the surjectivity assumption can be dropped.

The proof in the case a), when the target space $Y$ is strictly convex (e.g. when $Y=\bbbr^m$ equipped with the standard Euclidean norm) is easy, so we include it here for completeness.

Obviously, $F(0)=0$, since $F_r(0)=0$ for all $r$.
Assume $u,v\in (\bbbr^n,\sigma)$, $t\in \bbbr$ and $z=tu+(1-t)v$. It is an easy exercise to show that to check linearity of $F$ it suffices to show that $F(z)=tF(u)+(1-t)F(v)$ when $t\in[0,1]$.

Denote $r:=\sigma(u-v)$, then $\sigma(z-u)=(1-t)\sigma(u-v)=(1-t)r$, thus $z\in \bar{B}_\sigma(u,(1-t)r)\subset (\bbbr^n,\sigma)$. In the same way we show that $z\in\bar{B}_\sigma(v,tr)\subset (\bbbr^n,\sigma)$. Since $F$ is a pseudoisometry, it maps any ball $\bar{B}_\sigma(w,s)\subset (\bbbr^n,\sigma)$ into the corresponding ball $\overline{B}_Y(F(w),s)$. In particular,
$$
F(\bar{B}_\sigma(u,(1-t)r))\subset\bar{B}_Y(F(u),(1-t)r), \quad F(\bar{B}_\sigma(v,tr))\subset\bar{B}_Y(F(v),tr),
$$
and thus $F(z)\in \bar{B}_Y(F(u),(1-t)r)\cap \bar{B}_Y(F(v),tr)$. Note, however, that $\Vert F(u)-F(v)\Vert =\sigma(u-v)=r=(1-t)r+tr$, so these two balls in $Y$, which is strictly convex, intersect in only one point, namely in $tF(u)+(1-t)F(v)$. Thus $F(z)=tF(u)+(1-t)F(v)$, as desired.

For the case b), one easily checks that $F$ factorizes through $N:=\ker \sigma$: if $v=v'+v''$, $v'\in N$, $v''\in N^\perp$, then $\Vert F(v)-F(v'')\Vert=\sigma(v-v'')=\sigma(v')=0$. Then $\sigma$ is a norm on $N^\perp$ and $F|_{N^\perp}:(N^\perp,\sigma)\to (Y,\Vert\cdot\Vert)$ is an isometric embedding between normed spaces of the same dimension $\dim Y=\dim N^\perp=\rank \sigma$. Then $F|_{N^\perp}$ is continuous and injective, so by the Brouwer's invariance of domain theorem, its image in $Y$ is open. On the other hand, one easily checks that the image of $F|_{N^\perp}$ is closed. Indeed, assume that for some sequence  $(x_k)$ in $N^\perp$ and $y\in Y$ we have $F(x_k)\to y$ as $k\to\infty$. Since the sequence $(F(x_k))$ is bounded and $F$ is an isometric embedding, also the sequence $(x_k)$ is bounded in $(N^\perp,\sigma)$.  By passing to a subsequence, we may assume that $x_k\to x$, then by continuity $F(x_k)\to F(x)$, so $y=F(x)\in F(N^\perp)$. Thus the image of $F|_{N^\perp}$, being both open and closed, must equal the whole $Y$, and $F|_{N^\perp}$ is surjective. Finally, we invoke the Mazur-Ulam Theorem, which states that $F|_{N^\perp}$ is affine, and, since $F(0)=0$, linear. Hence $F$ is linear too.
\end{proof}

Next, we give two explicit examples of Lipschitz maps $f: \mathbb R^1 \supset (-\delta,+\delta) \to \mathbb R^2$ that illustrate the notion of strong metric differentiability and its relation to Fr\'echet differentiability.

Let us shortly explain the idea behind the examples. In both cases we construct maps that behave essentially like $t\mapsto (t,0)$ near $t=0$, so that the expected metric derivative is $\md(f,0)(t)=|t|$.

Suppose we have $f(t)=t u(t)$ (on a neighborhood of $t=0$) where $|u(t)|=1$. Then, $f$ is metrically differentiable at $t=0$, with $\md(f,0)(t)=|t|$, regardless of the behavior of $u(t)$, whereas Fr\'echet differentiability of $f$ depends on the behavior of $u(t)$ as $t \to 0$. We can even take a discontinuous $u(t)$ and still retain metric differentiability, e.g.\ $f(t)=(|t|,0)$ -- obviously $f$ is neither Fr\'echet nor even Gateaux differentiable.

Obviously, if the curve $f$ spirals infinitely around the origin with $t\to 0$, it cannot be Fr\'echet nor Gateaux differentiable at $t=0$.  It turns out, however, that strong metric differentiability allows for infinite rotations, as long as they are very slow, at smaller and smaller scales.

Our first example shows that
neither metric, nor strong metric differentiability properties are stable under addition, even for Lipschitz maps.
\begin{example}
There exist Lipschitz $f,g:(-e^{-e},e^{-e})\to\bbbr^2$ such that $f$ and $g$ are strongly metrically differentiable at $t=0$, but $f+g$ is not even metrically differentiable at $t=0$.

Let $f$ be given by the formula $f(t)=(t,0)$. Since $f$ is smooth, it is strongly metrically differentiable, with $\md(f,0)(t)=|f'(0)t|=|(t,0)|=|t|$.
    The construction of $g$ is more involved.

    Let $\theta(t)=\frac{\pi}{2}(1+\sin(\log(\log t^{-1})))$,
    $$
    u(s)=\begin{cases}\left(\cos\theta(|s|),\sin \theta(|s|)\right) & s\neq 0,\\
    (1,0) &s=0,\end{cases}
    $$
    and finally set $g(t)=\int_0^t u(s)\,ds$. Obviously, $u(t)$ is bounded, so its primitive $g$ is Lipschitz.

Showing that $g$ is strongly metrically differentiable at $t=0$, with $\md(g,0)(t)=|t|$, is a tedious exercise.  Careful analysis of the blow-up $G_r(t)=g(rt)/r$ shows that if we set $L_r(t)=tu(r)$, then $|G_r(t)-L_r(t)|\xrightarrow{r\to 0}0$ uniformly in $t$, for $|t|\leq 1$, thus the blow-ups $G_r(t)$ converge uniformly to the line segment $L_r(t)$, parameterized by arc length. In particular, $|G_r(t)-G_r(s)|$ is for small $r$ uniformly close to $|L_r(t)-L_r(s)|=|t-s|$. However, the direction $u(r)$ of $L_r(t)$ oscillates with $r\to 0$ and so does the length of the blow-up $te_1+G_r(t)$ of $f+g$, which for small $r$ is close to $te_1+L_r(t)=t(e_1+u(r))$. By a), Lemma \ref{lem:md l1.2}, $f+g$ cannot be metrically differentiable at $t=0$.
\end{example}

The next example shows that strong metric differentiability is not enough to guarantee Fr\'echet differentiability.
\begin{example}
The map
    $$
f:(-0.5,0.5)\to\bbbr^2,
\qquad
f(t)=\begin{cases}
    \left(t\cos\sqrt{-\ln |t|}, t\sin\sqrt{-\ln |t|}\right)& t\neq 0,\\
    (0,0) &t=0.
\end{cases}
$$
is Lipschitz, strongly metrically differentiable at $t=0$, but not Fr\'echet differentiable.
\end{example}
\begin{proof}
Indeed, $\lim_{t\to 0}(f(t)-f(0))/t=\lim_{t\to 0}(\cos\sqrt{-\ln |t|}, \sin\sqrt{-\ln |t|})$ does not exist, so $f$ is not (Fr\'echet) differentiable at $t=0$. However, if we write $\alpha_r(t)=\sqrt{-\ln(r|t|)}$, the blow-up map can be written as
$$
F_r(t)=\frac{f(rt)}{r}=t(\cos\alpha_r(t),\sin\alpha_r(t)),
$$
$F_r(0)=0$,
and by elementary calculation
$$
|F_r(t)-F_r(s)|^2=t^2+s^2-2ts\cos(\alpha_r(t)-\alpha_r(s))\xrightarrow{r\to 0}t^2+s^2-2ts=(t-s)^2,
$$
so for fixed $t,s$ we have $|F_r(t)-F_r(s)|\xrightarrow{r\to 0}|t-s|$. One can check that for $r\in (0,\frac{1}{2})$, we have the uniform estimate on the unit ball ($|s|,|t|\leq 1$):
$$
\big||F_r(t)-F_r(s)|-|t-s|\big|\leq \frac{1}{e\sqrt{-\ln r}}.
$$
 This, by b), Lemma \ref{lem:md l1.2}, proves that $f$ is indeed strongly metrically differentiable at $t=0$, with ${\md(f,0)(t)=|t|}$.


To see that $f$ is Lipschitz, it suffices to check that $|f'(t)|^2=1-\frac{1}{4\ln|t|}$ for all $t\neq 0$, so it is bounded for $t\in (-0.5,0.5)$.
\end{proof}

Finally, we prove the following proposition, see Proposition~\ref{MDT1}.

\begin{proposition}
\label{1:prop1}
Assume $f:\bbbr^n\to \bbbr^m$ is strongly metrically differentiable at ${p\in\bbbr^n}$. Assume moreover that all the partial derivatives $\partial_{x_i} f(p)$ exist. Then $f$ is Fr\'echet differentiable at $p$.
\end{proposition}

\begin{proof}
By Proposition~\ref{lem:md l1.3}, whenever the blow-up maps $F_r(\cdot)$ of $f$ at $p$ converge to some $F:\bbbr^n\to\bbbr^m$ on a sequence $r_j\to 0$, the limit map $F$ is linear -- and the convergence $F_{r_j}(v)\to F(v)$ is locally uniform in $v$. However,
$$
F_r(e_i)=\frac{f(p+re_i)-f(p)}{r}\xrightarrow{~~r\to 0~~}\partial_{x_i}f(p) \qquad \text{for }i=1,2,\ldots,n
$$
so $F$, regardless of the sequence $r_j$, is the same linear map, defined by $F(e_i)=\partial_{x_i}f(p)$. Since any convergent sequence $F_{r_j}(\cdot)$ has the same limit $F$, we have $\lim_{r\to 0} F_r(v)=F(v)$, and the convergence is locally uniform in $v$. Thus, by b), Remark~\ref{rem:md1}, $f$ is Fr\'echet differentiable at $p$.
\end{proof}

\section{Appendix}
\label{APPE}

\subsection{Palais extension theorem}
\label{PET}
The beautiful extension results presented in this section, Lemma~\ref{Pal1} and Theorem~\ref{Pal2}, are due to Palais~\cite{palais}. For further details, generalizations, and related references, we refer the reader to \cite{GGH}.

By a diffeomorphism $H:\bbbr^n\supset\Omega\to\bbbr^n$ we mean a diffeomorphism onto the image i.e.
diffeomorphism $H:\bbbr^n\supset\Omega\to H(\Omega)\subset\bbbr^n$. We say that $H$ is orientation preserving if $\det DH>0$ in $\Omega$.

When we say that
$H:\bar{B}(0,\rho)\to \bbbr^n$
is a $C^k$-diffeomorphism, we mean that $H$ is the restriction of a genuine
$C^k$-diffeomorphism defined on an open neighborhood of the closed ball.

\begin{lemma}
\label{Pal1}
Let $\bar{B}(0,r)\subset \mathbb R^n$, and suppose that
$H:\bar{B}(0,r)\to \mathbb R^n$ is an orientation preserving
$C^k$-diffeomorphism, $k\in\bbbn\cup\{\infty\}$, with $H(0)=0$. Then for every $\varepsilon>0$ there
is a $C^k$-diffeomorphism $\widetilde H:\mathbb R^n\to\mathbb R^n$ such that
\[
\widetilde H(x)=
\begin{cases}
H(x), & x\in \bar{B}(0,r),\\
x, & \operatorname{dist}(x,A)\geq \varepsilon,
\end{cases}
\qquad
\text{where }
A=\bar{B}(0,r)\cup H(\bar{B}(0,r)).
\]
\end{lemma}
For a proof of this lemma, see \cite[Lemma~2.2]{GGH}. The proof is elementary, but it is based on a very beautiful argument of Palais \cite{palais}.
We will show here how to use it to prove the following result of Palais \cite{palais}. See \cite{GGH} for more general results. While Theorem~\ref{Pal2} is a corollary of \cite[Theorem~1.4]{GGH} we show here a direct argument how to conclude it from Lemma~\ref{Pal1}.
\begin{theorem}[Palais]
\label{Pal2}
Let $k\in\bbbn\cup\{\infty\}$ and let $H:\bar{B}(x,r)\to \mathbb R^n$ be a $C^k$-diffeomorphism. Then there is a $C^k$-diffeomorphism $\widetilde H:\mathbb R^n\to\mathbb R^n$ such that
$\widetilde H=H \quad \text{on } \bar{B}(x,r)$.
Moreover, if $H$ is orientation preserving, we can find $\widetilde H$ that is identity outside a compact set.
\end{theorem}
\begin{proof}
First assume that $H$ is orientation preserving. Set
$h(u)=H(x+u)-H(x)$, for $u\in \bar B(0,r)$.
Then $h:\bar B(0,r)\to\mathbb R^n$ is an orientation preserving $C^k$-diffeomorphism and $h(0)=0$. By Lemma~\ref{Pal1}, there is a $C^k$-diffeomorphism $\widehat h:\mathbb R^n\to\mathbb R^n$ such that
$\widehat h=h$ on  $B(0,r)$,
and $\widehat h$ is the identity outside a compact set. Define
\[
F(z)=H(x)+\widehat h(z-x), \qquad z\in\mathbb R^n.
\]
Then $F$ is a $C^k$-diffeomorphism of $\mathbb R^n$ and $F=H$ on $\bar B(x,r)$. However, $F$ is not necessarily the identity outside a compact set. Indeed, if
$a=H(x)-x$,
then $F(z)=z+a$ for all $z$ outside a sufficiently large compact set. We now correct this translation at infinity without changing $F$ on $\bar B(x,r)$.

Let $K=H(\bar B(x,r))$. Choose $R>0$ so large that $K\subset B(0,R)$, and let $\eta:\mathbb R^n\to[0,1]$ be a smooth function such that
\[
\eta=0 \quad \text{on } B(0,R),
\qquad
\eta=1 \quad \text{on } \mathbb R^n\setminus B(0,R+1).
\]
Consider the smooth bounded vector field
\[
V(y)=-\eta(y)a.
\]
Let $\Phi_t$ be its flow. Since $V$ is bounded, the flow exists for all $t\in\mathbb R$, and $\Phi:=\Phi_1$ is a $C^\infty$-diffeomorphism of $\mathbb R^n$. Because $V=0$ on $B(0,R)$, we have
\[
\Phi(y)=y \quad \text{for } y\in B(0,R),
\]
and hence $\Phi$ is the identity on a neighborhood of $K$. On the other hand $V(y)=-a$ for $|y|>R+1$, and hence $\Phi(y)=y-a$ if $|y|>R+1+|a|$.
Now define
$\widetilde H=\Phi\circ F$.
Since $F(\bar B(x,r))=K$ and $\Phi$ is the identity on a neighborhood of $K$, we get
\[
\widetilde H=F=H \quad \text{on } \bar B(x,r),
\]
and for $z$ outside a sufficiently large compact set
\[
\widetilde H(z)=\Phi(F(z))=\Phi(z+a)=z+a-a=z.
\]
Thus, in the orientation preserving case, $\widetilde H$ can be chosen to be the identity outside a compact set.

It remains to treat the case when $H$ is orientation reversing. Let $R:\mathbb R^n\to\mathbb R^n$ be an orientation reversing linear isometry, for instance
\[
R(y_1,y_2,\ldots,y_n)=(-y_1,y_2,\ldots,y_n).
\]
Then $R\circ H$ is orientation preserving. By the orientation preserving case, there is a $C^k$-diffeomorphism $G:\mathbb R^n\to\mathbb R^n$ such that
$G=R\circ H \quad \text{on } \bar B(x,r)$.
Hence $\widetilde H=R^{-1}\circ G$ is a $C^k$-diffeomorphism of $\mathbb R^n$ and satisfies $\widetilde H=H$ on $\bar B(x,r)$. This proves the theorem.
\end{proof}

\subsection{Linear algebra}
We will identify linear maps $L:\mathbb{R}^{n}\to \mathbb{R}^{m}$ with
their matrix representations in the canonical bases of $\mathbb{R}^{n}$
and $\mathbb{R}^{m}$.

A linear map $U:\mathbb{R}^{n}\to \mathbb{R}^{m}$, $m\geq n$, is said
to be \emph{orthogonal} if it is an isometry onto the image
$U(\mathbb{R}^{n})\subset \mathbb{R}^{m}$, that is, if it preserves
lengths of vectors. This is equivalent to the condition that $U$ maps
the canonical orthonormal basis of $\mathbb{R}^{n}$ onto an orthonormal
set of vectors. Since the columns of the matrix $U$ are the images of
the canonical basis of $\mathbb{R}^{n}$, we get the following lemma.

\begin{lemma}
\label{AT9}
A linear map $U:\mathbb{R}^{n}\to \mathbb{R}^{m}$, $m\geq n$, is
orthogonal if and only if
$U^{T}U=I_{n}$.
\end{lemma}

\begin{lemma}
\label{AT10}
If $U:\mathbb{R}^{n}\to\mathbb{R}^{m}$, $m\geq n$, is orthogonal, then
$UU^{T}:\mathbb{R}^{m}\to\mathbb{R}^{m}$
is the orthogonal projection onto $U(\mathbb{R}^{n})\subset\mathbb{R}^{m}$.
\end{lemma}
\begin{proof}
It suffices to show that $UU^{T}$ fixes every vector in
$U(\mathbb{R}^{n})$ and vanishes on $U(\mathbb{R}^{n})^{\perp}$.

If $v\in U(\mathbb{R}^{n})$, then $v=Uu$ for some $u\in\mathbb{R}^{n}$.
Since $U$ is orthogonal, $U^{T}U=I_{n}$. Hence
\[
UU^{T}v=UU^{T}Uu=Uu=v.
\]

Now suppose that $v\in U(\mathbb{R}^{n})^{\perp}$. Then,
\[
\langle U^{T}v,u\rangle
=
\langle v,Uu\rangle
=
0
\qquad
\text{for every } u\in\bbbr^n.
\]
Therefore $U^{T}v=0$, and hence
$UU^{T}v=0$.
\end{proof}

We say that a linear map $L:\mathbb{R}^{n}\to \mathbb{R}^{n}$ is
\emph{symmetric} if $L=L^{T}$.

The next standard result in linear algebra describes the structure of any
linear map $L:\mathbb{R}^{n}\to \mathbb{R}^{m}$, $m\geq n$. For a proof
see, for example, \cite[Theorem~3.5]{evans}.

\begin{theorem}
\label{AT11}
Let $L:\mathbb{R}^{n}\to \mathbb{R}^{m}$, $m\geq n$, be linear. Then
there is a symmetric map $S:\mathbb{R}^{n}\to \mathbb{R}^{n}$ and an
orthogonal map $U:\mathbb{R}^{n}\to \mathbb{R}^{m}$ such that
$L=US$.
\end{theorem}

Recall that the Lebesgue measure coincides with $\H^n$. For a proof see Corollary~\ref{SJT2}.
If $L:\mathbb{R}^{n}\to \mathbb{R}^{n}$ is linear, then for any
measurable set $A\subset \mathbb{R}^{n}$,
\[
 |L(A)|=|\det L|\, |A|,
\qquad
\text{or equivalently,}
\qquad
\mathcal{H}^{n}(L(A))=|\det L|\,\mathcal{H}^{n}(A).
\]

Since orthogonal maps, being isometries, preserve measure, if
$L=US:\mathbb{R}^{n}\to \mathbb{R}^{m}$, $m\geq n$, is as in the
theorem, then
\[
\mathcal{H}^{n}(L(A))
=
\mathcal{H}^{n}(S(A))
=
|\det S|\,\mathcal{H}^{n}(A).
\]
Observe that
\[
\det(L^{T}L)
=
\det(S^{T}U^{T}US)
=
\det(S^{T}S)
=
|\det S|^{2}.
\]
Hence we have the following lemma.

\begin{lemma}
\label{AT8}
If $L:\mathbb{R}^{n}\to \mathbb{R}^{m}$, $m\geq n$, is a linear map,
then for any Lebesgue
measurable set $A\subset \mathbb{R}^{n}$,
\[
\mathcal{H}^{n}(L(A))
=
\sqrt{\det(L^{T}L)}\,\mathcal{H}^{n}(A).
\]
\end{lemma}
This result explains why, for mappings
$f:\mathbb{R}^{n}\to \mathbb{R}^{m}$, $m\geq n$, we define the
$n$-dimensional Jacobian as (cf.\ \eqref{INeq2})
\[
|J_{n}f(x)|
=
\sqrt{\det\big(Df(x)^{T}Df(x)\big)}.
\]
Indeed, Lemma~\ref{AT8} shows that
$|J_{n}f(x)|$ is the factor by which the linear map $Df(x)$ changes measure:
\[
\mathcal{H}^{n}\bigl(Df(x)(A)\bigr)
=
|J_{n}f(x)|\,\mathcal{H}^{n}(A).
\]

Now assume that $L:\mathbb{R}^{n}\to\mathbb{R}^{m}$, $m\leq n$, is
linear and $\operatorname{rank} L=m$. Then
$$
L|_{(\ker L)^{\perp}}:(\ker L)^{\perp}\to\mathbb{R}^{m}
$$
is an isomorphism, and we want to find out how this map changes measure between
these two $m$-dimensional spaces.

Let $U:\mathbb{R}^{m}\to(\ker L)^{\perp}$ be an orthogonal map, that is,
an isometry. According to Lemma~\ref{AT10},
$UU^{T}:\mathbb{R}^{n}\to\mathbb{R}^{n}$ is the orthogonal projection onto
$(\ker L)^{\perp}$. Since $L$ vanishes on $\ker L$, this implies
$L=LUU^{T}$. Therefore
$$
LL^{T}=LUU^{T}L^{T}=LU(LU)^{T},
\text{ and hence }
\det(LL^{T})=\det\bigl(LU(LU)^{T}\bigr)=|\det(LU)|^{2}.
$$

If $A\subset(\ker L)^{\perp}$ is measurable, then $A=U(B)$ for some
measurable set $B\subset\mathbb{R}^{m}$. Since $U$ is an isometry, it
preserves $m$-dimensional measure, and hence
$\mathcal{H}^{m}(A)=\mathcal{H}^{m}(B)$. Therefore
$$
\mathcal{H}^{m}(L(A))
=
\mathcal{H}^{m}(LU(B))
=
|\det(LU)|\,\mathcal{H}^{m}(B)
=
\sqrt{\det(LL^{T})}\,\mathcal{H}^{m}(A).
$$
We have proved the following lemma.

\begin{lemma}
If $L:\mathbb{R}^{n}\to\mathbb{R}^{m}$, $m\leq n$, is linear and
$\operatorname{rank} L=m$, then for any measurable set
$A\subset(\ker L)^{\perp}$ we have
$$
\mathcal{H}^{m}(L(A))
=
\sqrt{\det(LL^{T})}\,\mathcal{H}^{m}(A).
$$
\end{lemma}

This result also explains why for maps $f:\bbbr^n\to\bbbr^m$, $m\leq n$ we define the Jacobian
as (cf.\ \eqref{INeq3}):
$$
|J_mf|(x)=\sqrt{\det \big(Df(x)Df(x)^T\big)}\, .
$$
Namely, if $\rank Df(x)=m$, $Df(x)$ is an isomorphism of $\big(\ker Df(x)\big)^\perp$ onto $\bbbr^m$ and the above formula shows how $Df(x)$ changes the measure between sets in $\big(\ker Df(x)\big)^\perp$ and their images in $\bbbr^m$.
If $\rank Df(x)<m$ we simply set $|J_mf|(x)=0$.

\subsection{Measure theory and covering lemmata}
We will need the following less standard version of Lusin's theorem. Since this precise formulation is not easy to locate in the literature, we include a proof for the reader's convenience. The proof uses an argument from \cite{feldman}.
\begin{theorem}[Lusin]
\label{TH5}
Let $X$ be a metric space and $\mu$ a Borel measure on $X$ such that $X$ is a union of countably many open sets of finite measure. If $\Phi:X\to Y$ is a Borel measurable mapping into a separable metric space $Y$, then for any $\eps>0$, there is a closed set $F\subset X$ such that $\mu(X\setminus F)<\eps$ and $\Phi|_F:F\to Y$ is continuous.
\end{theorem}
\begin{proof}
By $\mathcal{B}(X)$ we will denote the $\sigma$-algebra of Borel sets in $X$.
We first prove the regularity fact that will be used below: {\em If $A\in\mathcal{B}(X)$
and $\eta>0$, then there is an open set $V\subset X$ such that
$A\subset V$ and $\mu(V\setminus A)<\eta$.}

First suppose that $\mu(X)<\infty$. Let $\mathcal R$ be the family of all Borel
sets $A\subset X$ such that for every $\eta>0$ there are a closed set $C$ and an
open set $G$ satisfying $C\subset A\subset G$ and
\[
\mu(G\setminus C)<\eta.
\]
We prove that $\mathcal R=\mathcal B(X)$.
If $A$ is closed, set
$G_k=\{x\in X:\operatorname{dist}(x,A)<1/k\}$. Then $\bigcap_k G_k=A$.
Since $\mu(G_1)<\infty$, continuity of measure from above gives
$\mu(G_k\setminus A)\to 0$. Hence every closed set belongs to $\mathcal R$.
It is easy to check that the class $\mathcal R$ is closed under complements.
It remains to check that $\mathcal R$ is closed under countable unions. Let
$A_i\in\mathcal R$ and $A=\bigcup_i A_i$. Choose closed $C_i$ and open $G_i$
such that $C_i\subset A_i\subset G_i$ and
\[
\mu(G_i\setminus C_i)<\eta 2^{-i-3}.
\]
Put $G:=\bigcup_i G_i$. Since $\mu(X)<\infty$, by continuity from below we can
choose $N$ so that $B_N:=\bigcup_{i=1}^N A_i$ satisfies
$\mu(A\setminus B_N)<\eta/2$.
Let $C:=\bigcup_{i=1}^N C_i$. Then $C$ is closed, $G$ is open,
$C\subset A\subset G$,
\[
\mu(G\setminus A)<\eta/8,
\qquad
\text{and}
\qquad
\mu(A\setminus C)<5\eta/8.
\]
Thus $\mu(G\setminus C)<\eta$, so $A\in\mathcal R$. Therefore $\mathcal R$ is a
$\sigma$-algebra containing all closed sets, and hence it contains all Borel sets.

Now return to the general case. Write $X=\bigcup_{k=1}^\infty \Omega_k$, where
$\Omega_k$ are open, $\Omega_k\subset \Omega_{k+1}$, and $\mu(\Omega_k)<\infty$.
Applying the finite case in the metric space $\Omega_k$ to $A\cap\Omega_k$, we
find a relatively open set $W_k\subset\Omega_k$ such that
$A\cap\Omega_k\subset W_k$ and
\[
\mu(W_k\setminus A)<\eta 2^{-k}.
\]
Since $\Omega_k$ is open in $X$, each $W_k$ is open in $X$. Thus
$V=\bigcup_k W_k$ is open, $A\subset V$, and
\[
\mu(V\setminus A)\leq \sum_{k=1}^\infty \mu(W_k\setminus A)<\eta.
\]
This proves the regularity fact.

Since $Y$ is separable, choose a countable base $\{U_i\}_{i=1}^\infty$ for its
topology. For each $i$, the set $A_i=\Phi^{-1}(U_i)$ is Borel. By the regularity fact,
there is an open set $V_i\subset X$ such that
\[
A_i\subset V_i,\qquad \mu(V_i\setminus A_i)<\varepsilon 2^{-i-2}.
\]
Set
\[
E=\bigcup_{i=1}^\infty (V_i\setminus A_i).
\]
Then $E$ is Borel and $\mu(E)<\varepsilon/4$. We claim that $\Phi|_{X\setminus E}$
is continuous. Let $g=\Phi|_{X\setminus E}$. For every $i$,
\[
g^{-1}(U_i)=V_i\cap (X\setminus E).
\]
Indeed, if $x\in g^{-1}(U_i)$, then $x\in A_i\subset V_i$ and $x\in X\setminus E$.
Conversely, if $x\in V_i\cap (X\setminus E)$, then $x\notin V_i\setminus A_i$,
so $x\in A_i$, and hence $g(x)\in U_i$.

Now let $U\subset Y$ be open. Since $\{U_i\}$ is a base, $U=\bigcup_j U_{i_j}$.
Therefore
\[
g^{-1}(U)=\bigcup_j g^{-1}(U_{i_j})
=\left(\bigcup_j V_{i_j}\right)\cap (X\setminus E),
\]
which is open in the relative topology of $X\setminus E$. Thus $g$ is continuous.

Finally, apply the regularity fact once more to the Borel set $E$. Choose an open
set $G\subset X$ such that $E\subset G$ and $\mu(G\setminus E)<\varepsilon/2$.
Let $F=X\setminus G$. Then $F$ is closed, $F\subset X\setminus E$, and hence
$\Phi|_F$ is continuous. Also
\[
\mu(X\setminus F)=\mu(G)=\mu(E)+\mu(G\setminus E)
<\varepsilon/4+\varepsilon/2<\varepsilon.
\]
The proof is complete.
\end{proof}

The rest of this section is devoted to covering theorems.

Recall that if $B=B(x,r)$ is a ball in a metric space, then for $\lambda>0$ we define $\lambda B:=B(x,\lambda r)$.

For a proof of the next result, see for example, \cite[Theorem~1.2]{heinonen}.
\begin{theorem}[$5r$-covering lemma]
\label{AT7}
Let $(X,d)$ be a metric space, and let $\mathcal{F}$ be a family of balls (open or closed) in $X$ with uniformly bounded diameters, i.e.,
\[
\sup_{B\in\mathcal{F}} \operatorname{diam}(B)<\infty.
\]
Then there exists a subfamily $\mathcal{G}\subset \mathcal{F}$ of pairwise disjoint balls such that
\[
\bigcup_{B\in\mathcal{F}} B \subset \bigcup_{B\in\mathcal{G}} 5B.
\]
If $X$ is separable, the family $\mathcal{G}$ is countable.
\end{theorem}

\begin{definition}
A Borel measure $\mu$ on a metric space $(X,d)$ is called \emph{doubling} if there exists a constant $C_d\ge 1$ such that $0<\mu(2B)\le C_d\,\mu(B)<\infty$
for every ball $B$ in $X$.
\end{definition}

\begin{definition}
Let $A\subset X$. A family $\mathcal{F}$ of closed balls in $X$ is called a \emph{Vitali covering} of $A$ if for every $x\in A$ and every $\eps>0$ there exists a ball
$B\in\mathcal{F}$ such that $x\in B$ and $\diam(B)<\eps$.
\end{definition}

For a proof of the next result, see for example, \cite[Theorem~1.6]{heinonen}.
\begin{theorem}[Vitali covering theorem for doubling measures]
\label{AT6}
Let $(X,d)$ be a metric space, let $\mu$ be a doubling Borel measure on $X$, let $A\subset X$ be Borel, and let $\mathcal{F}$ be a Vitali covering of $A$ by closed balls. Then there exists a countable  subfamily $\mathcal{G}\subset \mathcal{F}$ of pairwise disjoint balls such that
\[
\mu\Big(A\setminus \bigcup_{B\in\mathcal{G}}B\Big)=0.
\]
\end{theorem}

\subsection{Hausdorff measure}
\label{HAU}
In this section we will collect basic properties of the Hausdorff measure.

\begin{definition}
\label{HMD1}
Let $(X,d)$ be a metric space. We say that a family $\{A_i\}_{i=1}^\infty$ of  subsets of $X$ is a {\em $\delta$-covering} of $A\subset X$ if $A\subset\bigcup_i A_i$ and $\diam A_i\leq \delta$ for all $i$.
Fix $0 \leq s < \infty$. For a subset $A$ of $X$ and a $\delta \in (0,\infty]$, the {\em Hausdorff content} $\H^s_\delta$ is defined by
$$
\H^s_\delta (A)
= \inf \frac{\omega_s}{2^s}\sum_{i=1}^\infty (\diam A_i)^s,
\qquad
\text{where}
\qquad
\omega_s=\frac{\pi^{s/2}}{\Gamma(\tfrac{s}{2}+1)},
$$
and the infimum is taken over all $\delta$-coverings $A\subset\bigcup_{i=1}^\infty A_i$. By convention, when $s=0$ we set $(\diam A_i)^0=1$ if $A_i$ is nonempty and $(\diam \varnothing)^0=0$, so that every nonempty covering set contributes $1$ in dimension zero.

Note that $\omega_0=1$ and for $n\in\bbbn$, $\omega_n$ is the volume of the unit ball in $\bbbr^n$.
The {\em $s$-Hausdorff measure} of $A$ is then defined by
$\H^s(A)=\lim_{\delta\to 0^+} \H^s_\delta(A)$.
\end{definition}

Since the function $\delta\mapsto \H^s_\delta(A)$ is non-increasing, the limit in the definition of the Hausdorff measure $\H^s(A)$ exists and equals $\sup_{\delta>0} \H^s_\delta(A)$. Observe also that we can always assume that the sets $A_i$ are closed, since taking the closure of a set does not increase its diameter. One can also assume that the sets $A_i$ are open -- see Lemma below.

If $s=0$, the Hausdorff measure $\H^0$ is simply the counting measure.

\begin{lemma}\label{lem:hsop}
    Define by $$
\H^s_{\delta,op} (A)
= \inf \frac{\omega_s}{2^s}\sum_{i=1}^\infty (\diam A_i)^s,
\qquad
\text{where}
\qquad
\omega_s=\frac{\pi^{s/2}}{\Gamma(\tfrac{s}{2}+1)},
$$
and the infimum is taken over all $\delta$-coverings $A\subset\bigcup_{i=1}^\infty A_i$ by \emph{open} sets $A_i$. Then $\lim_{\delta\to 0^+} \H^s_{\delta,op}(A)=\H^s(A)$.
\end{lemma}
\begin{proof}
    Obviously, any $\delta$-covering by open sets is a $\delta$-covering, so $\H^s_{\delta,op} (A) \geq \H^s_{\delta} (A)$.

    To prove the lemma it suffices now to show that for every $0<\varepsilon<\delta$,
    \begin{equation}\label{eq:lem:hsop}
        \mathcal H^{s}_{\delta,op}(A) \leq \mathcal H^s_{\delta-\varepsilon}(A),
    \end{equation}

    For any $\eta\in(0,\delta)$ let $\{A_i\}$ be a $(\delta-\eps)$-covering of $A$ such that
    $$
    \frac{\omega_s}{2^s} \sum_{i=1}^\infty(\diam A_i)^s\leq \H^s_{\delta-\eps}(A)+\eta.
    $$

    Let $U_i=\{x\in X~:~d(x,A_i)<r_i\}$ for $i=1,2,\ldots$, with $r_i\in(0,\frac{\eps}{2})$ chosen sufficiently small to have $(\diam U_i)^s\leq (\diam A_i)^s+2^{s-i}\eta (\omega_s)^{-1}$. Then $U_i$ are open and  ${\diam U_i\leq \diam A_i+2r_i\leq \delta}$, so $\{U_i\}$ form an open $\delta$-cover of $A$. Also,
    $$
\H^s_{\delta,op}(A)\leq \frac{\omega_s}{2^s}\sum_{i=1}^\infty (\diam U_i)^s\leq \frac{\omega_s}{2^s}\sum_{i=1}^\infty (\diam A_i)^s+\eta\leq \H^s_{\delta-\eps}(A)+2\eta.
    $$
Since the above holds for all $\eta\in(0,\delta)$, we can take $\eta\to 0^+$, which proves \eqref{eq:lem:hsop} and concludes the proof of the lemma.
\end{proof}

\begin{proposition}
\label{AT2}
The Hausdorff measure is an outer measure defined on all subsets of $X$ and all
Borel sets are $\H^s$-measurable. Moreover, $\H^s$ is Borel regular meaning that for any $E\subset X$, there is a Borel set $F\subset X$ such that $E\subset F$ and $\H^s(E)=\H^s(F)$.
\end{proposition}
For a proof see e.g. \cite[Lemma~2.10]{EH1}.

Recall that the Lebesgue measure of a set $A$ is denoted by $\LL^n(A)$ or $|A|$.

The classical isodiametric inequality states that among all compact subsets $A\subset \bbbr^n$ with a given diameter, the ball has the largest volume. More precisely, the Lebesgue measure satisfies
$$
|A|\leq \omega_n\Big(\frac{\diam A}{2}\Big)^n=|B^n(0,\tfrac{1}{2}\diam A)|.
$$
This inequality plays a crucial role in the proof that the Hausdorff measure $\H^n$ coincides with the Lebesgue measure on $\bbbr^n$, see Corollary~\ref{SJT2}.

In fact, we will prove a more general isodiametric inequality in any finite dimensional normed space.

If $\sigma$ is a norm on $\bbbr^n$, we write $\bbbr^n_\sigma$ for $\bbbr^n$ equipped with the norm metric
$d(x,y)=\sigma(x-y)$.
Similarly, $B^n_\sigma(x,r)$ and $\bar{B}^n_\sigma(x,r)$ denote the open and closed balls in $\bbbr^n_\sigma$, respectively. For any set $A\subset \bbbr^n_\sigma$, we define
$\diam_\sigma(A)=\sup\{\sigma(x-y):x,y\in A\}$.
\begin{theorem}[Isodiametric inequality]
\label{AT3}
For any compact set $A\subset\bbbr^n_\sigma$
\begin{equation}
\label{eq:isod}
|A|\leq \big|\bar{B}^n_\sigma(0,\tfrac{1}{2}\diam_\sigma A)\big|.
\end{equation}
\end{theorem}
The proof presented below is standard and almost identical with that in the case of the Euclidean norm, cf.~\cite[Theorem 11.2.1]{buragoz}.

We will need the classical Brunn-Minkowski inequality, see for example, \cite[Theorem~8.3.1]{buragoz} or \cite[Theorem~7.1.1 and p.\ 371]{schneider}.
\begin{theorem}[Brunn-Minkowski inequality]
\label{AT4}
For any nonempty compact sets $A,B\subset\bbbr^n$ we have
$$
|A+B|^{1/n}\geq |A|^{1/n}+|B|^{1/n},
\quad
\text{where}
\quad
A+B=\{x+y:\, x\in A,\ y\in B\}.
$$
\end{theorem}
\begin{proof}[Proof of Theorem~\ref{AT3}]
Let
\[
A'=-A
\qquad
\text{and}
\qquad
F=\tfrac{1}{2}(A+A').
\]
Then $F$ is compact and centrally symmetric with respect to $0$. We claim that
\begin{equation}
\label{SJeq2}
|F|\ge  |A|
\qquad
\text{and}
\qquad
\diam_\sigma F\le \diam_\sigma A.
\end{equation}
The Brunn--Minkowski inequality gives
\[
|F|=
\big|\tfrac{1}{2}(A+A')\big|
\ge
\big(\,|\tfrac{1}{2}A|^{1/n}
+
|\tfrac{1}{2}A'|^{1/n}\, \big)^n
=|A|.
\]

Now let $x,y\in F$. Then there are points $x',y'\in A$ and $x'',y''\in A'$ such that
$x=\tfrac{1}{2}(x'+x'')$ and $y=\tfrac{1}{2} (y'+y'')$. Hence
\begin{align*}
\sigma(x-y)
=\tfrac{1}{2}\,\sigma(x'-y'+x''-y'')
\le \tfrac{1}{2}\bigl(\sigma(x'-y')+\sigma(x''-y'')\bigr)
\le  \diam_\sigma A,
\end{align*}
because $\diam_\sigma A'=\diam_\sigma A$. Thus
$\diam_\sigma F\le \diam_\sigma A$. The proof of \eqref{SJeq2} is complete.

Since $F$ is centrally symmetric, if $x\in F$ then also $-x\in F$, so
\[
2\sigma(x)=\sigma(x-(-x))\le \diam_\sigma F\le \diam_\sigma A.
\]
Therefore,
$F\subset \bar{B}_\sigma^n\big(0,\tfrac{1}{2}\diam_\sigma A\big)$ which together with $|A|\leq |F|$ proves the theorem.
\end{proof}

Since the proof of the isodiametric inequality is Euclidean in its nature, the next result is rather surprising.
It seems to follow from Federer’s general density estimates for Hausdorff measures; see Theorems 2.10.17 and 2.10.18 in \cite{federer}. However, Federer’s arguments are not easy to understand. For an elementary and self-contained proof, see \cite[Lemma~2.14]{EH1}.
\begin{theorem}[Asymptotic isodiametric inequality]
\label{AT1}
Suppose $\H^s(X)<\infty$, for some real $s\geq 0$. Fix $\eps>0$.
Then there is a set $Z\subset X$ of measure zero, $\H^s(Z)=0$, such that for every $x\in X\setminus Z$ there is $\delta_x>0$ with the property that for every set $E\subset X$, measurable or not, the following holds:
$$
x\in E\subset \bar{B}(x,\delta_x)
\quad
\Longrightarrow
\quad
\H^s(E)\leq (1+\eps)\omega_s\Big(\frac{\diam E}{2}\Big)^s.
$$
\end{theorem}

It is easy to see that $\H^n([0,1]^n)<\infty$. Indeed, it suffices to consider coverings of $[0,1]^n$ by $2^{nk}$ cubes of edge length $2^{-k}$. Since $\H^n$ is translation invariant, it follows that there is $c\geq 0$ such that $\H^n(E)=c\LL^n(E)$ for any Borel set $E$. Since both measures $\H^n$ and $\LL^n$ are Borel regular, it follows that $\H^n(E)=c\LL^n(E)$ for all Lebesgue measurable sets $E$. We will prove that in fact $c=1$ i.e., $\H^n(E)=\LL^n(E)$ for all Lebesgue measurable sets $E$, see Corollary~\ref{SJT2}. To this end it suffices to show that $\H^n(B^n(0,1))=\omega_n$. This is a consequence of a more general result, Proposition~\ref{AT5}.

We write $\H^s_\sigma$ for the Hausdorff measure on the metric space $\bbbr^n_\sigma$, while $\H^s$ denotes the Hausdorff measure with respect to the Euclidean metric. Similarly, $\H^s_{\sigma,\delta}$ denotes the $\delta$-Hausdorff content with respect to $\sigma$.

By the same argument as the one described above, if we can show that $0<\H^n_\sigma([0,1]^n)<\infty$, then $\H^n_\sigma=c\LL^n$ for some $c\in (0,\infty)$.

\begin{proposition}
\label{AT5}
For any norm $\sigma$ on $\bbbr^n$ we have $\H^n_\sigma(\bar{B}^n_\sigma(0,1))=\omega_n$.
\end{proposition}
\begin{corollary}
\label{SJT2}
$\LL^n=\H^n$ on $\bbbr^n$ on Lebesgue measurable sets.
\end{corollary}
\begin{corollary}
\label{AT12}
If $\sigma$ is a norm on $\bbbr^n$, then there is $c(\sigma)\in (0,\infty)$ such that $\H^n_\sigma(A)=c(\sigma)\H^n(A)$ for all sets $A\subset\bbbr^n$.
\end{corollary}
Indeed, $\H^n_\sigma(A)=c\LL^n(A)=c\H^n(A)$ on Lebesgue measurable sets and since both measures $\H^n_\sigma$ and $\H^n$ are Borel regular, $\H^n_\sigma=c\H^n$ on all sets, see also Proposition~\ref{SJT1}.
\begin{proof}[Proof of Proposition~\ref{AT5}]
Since any two norms in $\bbbr^n$ are equivalent, $\operatorname{id}:\bbbr^n\to\bbbr^n_\sigma$ is bi-Lipschitz, and hence the Hausdorff measures $\H^n$ and $\H^n_\sigma$ are comparable. Since $\H^n_\sigma([0,1]^n)\approx\H^n([0,1]^n)<\infty$, $\H^n_\sigma$ is finite on compact sets and hence
$$
\eta_n:=\H^n_\sigma(\bar{B}^n_\sigma(0,1))<\infty.
$$
Note that $\H^n_\sigma(\bar{B}^n_\sigma(x,r))=\H^n_\sigma(B^n_\sigma(x,r))=\eta_n r^n$ for any $x\in\bbbr^n$ and $r>0$.
Our aim is to prove that $\eta_n=\omega_n$.

First, let us show that $\eta_n\leq\omega_n$.

Assume otherwise, that for some $\beta>0$ we have $\eta_n\geq(1+\beta)\omega_n$. Fix $\delta, \eps>0$.  Since the measure $\H^n_\sigma$ is doubling, by Vitali's Covering Theorem~\ref{AT6} we can find pairwise disjoint balls $\bar{B}^n_\sigma(x_i,r_i)$, with $r_i<\delta$, that fill $B^n_\sigma(0,1)$ up to a set of measure zero, i.e.,
$$
B^n_\sigma(0,1)=Z\cup\bigcup_i \bar{B}^n_\sigma(x_i,r_i), \qquad \H^n_\sigma(Z)=0.
$$

Since $\H^n_\sigma$ is invariant under translations, and finite on compact sets, $\H^n_\sigma=c\LL^n$ for some $c\geq 0$ (on Lebesgue measurable sets). Since $\eta_n\geq (1+\beta)\omega_n>0$, we have that $c>0$. Note also that $B^n(x,r)\subset B_\sigma^n(x,c'r)$ for some $c'>0$ depending on $\sigma$ only.
Therefore, $|Z|=c^{-1}\H^n_\sigma(Z)=0$, and hence
the null set $Z$ can be covered by a countable family of balls $B^n_\sigma(y_j,s_j)$ such that $s_j<\delta$ and $\sum_j\omega_n s_j^n<\eps$. Then the balls $\bar{B}^n_\sigma(x_i,r_i)$ and $B^n_\sigma(y_j,s_j)$, $i,j=1,\ldots$, form a $2\delta$-covering of $B^n_\sigma(0,1)$ and thus
\begin{equation*}
\begin{split}
\eta_n
&\geq
\sum_i \H^n_\sigma( \bar{B}^n_\sigma(x_i,r_i))
=\sum_i \eta_n r_i^n\\
&\geq(1+\beta)\Big(\sum_i \omega_n r_i^n+\sum_j \omega_n s_j^n\Big)-(1+\beta)\sum_j \omega_n s_j^n\\
&\geq (1+\beta)\H^n_{\sigma,2\delta}(\bar{B}^n_\sigma(0,1))-(1+\beta)\eps.
\end{split}
\end{equation*}
The above holds for all $\delta>0$; taking $\delta\to 0$ we get
$
\eta_n\geq (1+\beta)\eta_n-(1+\beta)\eps.
$
Next, taking $\eps\to 0$ we get $\eta_n\geq (1+\beta)\eta_n$, which is a contradiction, since $\eta_n$ is positive and finite.

It remains to show that $\eta_n\geq\omega_n$.

Since $\H^n_\sigma=c\LL^n$, \eqref{eq:isod} yields
$$
\H^n_\sigma(A)\leq\H^n_\sigma\big(\bar{B}_\sigma^n(0,\tfrac{1}{2}\diam_\sigma A)\big)
\qquad
\text{for any bounded set } A\subset\bbbr^n.
$$
Indeed, if $A$ is compact, it is \eqref{eq:isod}. If $A$ is not compact, we replace it by its closure and observe that taking the closure does not increase the diameter of the set.

This implies that whenever $\{A_i\}_i$ is a family of bounded sets covering $\bar{B}^n_\sigma(0,1)$, we have
\begin{equation*}
\begin{split}
0<
\H^n(\bar{B}^n_\sigma(0,1))&\leq \sum_i \H^n(A_i)\leq
\sum_i\H^n(\bar{B}^n_\sigma(0,\tfrac{1}{2}\diam_\sigma A_i))
\\
&=\H^n(\bar{B}^n_\sigma(0,1)) \sum_i(\tfrac{1}{2}\diam_\sigma A_i)^n,
\end{split}
\end{equation*}
so $2^{-n}\sum_i(\diam_\sigma A_i)^n\geq 1$, hence $2^{-n}\omega_n\sum_i(\diam_\sigma A_i)^n\geq \omega_n$,
and upon taking the infimum over all $\delta$-coverings of $\bar{B}^n_\sigma(0,1)$ we get
$$
\H^n_{\sigma,\delta}(\bar{B}^n_\sigma(0,1)) \geq\omega_n,
$$
thus
$\eta_n=\sup_{\delta>0} \H^n_{\sigma,\delta}(\bar{B}^n_\sigma(0,1))\geq \omega_n$, which concludes the proof.
\end{proof}
We conclude with a standard observation concerning Borel representatives of measurable functions.
\begin{lemma}
\label{app:measborel}
Assume $X$ is a $\H^n$-$\sigma$-finite metric space. Then every $\H^n$ measurable function $u:X\to[0,\infty]$ has a Borel measurable representative $v$, i.e., there exists a Borel function $v:X\to [0,\infty]$ such that $u=v$ $\H^n$-a.e.
\end{lemma}
\begin{proof}
By $\sigma$-finiteness of $X$ and Borel regularity of $\H^n$, every
$\H^n$-measurable set differs from a Borel set by an $\H^n$-null set.

For every $q\in\qp$, let $E_q=\{x\in X:u(x)>q\}$.
Choose a Borel set $B_q\subset X$ such that $\H^n(E_q\triangle B_q)=0$.
Set
$$
N=\bigcup_{q\in\qp}(E_q\triangle B_q),
$$
so that $\H^n(N)=0$, and define
$$
v(x)=\sup\{q\in\qp~:~x\in B_q\},
$$
(by convention,  $\sup\varnothing=0$). The function $v$ is Borel measurable, since for every $t\geq 0$,
$$
\{v>t\}=\bigcup_{q\in\qp,\,q>t}B_q.
$$
If $x\not\in N$, then for every $q\in\qp$,
$$
x\in B_q\quad\Longleftrightarrow\quad u(x)>q,
$$
and hence $v(x)=u(x)$. Thus $u=v$ $\H^n$-a.e.
\end{proof}

\end{document}